\pdfoutput=1
\documentclass[a4paper,11pt]{amsart}
\usepackage[hmarginratio={1:1},vmarginratio={1:1},lmargin=60.0pt,tmargin=60.0pt]{geometry}

\allowdisplaybreaks

\usepackage[numbers]{natbib}
\usepackage[utf8]{inputenc}
\usepackage{latexsym,exscale,mathtools}
\usepackage{amssymb,amsmath,amsthm,amsfonts,enumitem}
\usepackage[table]{xcolor}
\usepackage{booktabs,array,tabularx,graphicx}
\usepackage{etoolbox,needspace}
\usepackage{placeins}

\setlist[enumerate]{itemsep=0.15cm,label=\emph{\upshape(\alph*)}}
\setlist[enumerate,2]{itemsep=0.15cm,label=\emph{\upshape(\roman*)}}

\usepackage{tikz}
\usetikzlibrary{arrows.meta}
\definecolor{orchid}{RGB}{143,40,194}

\definecolor{lava}{RGB}{207,16,32}
\definecolor{mydarkblue}{RGB}{10,10,170}
\definecolor{treegray}{gray}{0.88}
\definecolor{treegraylight}{gray}{0.96}
\IfFormatAtLeastTF{2026-06-01}{}{\usepackage{aliascnt}}

\definecolor{tomato}{RGB}{255,99,71}

\usepackage{todonotes}

\let\emph\relax
\DeclareTextFontCommand{\emph}{\bfseries\em}
\tikzset{
anchorbase/.style={baseline={([yshift=#1]current bounding box.center)}},
anchorbase/.default={-0.5ex},
every picture/.append style={line cap=round,line join=round,anchorbase},
a/.style={draw=black,line width=1pt},
web/.style={line width=.7pt},
sink/.style={circle,draw=black,fill=black,inner sep=1.6pt},
source/.style={circle,draw=black,fill=white,inner sep=1.6pt},
coupon/.style={draw=black,fill=treegray!45,inner sep=2.6pt,
rounded corners=1pt,font=\scriptsize},
arr/.style={-{Stealth[length=1.7mm]},line width=.8pt}
}

\newcommand{\Z}{\mathbb Z}
\newcommand{\F}{\mathbb K}
\newcommand{\one}{\mathbf1}
\newcommand{\Inv}{\operatorname{Inv}}

\newcommand{\inv}{\operatorname{inv}}
\newcommand{\word}[1]{#1}
\newcommand{\wordstart}{\mathtt{S}}
\newcommand{\wordmiddle}{\mathtt{M}}
\newcommand{\wordend}{\mathtt{E}}
\newcommand{\startmap}{{\Psi}_{{\wordstart}}}
\newcommand{\middlemap}{{\Psi}_{{\wordmiddle}}}
\newcommand{\finishmap}{{\Psi}_{{\wordend}}}
\newcommand{\createmap}{{\Psi}_+}
\newcommand{\removemap}{{\Psi}_-}
\newcommand{\pp}{\mathsf p}
\newcommand{\tailmap}{\boldsymbol{\rho}}

\newcommand{\repo}[1]{\nolinkurl{#1}}
\newcommand{\qpar}{q}
\newcommand{\tpar}{v}
\newcommand{\vpar}{{\tpar^{-1}}}

\def\NewTheorem#1#2{%
\IfFormatAtLeastTF{2026-06-01}{%
  \newtheorem{#1}[equation]{#2}%
}{%
  \newaliascnt{#1}{equation}%
  \newtheorem{#1}[#1]{#2}%
  \aliascntresetthe{#1}%
}%
\expandafter\def\csname #1autorefname\endcsname{#2}%
}

\numberwithin{equation}{subsection}

\theoremstyle{plain}
\NewTheorem{proposition}{Proposition}
\NewTheorem{theorem}{Theorem}
\NewTheorem{corollary}{Corollary}
\NewTheorem{lemma}{Lemma}
\theoremstyle{definition}
\NewTheorem{definition}{Definition}
\AtEndEnvironment{definition}{\null\hfill$\Diamond$}%
\NewTheorem{example}{Example}
\AtEndEnvironment{example}{\null\hfill$\Diamond$}%
\theoremstyle{remark}
\NewTheorem{remark}{Remark}
\newcommand{\remarkend}{\null\hfill$\Diamond$}
\AtEndEnvironment{remark}{\remarkend}%
\NewTheorem{question}{Question}
\AtEndEnvironment{question}{\null\hfill$\Diamond$}%

\usepackage[hypertexnames=false]{hyperref}
\usepackage{bookmark}
\hypersetup{
pdftoolbar=true,
pdfmenubar=true,
pdffitwindow=false,
pdfstartview={FitH},
pdftitle={Exponentially large mu-coefficients},
pdfauthor={Tom Goertzen, Abel Lacabanne, Daniel Tubbenhauer and Pedro Vaz},
pdfsubject={},
pdfcreator={Tom Goertzen, Abel Lacabanne, Daniel Tubbenhauer and Pedro Vaz},
pdfproducer={Tom Goertzen, Abel Lacabanne, Daniel Tubbenhauer and Pedro Vaz},
pdfkeywords={},
pdfnewwindow=true,
colorlinks=true,
linkcolor=mydarkblue,
citecolor=teal,
filecolor=magenta,
urlcolor=orchid,
linkbordercolor=lava,
citebordercolor=teal,
urlbordercolor=orchid,
linktocpage=true
}
\expandafter\def\csname subjclassname@2020\endcsname{2020 Mathematics Subject Classification}

\BeforeBeginEnvironment{definition}{\begin{samepage}}
\AfterEndEnvironment{definition}{\end{samepage}}
\BeforeBeginEnvironment{lemma}{\begin{samepage}}
\AfterEndEnvironment{lemma}{\end{samepage}}
\BeforeBeginEnvironment{proposition}{\begin{samepage}}
\AfterEndEnvironment{proposition}{\end{samepage}}
\BeforeBeginEnvironment{theorem}{\begin{samepage}}
\AfterEndEnvironment{theorem}{\end{samepage}}

\def\makeautorefname#1#2{\csdef{#1autorefname}{#2}}
\makeautorefname{section}{Section}%
\makeautorefname{subsection}{Section}%
\makeautorefname{subsubsection}{Section}%

\makeautorefname{theorem}{Theorem}%
\makeautorefname{lemma}{Lemma}%
\makeautorefname{proposition}{Proposition}%
\makeautorefname{corollary}{Corollary}%
\makeautorefname{example}{Example}%
\makeautorefname{remark}{Remark}%
\makeautorefname{definition}{Definition}%

\begin{document}
\title[Exponentially large mu-coefficients]{Exponentially large mu-coefficients}

\author{Tom Goertzen}
\address{T.G.: The University of Sydney,
School of Mathematics and Statistics F07,
NSW 2006, Australia,
\href{https://profiles.sydney.edu.au/tom.goertzen}
{profiles.sydney.edu.au/tom.goertzen},
\href{https://orcid.org/0009-0003-3399-3416}
{ORCID 0009-0003-3399-3416}}
\email{tom.goertzen@sydney.edu.au}

\author{Abel Lacabanne}
\address{A.L.: Laboratoire de Math\'ematiques Blaise Pascal (UMR 6620),
Universit\'e Clermont Auvergne,
Campus Universitaire des C\'ezeaux, 3 place Vasarely,
63178 Aubi\`ere Cedex, France,
\href{https://www.normalesup.org/~lacabanne/}
{www.normalesup.org/\textasciitilde lacabanne},
\href{https://orcid.org/0000-0001-8691-3270}
{ORCID 0000-0001-8691-3270}}
\email{abel.lacabanne@uca.fr}

\author{Daniel Tubbenhauer}
\address{D.T.: The University of Sydney,
School of Mathematics and Statistics F07,
Office Carslaw 827, NSW 2006, Australia,
\href{http://www.dtubbenhauer.com}
{www.dtubbenhauer.com},
\href{https://orcid.org/0000-0001-7265-5047}
{ORCID 0000-0001-7265-5047}}
\email{daniel.tubbenhauer@sydney.edu.au}

\author{Pedro Vaz}
\address{P.V.: Institut de Recherche en Math\'ematique et Physique,
Universit\'e catholique de Louvain,
Chemin du Cyclotron 2, 1348 Louvain-la-Neuve, Belgium,
\href{https://perso.uclouvain.be/pedro.vaz/}
{perso.uclouvain.be/pedro.vaz},
\href{https://orcid.org/0000-0001-9422-4707}
{ORCID 0000-0001-9422-4707}}
\email{pedro.vaz@uclouvain.be}
\date{}

\begin{abstract}
Using webs, we prove that Kazhdan--Lusztig mu-coefficients for symmetric groups grow at least exponentially with the rank, already for pairs labeled by a three-row rectangle.
\end{abstract}

\subjclass[2020]{Primary: 05E10, 20C08; Secondary: 17B37, 18M30.}
\keywords{Kazhdan--Lusztig polynomials, $\mu$-coefficients, three-row partitions,
$\mathfrak{sl}_3$ webs, dual canonical bases, Hecke algebras.}

\addtocontents{toc}{\protect\setcounter{tocdepth}{1}}
\maketitle
\tableofcontents

\section{Introduction}

We prove that $\mu$-coefficients for
$S_n=\operatorname{Aut}(\{1,\ldots,n\})$ grow at least
exponentially in $n$. By standard parabolic embeddings, this settles the boundedness question in all classical finite types and gives the same growth in their affine counterparts.

\subsection{The question}

Kazhdan--Lusztig (KL) polynomials \cite{KL} sit at the meeting point
of the representation theory of Hecke algebras, combinatorics and
the geometry of Schubert varieties. Among their coefficients,
the $\mu$-coefficients play a distinguished role. They are the
coefficients in the highest degree permitted by the KL degree
bound, taken to be zero when this bound is not attained.

\begin{remark}
There are plenty of applications, and we can only give a biased list.
For a fixed Coxeter system, the full collection of $\mu$-coefficients determines all KL polynomials recursively,
and the $\mu$-coefficients determine the action of simple reflections on the KL basis and 
give the edge weights in the KL cell
$W$-graphs \cite{KL}. Together with descent data, these weights
determine the cell representations, which in type $A$ give all
simple complex representations of symmetric groups.
For finite Weyl groups, they also determine composition
multiplicities of Verma modules in regular integral blocks of
category $\mathcal O$, via the KL character formula \cite{KL}.
For Weyl groups, $\mu(x,y)$ gives a local intersection cohomology
dimension of a Schubert variety in degree $\ell(y)-\ell(x)-1$
\cite{KLPositivity}.
The combinatorial invariance conjecture, due independently to
Lusztig and Dyer \cite{Dyer}, predicts, in particular, that
$\mu(x,y)$ is an invariant of the abstract Bruhat interval $[x,y]$.
(A proof with Lean verification was recently claimed in
\cite{OpenAICIC}.)
These examples are just the tip of the iceberg.
\end{remark}

The $\mu$-coefficients are nonnegative integers \cite{KLPositivity}.
For a long time, they were conjectured to take only the values
$0$ and $1$, see e.g. \cite{LS,GarsiaMcLarnan}. McLarnan and Warrington disproved this conjecture
in 2002 \cite{MW}, and subsequent computations produced
considerably larger examples \cite{Warrington}.
This left a natural question that remained open for more
than two decades: Are the $\mu$-coefficients bounded? If not, how quickly can they grow?

\begin{remark}
The question is different from asking whether the coefficients
of KL polynomials can be arbitrarily large. Indeed, Polo
showed that every polynomial with nonnegative integer
coefficients and constant term $1$ occurs as a KL polynomial
of some symmetric group \cite{Polo99}. However, the
$\mu$-coefficient must occur in the highest degree permitted
by the KL degree bound, and Polo's construction does not
ensure this.
\end{remark}

Scott--Xi constructed examples with $\mu=n+2$ in affine type
$\widetilde A_n$ for $n\ge4$ \cite{SX}, proving
unboundedness as the rank varies.
L\"ubeck's computations \cite{Luebeck} suggested exponential growth
with the rank for first cohomology dimensions; see also
\cite{ParkerStewart} for upper bounds.
For finite symmetric groups, even unboundedness remained open.
Our explicit family settles this question with a much stronger
conclusion (cf. \autoref{thm:rankbound}):

\[
\setlength{\fboxsep}{2pt}
\colorbox{blue!6}{
\begin{minipage}{0.85\textwidth}
\centering
\emph{The $\mu$-coefficients of symmetric groups grow at least
exponentially in the rank, already within left cells of
three-row rectangular shape.}
\end{minipage}}
\]

This is a dramatic jump in complexity compared with many known
results and examples.
Several natural classes of pairs in finite and affine types
satisfy $0$-$1$ bounds \cite{Green,Jones,Xi}, while the explicit
unbounded affine families of Scott--Xi and Chen \cite{SX,Chen}
exhibit only linear growth with the rank.
That exponential growth already occurs for three-row rectangles
makes the contrast even sharper: within two-row partitions, all
$\mu$-coefficients lie in $\{0,1\}$, see e.g.
\cite{LS}.

Our contribution is also methodological: we introduce a new
approach to $\mu$-coefficients using webs, in which the
construction of these striking examples becomes ``obvious''
despite their complicated description as permutations.
(Proving that the construction works remains very difficult.)

\begin{remark}
We will pass freely between
$\mu$-coefficients in symmetric groups and those occurring inside a
single left cell. More precisely, Lanini--McNamara prove that every $\mu$-coefficient
arising in a symmetric group also occurs for a pair of elements in a
common left cell of a possibly larger symmetric group
\cite[Corollary 4.3]{LM}, but
exponential growth inside left cells immediately implies exponential
growth for symmetric groups in general, cf. \autoref{thm:rankbound}.
\end{remark}

Finally, at the end of \autoref{S:Wrap}, we sketch some 
interesting consequences of our main result.

\subsection{The main result in more detail}

Our family can be written down using two words. A row word records a
standard tableau by assigning to entry $i$ the letter $A$, $B$ or $C$
according to whether $i$ lies in the first, second or third row. For example, the $(2,2,2)$ tableau below gives the row word and permutation
\[
\begin{array}{cc}
1&2\\
3&5\\
4&6
\end{array}
\quad\leftrightsquigarrow\quad
AABCBC
\quad\leftrightsquigarrow\quad
(4,6,3,5,1,2)
\quad\leftrightsquigarrow\quad
\begin{tikzpicture}[
  x=.42cm,y=.75cm,font=\scriptsize,
  baseline={([yshift=-.5ex]current bounding box.center)}
]
\foreach \i/\j in {1/4,2/6,3/3,4/5,5/1,6/2}
  \draw[line width=.55pt] (\i,0)--(\j,2);
\foreach \i in {1,...,6}{
  \node[below=2pt] at (\i,0) {$\i$};
  \node[above=2pt] at (\i,2) {$\i$};
}
\end{tikzpicture}
.
\]
We read rows from bottom to top, left to right within each row, to get
the permutation in one-line notation. Strands connect bottom position
$i$ to the top position given by the $i$th entry.

\begin{theorem}\label{thm:main}
For $r\ge0$, let $S_r,T_r$ be the standard tableaux with row words
\[
(ABC)^{4+6r},\qquad
\underbrace{ABAAAC}_{\wordstart}(\underbrace{BBBAACACC}_{\wordmiddle})^{2r}\underbrace{BBCBCC}_{\wordend}.
\]
For the associated permutations $x_{2r}<y_{2r}$ in $S_{12+18r}$, cf. \autoref{fig:permutation-pair}, one has
\[
\mu(x_{2r},y_{2r})=2^r.
\]
They lie in a common left cell of shape $(4+6r,4+6r,4+6r)$.
\end{theorem}

\begin{figure}[ht]
\centering
\includegraphics[width=\linewidth]{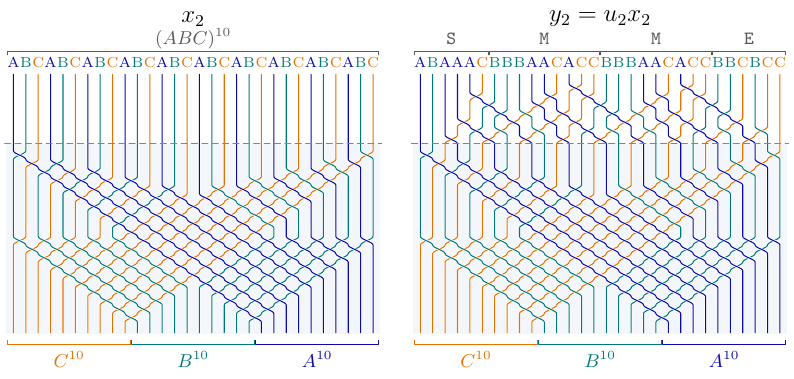}
\caption{The permutations $x_2$ and $y_2$, corresponding to $r=1$.
Strands are colored by tableau row. Reading from bottom to top,
the strand starting at position $i$ ends at the permutation value
at $i$.}
\label{fig:permutation-pair}
\end{figure}

For example, the first four values are
\[
\begin{array}{c|cccc}
r&0&1&2&3\\\hline
\text{symmetric group}&S_{12}&S_{30}&S_{48}&S_{66}\\
\text{cell shape}&(4,4,4)&(10,10,10)&(16,16,16)&(22,22,22)\\
\mu(x_{2r},y_{2r})&1&2&4&8
\end{array}
\]
(The length gap is
$\delta_r=\ell(y_{2r})-\ell(x_{2r})=34r+11$, so the coefficient
occurs at degree $17r+5$. Since
$\mu(x_{2r},y_{2r})=2^{(\delta_r-11)/34}$, the growth is
exponential in the length gap itself.) Thus the permutations, their cell and the degree of the
coefficient are all explicit.
Below, $m$ counts the nine-letter blocks; the theorem uses $m=2r$.
For odd $m$, the corresponding $\mu$-coefficient is zero.

\subsection{The web picture}

An $\mathfrak{sl}_3$ web is a planar trivalent graph representing an
invariant tensor of $\mathfrak{sl}_3$, cf. \cite{Kuperberg}. For three-row rectangles, standard tableaux index
both the web basis and the dual canonical basis of the same invariant
space. Frenkel--Khovanov--Kirillov (FKK) identify tensor coordinates of dual
canonical vectors with normalized parabolic KL polynomials \cite{FKK}.
This gives a way to look for large KL coefficients using pictures.

Khovanov--Kuperberg (KK) showed that a basis web need not be dual canonical
\cite{KK}. Their counterexample contains a hexagon surrounded by six
hexagons, the ``hexagon-in-hexagons'' picture. Our family repeats this
pattern along a strip. Heuristically, cf. \autoref{fig:intro-hexagons}, the construction is:
\[
\setlength{\fboxsep}{2pt}
\colorbox{blue!6}{
\begin{minipage}{0.85\textwidth}
\centering
\emph{Put hexagon-in-hexagons patches next to one another,
join them along three strands, and attach fixed start and end pieces.}
\end{minipage}}
\]

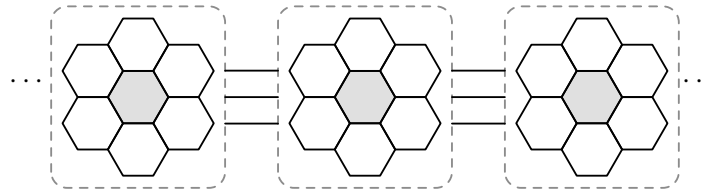
\begin{figure}[!ht]
\centering
\begin{tikzpicture}[x=1cm,y=1cm,line width=.7pt]
\foreach \k in {0,1,2}{
\begin{scope}[shift={({3*\k},0)}]
\draw[black!45,dashed,rounded corners=6pt]
(-1.15,-1.2) rectangle (1.15,1.2);
\fill[black!12]
(0:.4)--(60:.4)--(120:.4)--(180:.4)--(240:.4)--(300:.4)--cycle;
\draw
(0:.4)--(60:.4)--(120:.4)--(180:.4)--(240:.4)--(300:.4)--cycle;
\foreach \a in {30,90,150,210,270,330}{
\begin{scope}[shift={({.4*sqrt(3)*cos(\a)},
{.4*sqrt(3)*sin(\a)})}]
\draw
(0:.4)--(60:.4)--(120:.4)--(180:.4)--(240:.4)--(300:.4)--cycle;
\end{scope}
}
\end{scope}
}
\foreach \k in {0,1}{
\foreach \y in {-.35,0,.35}{
\draw ({3*\k+1.15},\y)--({3*(\k+1)-1.15},\y);
}
}
\node at (-1.45,.2) {$\cdots$};
\node at (7.45,.2) {$\cdots$};
\end{tikzpicture}
\caption{A schematic of the repeated hexagon-in-hexagons pattern and
the three-strand gluing. The central hexagons are shaded. Orientations,
external legs and the precise attachments are omitted; the actual
pieces are given in \autoref{sec:construction}.}
\label{fig:intro-hexagons}
\end{figure}

The exact construction is slightly different; see \autoref{fig:family24} for examples.

\subsection{The proof in a nutshell}

Repeating the picture gives a web, but we need the corresponding
dual canonical vector before applying the KL interpretation. Write
$W_w$ for the web with tableau label $w$, and $b_w$ for the dual
canonical vector with the same label. The change of basis has the form
\[
b_w=W_w+\sum_{z\prec w}\mathcal K_{z,w}(\tpar)W_z,\qquad \mathcal K_{z,w}(\tpar)\in\Z[\tpar,\tpar^{-1}],
\]
where the sum runs over lower tableau labels in the triangular order.
This unitriangular change of basis is described in
\cite[Corollary 3.7]{TubbenhauerCrystal}, using the intermediate
crystal basis construction of Leclerc--Toffin \cite{LeclercToffin}.
The correction removes constant and negative powers of $\tpar$ from the
nonleading tensor coordinates while preserving bar invariance.

For our family, the FKK normalization turns the highest allowed KL
degree into the power $\tpar$. Thus $\mu$ is the coefficient of $\tpar$ in the
relevant tensor coordinate; see \autoref{thm:FKK}.

This correction matters already for the first example with $\mu=2$.
The relevant coordinate of the raw web begins
$\tpar^{-1}+18\tpar+125\tpar^3+\cdots$, whereas the coordinate of its dual
canonical vector begins $2\tpar+70\tpar^3+\cdots$. The lower terms remove
$\tpar^{-1}+16\tpar+\cdots$, so the required coefficient of $\tpar$ is $18-16=2$.
\autoref{sec:basischange} gives the change of basis, with pictures
of the lower webs and a smaller four-dimensional example.

For the whole family, we incorporate the correction into the repeated
block. The resulting map uses the fixed interface
$V^{\otimes3}\oplus\one$, of dimension $28$, where $V$ is the
three-dimensional fundamental representation of $\mathfrak{sl}_3$ and $\one$ is the
trivial representation. The number of external strands grows, but
the interface stays the same.

The proof then has two parts. First, local web identities and degree bounds show that the
corrected tensor is dual canonical for any number of blocks, see
\autoref{prop:canonical}.
Second, a two-state calculation gives the required coefficient, see
\autoref{loc:coefficient}. Adding two blocks doubles its first-order
coefficient. The FKK identification and the ordinary KL degree bound
then give $\mu(x_{2r},y_{2r})=2^r$.

\begin{remark}
Web algebras provide a categorical explanation of the change of
basis: webs define projective modules, while suitably normalized
indecomposable projectives give dual canonical vectors
\cite{MackaayDual,MPT,RobertFamily,RobertCharacterization,TubbenhauerGln}. Our proof uses
the tensor characterization of the dual canonical basis.
\end{remark}

\subsection{Computations and Lean verification}

We found the family by computing left cell $W$-graphs using Nguyen's
method \cite{Nguyen}. We computed the cell $W$-graphs for
all shapes of size at most $19$ and selected larger shapes. For $(8,6,5)$, the nonzero weights are
$1$ and $5$. For $(9,7,7)$, of dimension $38\,244\,492$, they are
$1,2,3,4,5,7,19$. Completing selected tableau pairs to rectangles
produced web pictures in which the repeated pattern became visible.
This continues the computational study of large KL coefficients in
\cite{LTVBigData}; a different approach to cell bases through
optimization appears in \cite{GoertzenWilliamson}. \autoref{sec:methods} describes the discovery,
including AI assistance. A Lean verification and more code related to the paper 
can be found in \cite{GoLaTuVa-mu-code}.

\begin{remark}\label{rem:verification}
The Lean formalization verifies the local tensor and diagram
calculations, the lattice induction and the two-state readout for
every block count. The cited representation theoretic results, e.g. \cite{FKK,KK}, are
explicit inputs (in other words, we only verify the 
contributions of this paper, not its dependencies).
\end{remark}

\subsection*{Acknowledgments}

We thank the ANU Kioloa Coastal Campus, where the common interest of
TG and DT was discovered at a workshop organized by Anthony Licata
and Oded Yacobi. We thank Keeley Hoek for helping with
translating C code into very efficient Rust code for computing
$W$-graphs, and Lukas Lewark for making the
\texttt{webdraw} program available for computing Kuperberg polynomials.

TG acknowledges funding by the Australian Government
through the ARC Discovery Project DP230102982. DT was supported by 
ARC Future Fellowship FT230100489 and notes for the record that, 
shocking as it may seem, actual human thought went into constructing these pages.
P.V. is supported by the Fonds de la Recherche Scientifique-FNRS (Belgium) under Grant no. J.0235.26.

\section{KL polynomials, cells and the family}\label{sec:preliminaries}

This paper is largely self-contained, although we assume some familiarity with KL combinatorics (for the symmetric group) and webs. Standard references include \cite{BjornerBrenti,Kuperberg}.

\subsection{KL polynomials and the critical coefficient}

Let $S_n$ be the symmetric group with adjacent transpositions
$s_i=(i,i+1)$. We compose permutations by $(xy)(i)=x(y(i))$
and write them in one-line notation. 
Write $\ell(w)$ for the number of inversions of $w$, called its length.
This is also the minimum number of adjacent transpositions in a product equal to $w$;
an expression attaining this minimum is reduced.
The Bruhat relation $x\le y$ means that a reduced expression for $x$
is a subword of a reduced expression for $y$.
The Hecke algebra over $\mathbb Z[\qpar^{1/2},\qpar^{-1/2}]$ has basis
$\{T_w\mid w\in S_n\}$, braid relations, and quadratic relations
$(T_s-\qpar)(T_s+1)=0$. Its KL basis is characterized by
\[
C'_y=\qpar^{-\ell(y)/2}\sum_{x\le y}P_{x,y}(\qpar)T_x,
\qquad \overline{C'_y}=C'_y.
\]
Here $P_{x,y}(\qpar)\in\mathbb Z[\qpar]$, $\overline{\qpar^{1/2}}=\qpar^{-1/2}$,
$\overline{T_w}=T_{w^{-1}}^{-1}$, $P_{y,y}=1$, and
$\deg P_{x,y}\le(\ell(y)-\ell(x)-1)/2$ for $x<y$ \cite{KL}. The following 
are the \emph{$\mu$-coefficients}.

\begin{definition}
We write
\[
\mu(x,y)=
\begin{cases}
[\qpar^{(\ell(y)-\ell(x)-1)/2}]P_{x,y}(\qpar)&x<y\text{ with odd length gap},\\
0&\text{otherwise}.
\end{cases}
\]
The notation $[\qpar^k]f(\qpar)$ means the coefficient of $\qpar^k$ in $f(\qpar)$,
and similarly for other variables. 
\end{definition}

In the definition of $\mu$, the exponent is the allowed top degree,
which need not be the actual degree of the polynomial.
For example, $P_{x,y}=1$ gives $\mu(x,y)=0$ when the length gap is
greater than one, even though the polynomial has leading coefficient one.

\subsection{Left cells and standard tableaux}

The multiplication formula
\[
C'_s C'_w=
\begin{cases}
(\qpar^{1/2}+\qpar^{-1/2})C'_w,&sw<w,\\
C'_{sw}+\displaystyle\sum_{\substack{z<w,sz<z}}\mu(z,w)C'_z,&sw>w,
\end{cases}
\]
explains why the $\mu$-coefficients determine cell representations
\cite{KL}. Define $x\preceq_L y$ if there is a sequence
$y=w_0,w_1,\ldots,w_k=x$ in which $C'_{w_{j+1}}$ occurs with nonzero
coefficient in some $C'_sC'_{w_j}$. The equivalence classes for mutual
reachability are the left cells. The spans of downward-closed
unions of cells are left ideals; successive quotients give the
left cell modules. Descent sets and $\mu$-weights encode their
generator actions in a $W$-graph.

A standard tableau of shape $\lambda\vdash n$ fills the Young diagram
of $\lambda$ with $1,\ldots,n$, increasingly along rows and columns.
Use ordinary row insertion for the Robinson--Schensted--Knuth (RSK) correspondence
$w\mapsto(P(w),Q(w))$, where $P$ is the insertion tableau and $Q$ the
recording tableau. In type $A$, the left cells are exactly the sets
with fixed $Q$, see \cite[Theorem A]{Ariki}; fixing only its shape gives a two-sided cell, cf. \cite[p.~199, Exercise~11(b)]{BjornerBrenti}.
The resulting cell modules, specialized at $\qpar=1$,
are simple and give all simple representations of $S_n$.
They are Specht modules; transposing the partition corresponds to
the sign twist between the two usual KL basis conventions. Throughout
this paper, a cell's shape means the shape of its recording tableau
under row insertion, independently of this convention for labeling representations.

\subsection{Words and the explicit family}

To specify both the diagrams and the permutations, we use three-letter
words. This avoids choosing an auxiliary tableau in the RSK correspondence. We use the English convention for tableaux.

\begin{definition}\label{def:words2}
Order $A<B<C$ and let $N\ge1$. A word of content $(N,N,N)$ contains $N$ copies of each
letter. It is dominant if every prefix satisfies
$\#A\ge\#B\ge\#C$. Its tableau has the positions of $A,B,C$ in its
first, second, and third rows. 
\end{definition}

\begin{lemma}
Dominance is exactly the condition
that this row-increasing tableau is standard.
\end{lemma}

\begin{proof}
The $i$th $A$ precedes the $i$th $B$, and the $i$th $B$
precedes the $i$th $C$, exactly when every prefix has
$\#A\ge\#B\ge\#C$. These are precisely the column inequalities.
\end{proof}

\begin{definition}\label{def:words}
For any word $w$ of this content, define $\pi(w)\in S_{3N}$ by listing
the positions of $C$, then of $B$, then of $A$, increasingly within
each list. 
\end{definition}

\begin{example}
For instance,
\[
ABCABC\longleftrightarrow
\begin{array}{cc}1&4\\2&5\\3&6\end{array},
\qquad \pi(ABCABC)=(3,6,2,5,1,4),
\]
where the final entry is one-line notation.
\end{example}

The map in \autoref{def:words} does more than assign permutations
to tableaux of equal shape: it fixes their recording tableau.

\begin{lemma}\label{lem:rectangular-cell}
Let $w$ be a dominant word of content $(N,N,N)$, with tableau $T(w)$.
Then
\[
P\big(\pi(w)\big)=T(w),\qquad
Q\big(\pi(w)\big)=
\begin{array}{cccc}
1&2&\cdots&N\\
N+1&N+2&\cdots&2N\\
2N+1&2N+2&\cdots&3N
\end{array}.
\]
Consequently, all these permutations belong to a single left cell
of shape $(N,N,N)$.
\end{lemma}

\begin{proof}
Write the three rows of $T(w)$ as $(a_i)$, $(b_i)$, $(c_i)$, so that
$a_i<b_i<c_i$. The permutation $\pi(w)$ is the concatenation of
$(c_i)$, $(b_i)$, $(a_i)$. Inserting the first block creates the first
row. When $b_i$ is inserted, the first row is
$b_1,\ldots,b_{i-1},c_i,\ldots,c_N$: thus $c_i$ is bumped and appended
to the second row. The next block similarly bumps $b_i$ and then
$c_i$, creating the $i$th box of the third row. Hence the insertion
tableau is $T(w)$, and successive blocks create all boxes in rows
one, two and three, in that order. This gives the displayed $Q$.
The type $A$ cell classification gives the last assertion.
\end{proof}

\autoref{lem:rectangular-cell} gives the common left cell assertion
in \autoref{thm:main}.

\begin{definition}\label{D:words}
For $m\ge0$, define
\begin{equation}\label{eq:words}
\begin{gathered}
\wordstart=\word{ABAAAC}\text{ (start)},\qquad \wordmiddle=\word{BBBAACACC}\text{ (middle)},\qquad \wordend=\word{BBCBCC}\text{ (end)},\\
N_m=4+3m,\qquad a_m=(ABC)^{N_m},\qquad c_m=\wordstart\wordmiddle^m\wordend,\\
x_m=\pi(a_m),\qquad y_m=\pi(c_m).
\end{gathered}
\end{equation}
We call $m$ the block count; $\wordstart,\wordmiddle,\wordend$ are the prefix,
repeated block and tail, respectively.
\end{definition}

\begin{lemma}
The words $a_m,c_m$ are dominant, and $x_m,y_m$ correspond to
tableaux of shape $(N_m,N_m,N_m)$.
\end{lemma}

\begin{proof}
Dominance of $a_m=(ABC)^{N_m}$ is immediate.
The contents of $\wordstart,\wordmiddle,\wordend$ are $(4,1,1)$, $(3,3,3)$ and $(0,3,3)$,
respectively, so both words have content $(N_m,N_m,N_m)$.

Every prefix of $\wordstart$ is dominant, and its final prefix differences
are $(3,0)$. Starting from these differences, every prefix of $\wordmiddle$
preserves dominance, and the whole block returns to $(3,0)$.
The final block $\wordend$ preserves dominance and ends at $(0,0)$.
Thus $c_m$ is dominant. The preceding lemma gives the claimed
tableau shape.
\end{proof}

\begin{remark}\label{rem:minimal-cell}
The tableau with row word $a_m$ is filled down columns. The permutation
$x_m$ is the least element of its left cell in Bruhat order and in left
weak order; cf. \cite[Remark 6.21]{Nguyen}. Indeed, let $u$ relabel this
tableau to any standard tableau $T(w)$ of the same shape. Then
$\pi(w)=ux_m$ and $\ell(\pi(w))=\ell(u)+\ell(x_m)$: an inversion of $u$
can only go from an earlier column and a lower row to a later column
and a higher row. In particular, $x_m\le y_m$.
\end{remark}

\begin{lemma}\label{lem:lengths}
For the words in \autoref{eq:words}, $x_m<y_m$ and
\[
\ell(x_m)=\frac{27m^2+81m+60}{2},\qquad
\ell(y_m)=\frac{27m^2+115m+82}{2},\qquad
\delta_m:=\ell(y_m)-\ell(x_m)=17m+11.
\]
\end{lemma}

\begin{proof}
Put $N=N_m$. The inequality $x_m\le y_m$ follows from
\autoref{rem:minimal-cell}.

For $i<j$, the pair $(i,j)$ is an inversion of $\pi(w)^{-1}$
exactly when $w_i<w_j$.
Within $\wordstart,\wordmiddle,\wordend$, there are $6,17,8$ such pairs,
respectively. Each pair of distinct pieces, in the order in $c_m$,
contributes $27$. Hence
\[
\ell(y_m)=6+17m+8+27\left(2m+1+\binom m2\right),\qquad
\ell(x_m)=3N+3\binom N2.
\]
Expanding gives the stated lengths and the strict inequality.
\end{proof}

For odd $m$, the length gap in
\autoref{lem:lengths} is even, so $\mu(x_m,y_m)=0$.
Thus it remains to compute $\mu(x_m,y_m)$ for even $m$.

\section{Webs and tensor coefficients}\label{sec:webs}

We now turn to webs. The basic idea is simple. The $\mathfrak{sl}_3$ invariant spaces $\Inv(V^{\otimes3N})$,
where $V$ is the three-dimensional fundamental representation, carry two natural bases: the web basis $W_w$, which is concrete and comparatively easy to compute, and the dual canonical basis $b_w$, which is subtler but, up to the standard identifications, corresponds to the KL basis. Thus our main task is to understand the change of basis from $W_w$ to $b_w$. Once this is explicit enough for our family, \autoref{thm:main} follows. This will occupy most of the paper.

\subsection{The planar calculus}

Fix $\F=\mathbb Q(\tpar)$ and $R=\mathbb Z[\tpar]$, with $\tpar$ a formal
parameter. Joining strands composes maps; placing
pictures side by side takes their tensor product. 
Moreover, before we even start recalling webs, we have the structure
of an operad/planar algebra:

\begin{lemma}\label{lem:rotation}
For a pair of three-row rectangular tableaux, let $x,y$ be their
associated permutations from \autoref{def:words}. The symmetrized
coefficient $\mu(x,y)+\mu(y,x)$ is unchanged after simultaneously
rotating their webs and applying \autoref{def:words} to the resulting
tableaux.
\end{lemma}

\begin{proof}
Rotation of a three-row rectangular web corresponds to promotion
of its tableau, see \cite[Theorem 2.5]{PPR}, and simultaneous promotion
preserves the symmetrized coefficient, see \cite[Proposition 3.2]{Rhoades}.
Rhoades writes this coefficient as $\mu[P,Q]$; the maximum of the
two ordered coefficients equals our sum, since at most one is nonzero.
\end{proof}

We draw most webs in a disk. When a picture specifies a map
or a tensor coordinate, its boundary order and labels are fixed.

\begin{definition}\label{def:web}
An $A_2$ web, or $\mathfrak{sl}_3$ web, is a planar graph in a disk or rectangle, with
ordered univalent boundary vertices and trivalent internal vertices.
Internal vertices are bipartite, and we draw them hollow and filled
for sources and sinks. Every edge is oriented from a source to a sink;
boundary signs use the input/output convention below. Closed circle components
are allowed. The boundary order, boundary signs, and cyclic orders of
edges at vertices are part of the data.

Set
$d=\tpar+\tpar^{-1}$, and $e=\tpar^2+1+\tpar^{-2}=d^2-1$.
The web category consists of $\F$-linear combinations of these
webs, modulo the circle, digon, and square relations in
\autoref{fig:relations}.

A web is non-elliptic if it has no circle component and no
internal digon or square face.
\end{definition}

\begin{figure}[!ht]
\centering
\begin{tikzpicture}[x=.8cm,y=.8cm,font=\small,baseline=0]
\draw[web] (0,0) circle (.45);
\node at (1.25,0) {$=e$,};
\begin{scope}[shift={(3,0)}]
\draw[web] (-.85,0)--(-.4,0) (.4,0)--(.85,0);
\draw[web] (-.4,0) to[bend left=55] (.4,0);
\draw[web] (-.4,0) to[bend right=55] (.4,0);
\node[source] at (-.4,0) {};\node[sink] at (.4,0) {};
\node at (1.25,0) {$=d$};\draw[web] (1.8,0)--(2.9,0);
\node at (3.1,0) {$,$};
\end{scope}
\begin{scope}[shift={(8.1,0)}]
\draw[web] (-.4,-.4) rectangle (.4,.4);
\draw[web]
  (-.4,.4)--(-.8,.8) (.4,.4)--(.8,.8)
  (-.4,-.4)--(-.8,-.8) (.4,-.4)--(.8,-.8);
\node[source] at (-.4,.4) {};
\node[sink] at (.4,.4) {};
\node[sink] at (-.4,-.4) {};
\node[source] at (.4,-.4) {};
\node at (1.4,0) {$=$};
\begin{scope}[shift={(2.8,0)}]
\draw[web] (-.8,.8) to[bend left=20] (-.8,-.8);
\draw[web] (.8,.8) to[bend right=20] (.8,-.8);
\end{scope}
\node at (4,0) {$+$};
\begin{scope}[shift={(5.2,0)}]
\draw[web] (-.8,-.8) to[bend left=20] (.8,-.8);
\draw[web] (-.8,.8) to[bend right=20] (.8,.8);
\end{scope}
\node at (6.4,0) {$.$};
\end{scope}
\end{tikzpicture}
\caption{The $A_2$ relations. Hollow and filled internal vertices are
sources and sinks, respectively. All undisplayed parts of each web
remain fixed. Reversing all vertex types gives the same relations.}
\label{fig:relations}
\end{figure}

\begin{remark}
For the reader familiar with e.g. \cite{Kuperberg}:
hollow and filled internal vertices specify orientation, so arrows
can be omitted. At a positive output boundary vertex the edge enters the diagram;
at a positive input it leaves the diagram. Negative signs reverse
these conventions.
\end{remark}

Let $V^+$ be the three-dimensional fundamental $U_{\vpar}(\mathfrak{sl}_3)$
module and $V^-$ its dual, in the normalization of \autoref{eq:localmaps} (we follow \cite{KK}); put $V=V^+$. 
An ordered sequence of signs represents a tensor product of these
modules. Vertical composition and horizontal juxtaposition represent
composition and tensor product of maps. A web with no inputs and $n$
positive outputs is therefore a map $\one\to V^{\otimes n}$, or
equivalently an invariant tensor. Here $\one$ is the one-dimensional
trivial module.

\begin{proposition}\label{thm:spider}
Non-elliptic $A_2$ webs form a basis of the corresponding
invariant tensor space. The web relations present the associated
intertwiner calculus.
\end{proposition}

\begin{proof}
This is the $A_2$ case of \cite[Sections 4--6, Theorems 5.1 and 6.1]{Kuperberg}.
\end{proof}

The non-elliptic webs form bases of the invariant spaces, while the
standard tensors form a basis of the ambient tensor power. For $n=3N$,
the all-positive invariant space is the multiplicity space corresponding
to the rectangular partition $(N,N,N)$ under quantum Schur--Weyl
duality; see \cite[Section 2]{FKK}. Below we work with the tensor
realization.

\subsection{Tableaux and the displayed family}

The dominant word of \autoref{def:words2} determines the leading
boundary state of a non-elliptic web. 
Starting from the tableau, read the second row from left to right,
matching each entry with the nearest unused smaller first-row entry.
Then match the third row with the second row in the same way.
Insert a trivalent vertex above each second-row boundary point,
and replace each crossing by a dumbbell, e.g.
\[
\begin{tikzpicture}[x=.8cm,y=.8cm,font=\small,baseline=0]
\draw[web] (-.8,.8) to (.8,-.8);
\draw[web] (-.8,-.8) to (.8,.8);
\end{tikzpicture}
\rightsquigarrow
\begin{tikzpicture}[x=.8cm,y=.8cm,font=\small,baseline=0]
\draw[web] (.8,.8) to (0,.4);
\draw[web] (-.8,.8) to (0,.4);
\draw[web] (.8,-.8) to (0,-.4);
\draw[web] (-.8,-.8) to (0,-.4);
\draw[web] (0,.4) to (0,-.4);
\node[source] at (0,.4) {};
\node[sink] at (0,-.4) {};
\end{tikzpicture}
\text{ or }
\begin{tikzpicture}[x=.8cm,y=.8cm,font=\small,baseline=0]
\draw[web] (.8,.8) to (.4,0);
\draw[web] (-.8,.8) to (-.4,0);
\draw[web] (.8,-.8) to (.4,0);
\draw[web] (-.8,-.8) to (-.4,0);
\draw[web] (.4,0) to (-.4,0);
\node[source] at (.4,0) {};
\node[sink] at (-.4,0) {};
\end{tikzpicture}
,
\]
where the choice between I (middle) and H (right) is determined by bipartiteness.
We denote the resulting web by $W_w$.

\begin{example}\label{ex:tableau-to-web}
For $w=AABCBC$, the first matching gives $(2,3),(1,5)$,
and the second gives $(3,4),(5,6)$.
The arcs do not cross:
\[
\begin{tikzpicture}[x=.8cm,y=.8cm,font=\small]
\node at (-3.9,0)
  {$\begin{array}{cc}1&2\\3&5\\4&6\end{array}$};
\node at (-2.65,0) {$\longmapsto$};
\foreach \i/\a in {1/180,2/120,3/60,4/0,5/-60,6/-120}
  \coordinate (b\i) at (\a:1.3);
\draw[web,draw=black] (b2) to[bend right=15] (b3);
\draw[web,draw=black] (b1) to[bend right=15] (b5);
\draw[web,draw=black] (b3) to[bend right=15] (b4);
\draw[web,draw=black] (b5) to[bend right=15] (b6);
\foreach \i/\a in {1/180,2/120,3/60,4/0,5/-60,6/-120} {
  \node[circle,draw,fill=white,inner sep=.9pt] at (b\i) {};
  \node at (\a:1.6) {$\i$};
}
\node at (0,-1.9) {arc diagram};
\node at (2.35,0) {$\longmapsto$};
\begin{scope}[shift={(4.7,0)}]
\foreach \i/\a in {1/180,2/120,3/60,4/0,5/-60,6/-120}
  \coordinate (b\i) at (\a:1.3);
\coordinate (R) at (60:.65);
\coordinate (S) at (240:.65);
\draw[web,draw=black]
  (b2)--(R)--(b3) (R)--(b4)
  (b1)--(S)--(b5) (S)--(b6);
\foreach \v in {R,S}
  \node[sink,draw=black,fill=black] at (\v) {};
\foreach \i/\a in {1/180,2/120,3/60,4/0,5/-60,6/-120} {
  \node[circle,draw,fill=white,inner sep=.9pt] at (b\i) {};
  \node at (\a:1.6) {$\i$};
}
\node at (0,-1.9) {$W_w$};
\end{scope}
\end{tikzpicture}
.
\]
Resolving the two boundary junctions gives tripods on
$(2,3,4)$ and $(1,5,6)$.

For $w=ABACBC$, the first matching gives $(1,2),(3,5)$,
and the second gives $(2,4),(5,6)$.
The arcs $(2,4)$ and $(3,5)$ cross once:
\[
\begin{tikzpicture}[x=.8cm,y=.8cm,font=\small]
\node at (-3.9,0)
  {$\begin{array}{cc}1&3\\2&5\\4&6\end{array}$};
\node at (-2.65,0) {$\longmapsto$};
\foreach \i/\a in {1/180,2/120,3/60,4/0,5/-60,6/-120}
  \coordinate (b\i) at (\a:1.3);
\draw[web] (b1) to[bend right=15] (b2);
\draw[web] (b3) to[bend right=15] (b5);
\draw[web] (b2) to[bend right=15] (b4);
\draw[web] (b5) to[bend right=15] (b6);
\foreach \i/\a in {1/180,2/120,3/60,4/0,5/-60,6/-120} {
  \node[circle,draw,fill=white,inner sep=.9pt] at (b\i) {};
  \node at (\a:1.6) {$\i$};
}
\node at (0,-1.9) {arc diagram};
\node at (2.35,0) {$\longmapsto$};
\begin{scope}[shift={(4.7,0)}]
\foreach \i/\a in {1/180,2/120,3/60,4/0,5/-60,6/-120}
  \coordinate (b\i) at (\a:1.3);
\coordinate (L) at (150:.65);
\coordinate (R) at (30:.65);
\coordinate (S) at (270:.65);
\draw[web]
  (b1)--(L)--(b2)
  (b3)--(R)--(b4)
  (b5)--(S)--(b6);
\draw[web] (L)--(0,0)--(R) (0,0)--(S);
\foreach \v in {L,R,S}
  \node[sink] at (\v) {};
\node[source] at (0,0) {};
\foreach \i/\a in {1/180,2/120,3/60,4/0,5/-60,6/-120} {
  \node[circle,draw,fill=white,inner sep=.9pt] at (b\i) {};
  \node at (\a:1.6) {$\i$};
}
\node at (0,-1.9) {$W_w$};
\end{scope}
\end{tikzpicture}
.
\]
Resolving the two boundary junctions and the crossing gives $W_w$.
\end{example}

\begin{lemma}
The web $W_w$ is non-elliptic.
\end{lemma}

\begin{proof}
The KK growth
algorithm constructs this web directly from the word
\cite[Section 5]{KK}. The above description is Tymoczko's
equivalent construction \cite[Sections 2--4, Theorem 4.9]{Tymoczko}.
\end{proof}

\begin{remark}
There is another useful tableau-to-web construction in
\cite[Section 4.1]{MPT}. The algorithm therein is
more general: it allows arbitrary sign strings and semistandard tableaux,
and naturally produces a web together with a flow. In the all-positive
standard case considered here, it recovers the same KK
web. Tymoczko's construction is more convenient for our purposes because
the intermediate arc diagram makes the web easy to draw.
\end{remark}

\subsection{Elementary tensors and decorations}\label{sec:elementarytensors}

We fix the coordinates of the elementary web maps in
\autoref{fig:generators}. Here we draw maps in a rectangle,
with inputs at the bottom and outputs at the top.

\begin{figure}[!ht]
\centering
\begin{tikzpicture}[x=1cm,y=.8cm,font=\small]
\foreach \x in {8.6,11.1,13.6} {
  \begin{scope}[shift={(\x,0)}]
  \draw[line width=.45pt]
    (-.85,1.6)--(-1,1.6)--(-1,-2.3)--(-.85,-2.3)
    (.85,1.6)--(1,1.6)--(1,-2.3)--(.85,-2.3);
  \end{scope}
}
\draw[web] (-.55,1)--(0,0)--(.55,1) (0,0)--(0,-1);
\node[sink] at (0,0) {};
\node[above] at (-.55,1) {$+$};
\node[above] at (.55,1) {$+$};
\node[below] at (0,-1) {$-$};
\node at (0,-1.9) {$y$};
\begin{scope}[shift={(2.1,0)}]
\draw[web] (-.55,-1)--(0,0)--(.55,-1) (0,0)--(0,1);
\node[source] at (0,0) {};
\node[above] at (0,1) {$-$};
\node[below] at (-.55,-1) {$+$};
\node[below] at (.55,-1) {$+$};
\node at (0,-1.9) {$\lambda$};
\end{scope}
\begin{scope}[shift={(4.2,0)}]
\draw[web] (-.55,1)
  .. controls (-.55,-.65) and (.55,-.65) .. (.55,1);
\node[above] at (-.55,1) {$+$};
\node[above] at (.55,1) {$-$};
\node at (0,-1.9) {$\sigma'$};
\end{scope}
\begin{scope}[shift={(6.3,0)}]
\draw[web] (-.55,-1)
  .. controls (-.55,.65) and (.55,.65) .. (.55,-1);
\node[below] at (-.55,-1) {$-$};
\node[below] at (.55,-1) {$+$};
\node at (0,-1.9) {$\sigma$};
\end{scope}
\begin{scope}[shift={(8.6,0)}]
\draw[web] (-.5,1)--(-.5,0)--(-.5,-1)
  (.5,1)--(.5,0)--(.5,-1) (-.5,0)--(.5,0);
\node[sink] at (-.5,0) {};
\node[source] at (.5,0) {};
\node[above] at (-.5,1) {$+$};
\node[above] at (.5,1) {$-$};
\node[below] at (-.5,-1) {$-$};
\node[below] at (.5,-1) {$+$};
\node at (0,-1.9) {$H$};
\end{scope}
\begin{scope}[shift={(11.1,0)}]
\draw[web] (-.6,1)--(0,-.35)--(.6,1) (0,-.35)--(0,1);
\node[sink] at (0,-.35) {};
\foreach \x in {-.6,0,.6}
  \node[above] at (\x,1) {$+$};
\node at (0,-1.9) {$\tau$};
\end{scope}
\begin{scope}[shift={(13.6,0)}]
\draw[web] (-.6,-1)--(0,.35)--(.6,-1) (0,.35)--(0,-1);
\node[source] at (0,.35) {};
\foreach \x in {-.6,0,.6}
  \node[below] at (\x,-1) {$+$};
\node at (0,-1.9) {$\eta$};
\end{scope}
\end{tikzpicture}
\caption{The generators $y,\lambda,\sigma',\sigma$ are the split,
merge, cup and cap. The derived maps $H,\tau,\eta$ are bracketed;
$\tau$ is the tripod and $\eta$ its evaluation. Inputs are at the
bottom and outputs at the top; $+,-$ specify modules. Reversing
all signs and vertex types gives the other sign patterns.}
\label{fig:generators}
\end{figure}
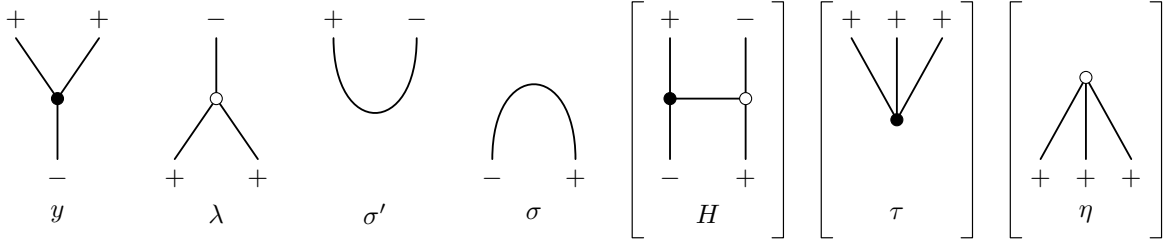

Write $e_A,e_B,e_C$ for the standard basis vectors of either
signed module $V^\pm$, labeled by $1,0,-1$, respectively.
Their weights are $\varepsilon_1,\varepsilon_2,\varepsilon_3$
in $V^+$ and $-\varepsilon_3,-\varepsilon_2,-\varepsilon_1$ in $V^-$.
Here the $\varepsilon_i$ are the coordinate weights on the diagonal
traceless Cartan subalgebra, so $\varepsilon_1+\varepsilon_2+\varepsilon_3=0$.
For a word $w=w_1\cdots w_n$ in $\{A,B,C\}$, put
$e_w=e_{w_1}\otimes\cdots\otimes e_{w_n}$.
We often abbreviate $e_A,e_B,e_C$ to $A,B,C$ and $e_w$ to $w$,
with the signed tensor factors understood.
For example, $AB=e_A\otimes e_B$.
Comparisons use the order $A<B<C$, and $\delta_{a<b}$ is an indicator.
The notation $[e_w]f$ denotes the coefficient of $e_w$ in $f$.

The cup and cap are
$\sigma':\one\to V^+\otimes V^-$ and
$\sigma:V^-\otimes V^+\to\one$.
The trivalent maps and the cup and cap are given in coordinates
by the following formulas. Here $\nu(a,b)=\nu(ab)$ is the label assigned to the unordered pair $\{a,b\}$ below.
\begin{equation}\label{eq:localmaps}
\begin{gathered}
y(A)=AB+\tpar BA,\qquad y(B)=AC+\tpar CA,\qquad y(C)=BC+\tpar CB,\\
\lambda(ab)=\tpar^{-\delta_{a<b}}\nu(a,b)\quad(a\ne b),
\qquad \lambda(aa)=0,\\
\nu(AB)=A,\qquad \nu(AC)=B,\qquad \nu(BC)=C,
\qquad \nu(a,b)=\nu(b,a),\\[2pt]
\sigma'(1)=AC+\tpar BB+\tpar^2CA,\qquad
\sigma(AC)=\tpar^{-2},\quad \sigma(BB)=\tpar^{-1},\quad \sigma(CA)=1,\\
\sigma(s)=0\quad(s\notin\{AC,BB,CA\}).
\end{gathered}
\end{equation}
The same formulas apply after globally reversing the signs and vertex
types. The map $H$ exchanges opposite boundary signs by a planar web.

\begin{remark}
The formulas in \autoref{eq:localmaps} are as in
\cite[Section 2]{KK}, with $\tpar=-q_{\rm KK}^{-1/2}$.
\end{remark}

For the derived maps in \autoref{fig:generators}, put
$H=(I\otimes\lambda)(y\otimes I)$,
$\tau=(I\otimes y)\sigma'$ and $\eta=\sigma(\lambda\otimes I)$.
Their coordinates are
\[
\begin{array}{c|ccccccccc}
s&AA&AB&AC&BA&BB&BC&CA&CB&CC\\\hline
H(s)&AA&BA&\tpar^{-1} AC+BB&AB&AC+CA&CB&BB+\tpar CA&BC&CC
\end{array},
\]
while the nonzero coordinates of $\tau,\eta$, ordered as $ABC,ACB,BAC,BCA,CAB,CBA$, are
\begin{equation}\label{eq:tripod}
\tau=(1,\tpar,\tpar,\tpar^2,\tpar^2,\tpar^3)^{\mathsf T},\qquad
\eta=(\tpar^{-3},\tpar^{-2},\tpar^{-2},\tpar^{-1},\tpar^{-1},1).
\end{equation}
In particular,
$\tau^{\mathsf T}=\tpar^3\eta$ and $\eta\tau=de$.
Transpose throughout means ordinary coordinate transpose over $\F$, and we write $\tau^j$ for the tensor
product of $j$ tripods.

\begin{remark}
Fix the states on the boundary edges of a web.
A compatible decoration labels each internal edge by $A,B$ or $C$
so that the corresponding local coefficients in
\autoref{eq:localmaps} are nonzero.
Its weight is the product of these coefficients.
The coefficient for the fixed boundary states is the sum of these
weights over all compatible decorations.
\end{remark}

\subsection{The dual canonical condition}

We first recall the bar involution. For a sign word
$S=s_1\cdots s_n$, write
$V^S=V^{s_1}\otimes\cdots\otimes V^{s_n}$.
On $V^+$ and $V^-$, the bar fixes the standard basis vectors and sends
$\tpar$ to $\tpar^{-1}$. On tensor products it is not the entrywise bar.
Instead, Lusztig's tensor bar
$\Phi_S\colon V^S\longrightarrow V^S$
is defined inductively using the quasi $R$-matrix $\overline\Theta$: if
$S=TT'$, then
$\Phi_S(x\otimes y)
=
\overline\Theta\bigl(\Phi_T(x)\otimes\Phi_{T'}(y)\bigr)$.
Thus $\Phi_S$ is $\tpar$-antilinear and is the bar involution relevant for
the dual canonical basis. We will use \autoref{lem:web-bar} to establish
bar invariance without computing the tensor bar explicitly. For its
construction, see \cite[Chapter 27]{Lusztig} and
\cite[Section 6]{KK}.

\begin{lemma}\label{lem:web-bar}
Every web (without linear combination) is bar equivariant. More precisely, if a web represents
a map
$W\colon V^S\longrightarrow V^{S'}$,
then
$\Phi_{S'}\circ W=W\circ\Phi_S$.
\end{lemma}

\begin{proof}
This is proved for arbitrary monomial web maps in the proof of
\cite[Proposition 2]{KK}.
\end{proof}

The point of \autoref{lem:web-bar} is that bar invariance does not
distinguish the web basis from the dual canonical basis. The difference
is the triangularity condition.

\begin{proposition}\label{thm:canonical}
For every sign word $S$ of length $n$ and every state word
$w\in\{A,B,C\}^n$, there is a unique tensor $b_w\in V^S$
(the dual canonical basis element) fixed by $\Phi_S$ and satisfying
\[
b_w\equiv e_w\pmod{\tpar\bigoplus_z R e_z}.
\]
For $S=+^{3N}$, the vectors $b_w$ indexed by dominant words of
content $(N,N,N)$ form a basis of $\Inv(V^{\otimes3N})$.
\end{proposition}

\begin{proof}
The canonical basis input is Lusztig's tensor product construction
\cite[Chapter 27]{Lusztig}.
In the exact conventions used here, see
\cite[Theorem 3]{KK}. The invariant subspace assertion is
\cite[Section 27.2.5]{Lusztig}, recalled in \cite[Section 6]{KK}.
\end{proof}

\begin{example}\label{ex:six}
There are five dominant words for $(2,2,2)$:
\[
w_1=\word{AABBCC},\quad w_2=\word{AABCBC},\quad
w_3=\word{ABABCC},\quad w_4=\word{ABACBC},\quad w_5=\word{ABCABC}.
\]
Their five webs are shown in \autoref{fig:six}. Columns in the following
matrix are webs, and rows are the selected standard tensor states $e_{w_i}$:
\begin{equation}\label{eq:sixcoordinates}
\bigl([e_{w_i}]W_{w_j}\bigr)_{i,j}=
\begin{pmatrix}
1&0&0&0&0\\
\tpar&1&0&0&0\\
\tpar&0&1&0&0\\
\tpar^2&\tpar&\tpar&1&0\\
\tpar&\tpar^2&\tpar^2&\tpar&1
\end{pmatrix}.
\end{equation}
For example, $[e_{ABCABC}]W_{AABBCC}=\tpar$, while
$[e_{ABCABC}]W_{ABCABC}=1$. The entries come from the elementary
tables above. All five webs are dual canonical, also by \cite[Theorem 1]{KK}:
in their full tensor
expansions the leading coordinate is $1$ and every other coordinate
belongs to $\tpar\Z[\tpar]$. Thus \autoref{eq:sixcoordinates} is a tensor-coordinate
matrix, even though the web-to-dual-canonical matrix here is the identity.
\end{example}

\begin{figure}[!ht]
\centering
\includegraphics[width=\linewidth]{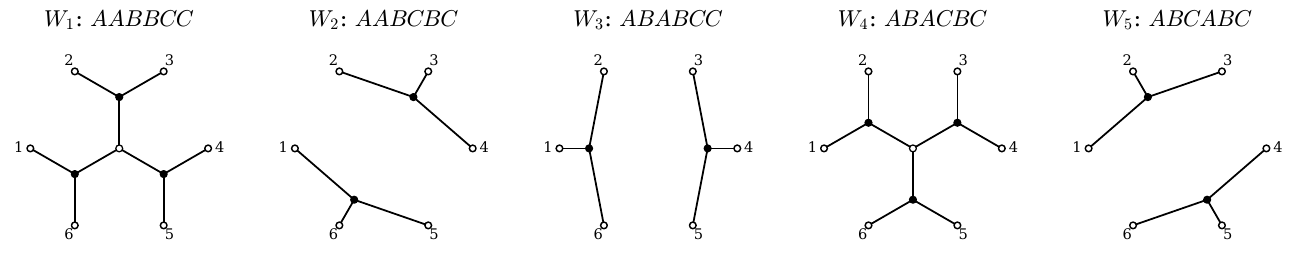}
\caption{The five all-positive six-point webs, in the order of
\autoref{ex:six}. The boundary order starts at the left and goes clockwise.}
\label{fig:six}
\end{figure}

\subsection{A teaser for our main family}

\autoref{fig:family01} and \autoref{fig:family24} show
$W_{a_m}$ and $W_{c_m}$ for $m\in\{0,1,2,4\}$. These webs are
non-elliptic, but $W_{c_m}$ need not be dual canonical. By
\autoref{thm:main} and \autoref{lem:lengths}, the corresponding
$\mu$-coefficients are $1,0,2,4$, respectively.

\begin{figure}[!ht]
\centering
\includegraphics[width=\linewidth]{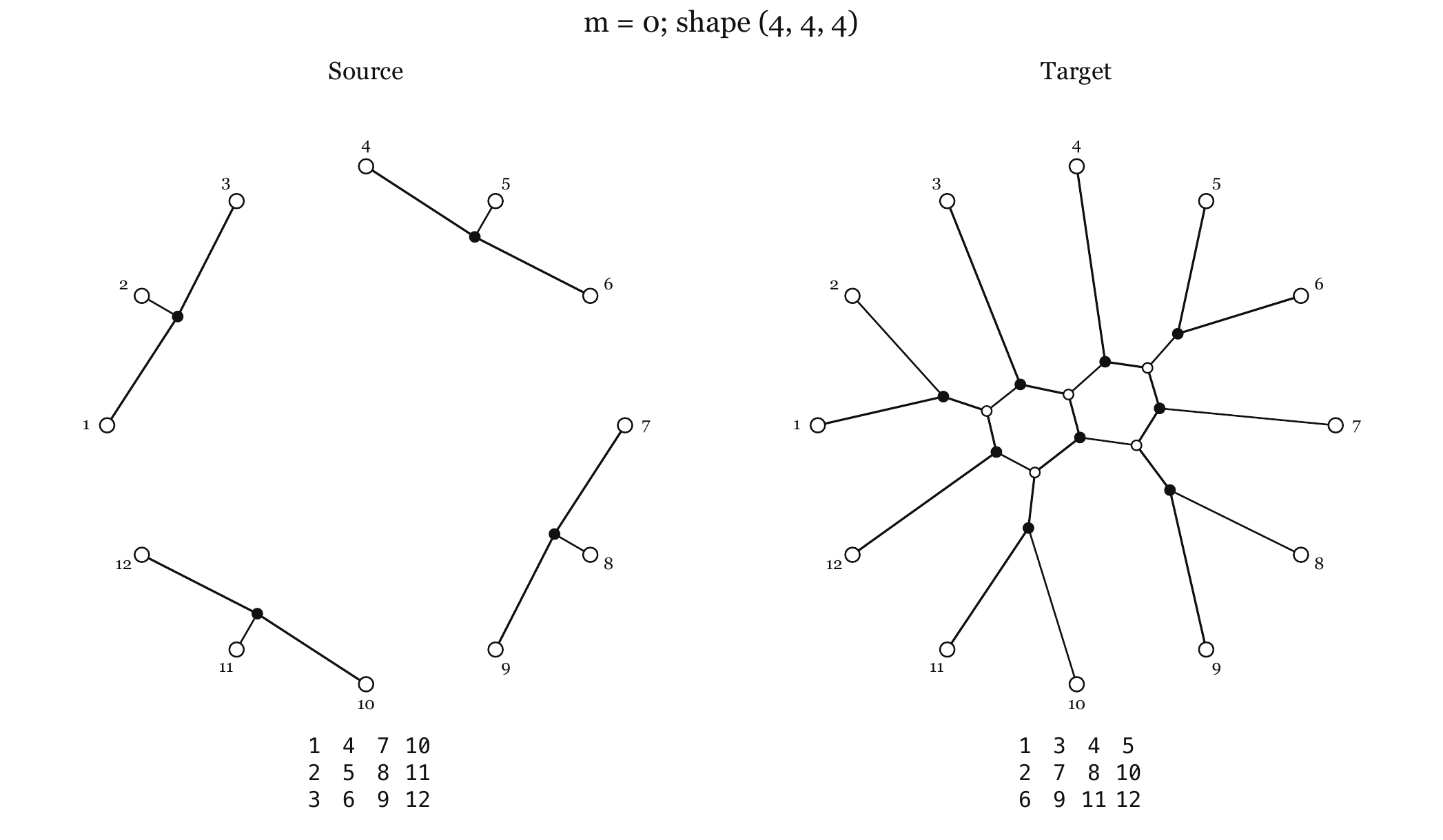}\par\medskip
\includegraphics[width=\linewidth]{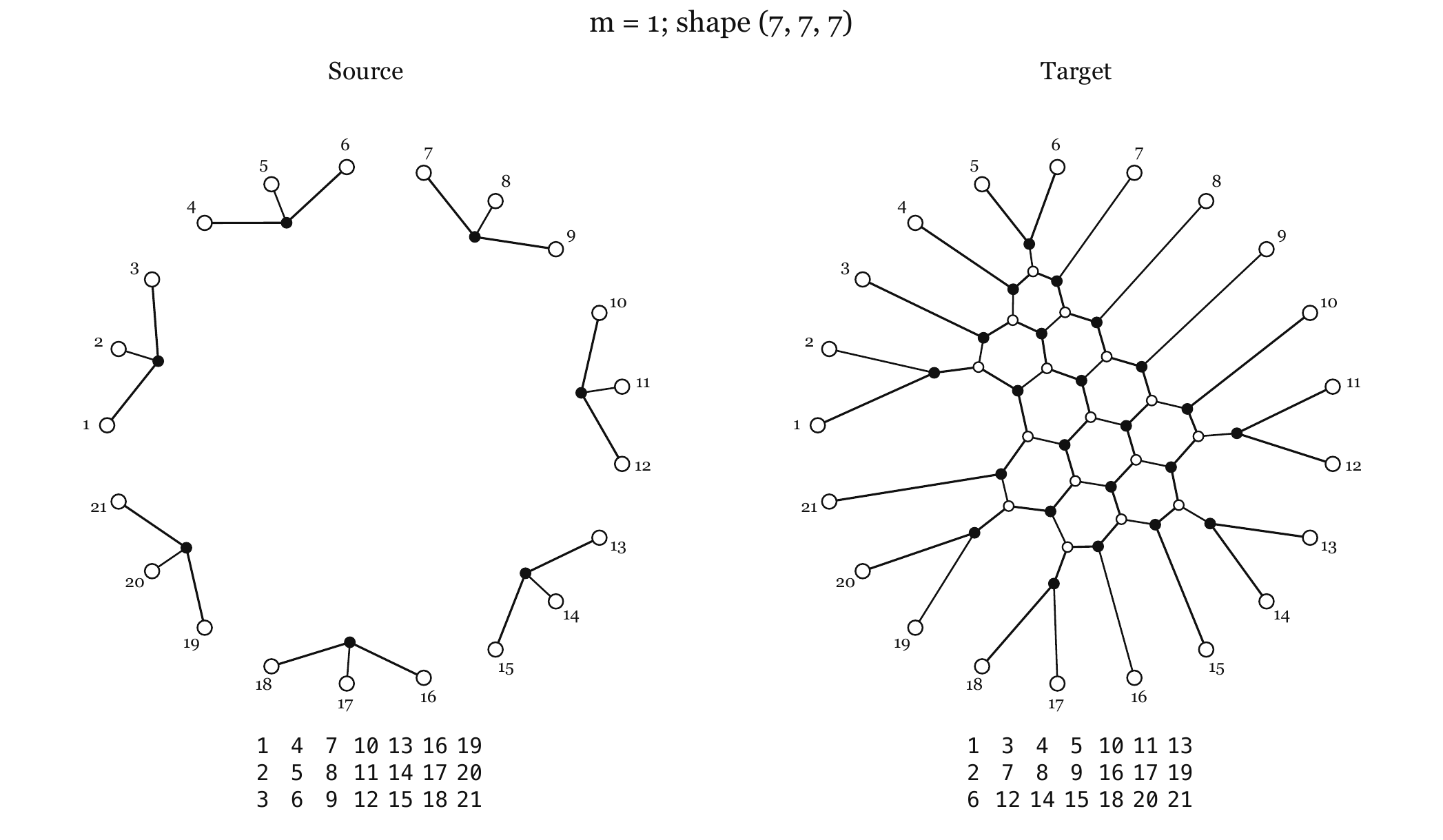}
\caption{The indexing webs for $m=0$ and $m=1$, with tableaux below
each diagram. Boundary labels are fixed; edges join opposite vertex
types.}
\label{fig:family01}
\end{figure}

\begin{figure}[!ht]
\centering
\includegraphics[width=\linewidth]{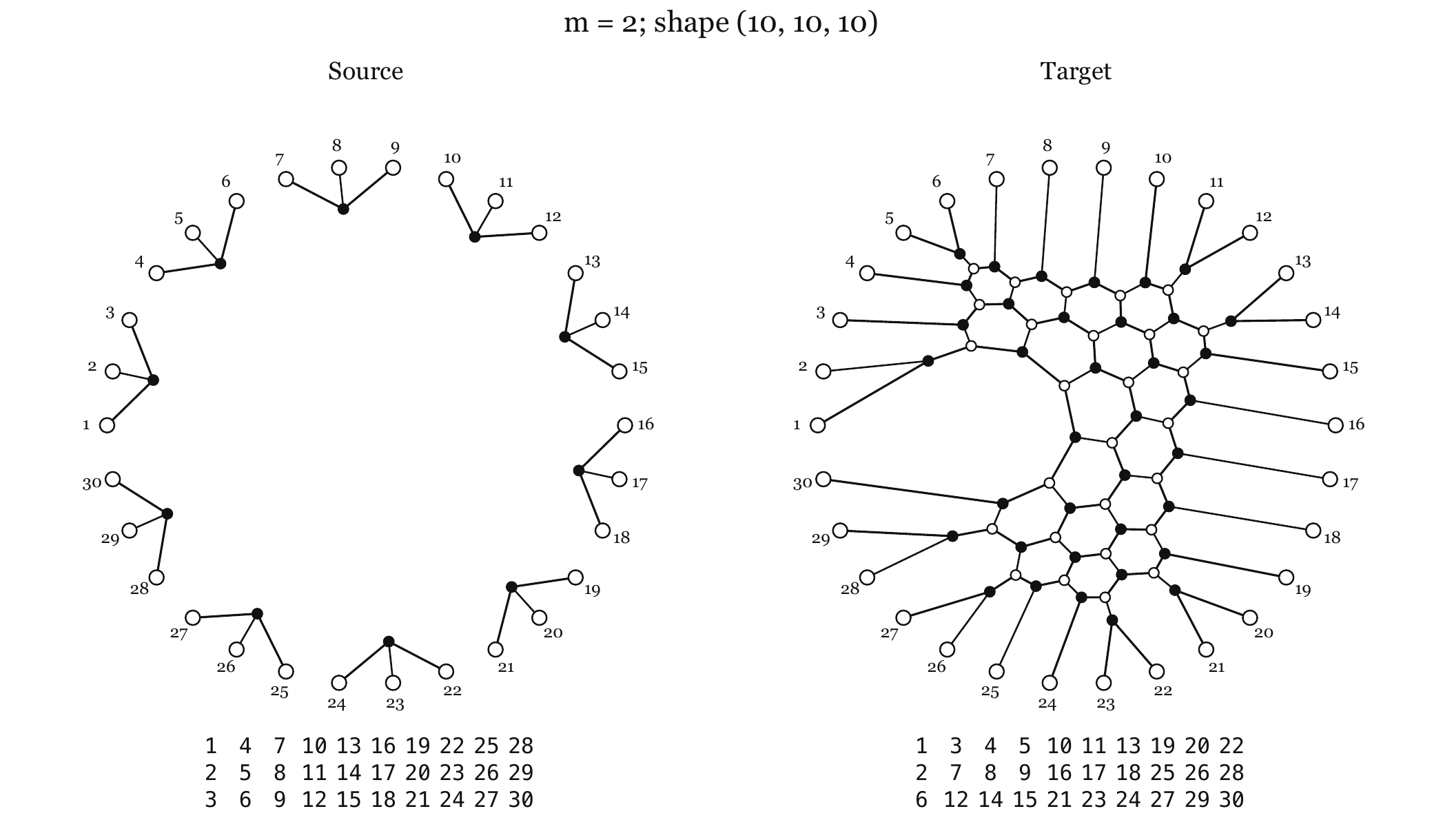}\par\medskip
\includegraphics[width=\linewidth]{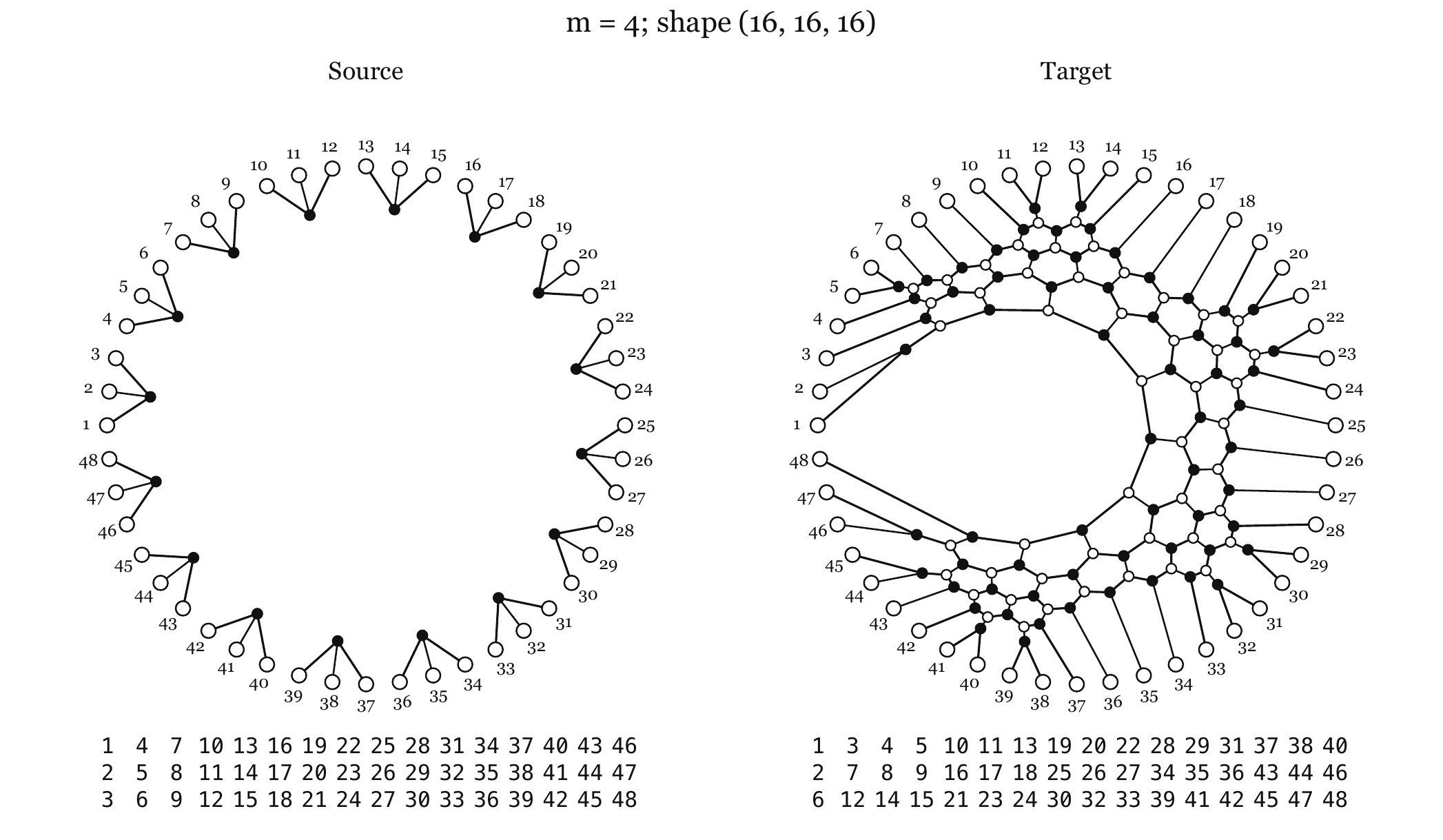}
\caption{The members $m=2$ and $m=4$, of shapes $(10,10,10)$ and
$(16,16,16)$.}
\label{fig:family24}
\end{figure}

\section{The coefficient to be computed}\label{sec:dictionary}

\addtocounter{subsection}{1}

The following statement reduces the calculation of
$\mu(x_m,y_m)$ to a single tensor coordinate.

\begin{proposition}\label{thm:FKK}
For every even $m\ge0$,
$\mu(x_m,y_m)=[\tpar]([e_{a_m}]b_{c_m})$.
\end{proposition}

\begin{proof}
Put $N=N_m$, $x=x_m$, $y=y_m$, and
$\delta=\ell(y)-\ell(x)$.
The permutations $x,y$ are minimal length representatives of
the cosets $xW_J,yW_J$ in $S_{3N}/W_J$, where
$W_J=S_N\times S_N\times S_N$ and $J$ omits the simple
reflections in positions $N,2N$.

Use Deodhar's parabolic polynomial \cite{Deodhar} in the
alternating convention of \cite[(1.6)]{FKK}:
\begin{equation}\label{eq:parabolic}
P^{J,+}_{x,y}(\qpar)
=\sum_{h\in W_J,\,xh\le y}(-1)^{\ell(h)}P_{xh,y}(\qpar).
\end{equation}
The FKK identification \cite[Theorem $2.5'$]{FKK} gives
\begin{equation}\label{eq:FKK}
[e_{a_m}]b_{c_m}
=\tpar^{\delta_m}P^{J,+}_{x_m,y_m}(\tpar^{-2}).
\end{equation}
To check the conventions, write FKK's quantum parameter as $z$
and their polynomial variable as $Q=z^{-2}$, and set $z=-\tpar^{-1}$.
We use their module with $u=Q$, which gives the alternating
sum in \autoref{eq:parabolic}.
The bar in their equation (1.4) replaces $Q$ by
$Q^{-1}=\tpar^{-2}$, while $(-z)^{-\delta}=\tpar^{\delta}$.
Their map (2.16) identifies the standard parabolic basis with
the standard tensors without an additional sign, and (2.17)
fixes the dual canonical normalization.
Here $A,B,C$ correspond to their
$\mathbf e_1,\mathbf e_2,\mathbf e_3$.
Their initial word $C^NB^NA^N$ gives our representatives
$\pi(w)$. The coproduct in \cite[Section 2]{KK} is FKK's
second coproduct (2.12), and their tensor bar $\psi'$ is our $\Phi$.

By \autoref{lem:lengths}, $\delta=17m+11$ is odd and at least $11$.
For $h\ne1$ appearing in \autoref{eq:parabolic}, minimality
gives $\ell(xh)=\ell(x)+\ell(h)$. Also $xh\ne y$, since $x,y$ are
distinct minimal representatives, so $xh<y$.
The ordinary KL degree bound, together with $\ell(h)\ge1$
and the oddness of $\delta$, therefore gives
\[
\deg P_{xh,y}\le
\left\lfloor\tfrac{\delta-\ell(h)-1}{2}\right\rfloor
\le\tfrac{\delta-3}{2}.
\]
Thus, only $h=1$ contributes to degree $(\delta-1)/2$ in
\autoref{eq:parabolic}, and this coefficient is $\mu(x,y)$.
Under \autoref{eq:FKK}, a term $\qpar^k$ becomes $\tpar^{\delta-2k}$.
Thus, the coefficient of $\tpar$ is precisely this critical coefficient.
\end{proof}

The remaining sections use the change of basis from webs to the
dual canonical basis to construct $b_{c_m}$ and compute $\mu(x_m,y_m)$.

\begin{example}
For the six-point example in \autoref{ex:six}, take
$a=ABCABC$, $c=AABBCC$. Then
$x=\pi(a)=(3,6,2,5,1,4)$, $y=\pi(c)=(5,6,3,4,1,2)$ and
$\ell(y)-\ell(x)=3$. Direct KL recursion gives $P_{x,y}(\qpar)=1+\qpar$
and $P^{J,+}_{x,y}(\qpar)=\qpar$. The same FKK dictionary gives
$[e_a]b_c=\tpar^3P^{J,+}_{x,y}(\tpar^{-2})=\tpar$, agreeing with
\autoref{eq:sixcoordinates}.
\end{example}

\begin{example}\label{ex:fkk-two}
For $m=2$, the permutations lie in $S_{30}$ and their length gap
is $45$. By \autoref{eq:raw-two-coordinate} below and
\autoref{thm:FKK}, $\mu(x_2,y_2)=2$.
More explicitly, \autoref{eq:FKK} gives
\[
P^{J,+}_{x_2,y_2}(\qpar)
=4\qpar^{16}+77\qpar^{17}+377\qpar^{18}+607\qpar^{19}
+355\qpar^{20}+70\qpar^{21}+2\qpar^{22}.
\]
Indeed, $\tpar^{k-45}=\qpar^{(45-k)/2}$ under $\qpar=\tpar^{-2}$,
so the term $2\tpar$
becomes $2\qpar^{22}$. The degree $22$ is the critical degree,
and the ordinary KL polynomial has the same coefficient $2$ there.
\end{example}

\section{From webs to the dual canonical basis}\label{sec:basischange}

We now explain a crucial change of basis.

\subsection{Three kinds of coordinates}

Fix a boundary with $3N$ positive signs and write $\mathcal D_N$ for
the dominant words of content $(N,N,N)$. A word $w$ in $\mathcal D_N$ labels
two invariant tensors: the non-elliptic (ne) web $W_w$ and the dual canonical (dc) vector $b_w$.
For any state word $s\in\{A,B,C\}^{3N}$, the state tensor (s) $e_s$
is a third object; it is usually not invariant. The word $s$ need not be dominant. 
For clarity, the three bases and their uses are
\begin{center}
\begin{tabular}{lll}
\toprule
Basis&Space&Use in the proof\\\midrule
$\{e_s\}$, all state words&$V^{\otimes3N}$&Tensor coefficients\\
$\{W_w:w\in\mathcal D_N\}$&$\Inv(V^{\otimes3N})$&Planar construction\\
$\{b_w:w\in\mathcal D_N\}$&$\Inv(V^{\otimes3N})$&FKK and KL coefficients\\\bottomrule
\end{tabular}
.
\end{center}
Order state strings lexicographically with $C\prec B\prec A$; write
$z\prec w$ when the state string for $z$ is smaller. This is the
reverse of lexicographic order on words with $A<B<C$. Then \cite[proof of Proposition $2.4'$]{FKK}, \cite[Corollary 4.5]{KashiwaraTanisaki},
\cite[Theorem 2]{KK}, 
and \cite[Corollary 3.7]{TubbenhauerCrystal} give
\begin{gather*}
\begin{aligned}
W_w&=e_w+\sum_{s\prec w}\mathcal M_{s,w}(\tpar)e_s,
&&\mathcal M_{s,w}(\tpar)\in\Z_{\geq 0}[\tpar,\tpar^{-1}],\text{(ne in s coordinates)}\\
b_w&=e_w+\sum_{s\prec w}\mathcal N_{s,w}(\tpar)e_s,
&&\mathcal N_{s,w}(\tpar)\in \tpar\Z_{\geq 0}[\tpar],\text{(dc in s coordinates)}\\
b_w&=W_w+\sum_{z\prec w}\mathcal K_{z,w}(\tpar)W_z,
&&\mathcal K_{z,w}(\tpar)\in\Z[\tpar,\tpar^{-1}],\text{(dc in ne coordinates)}.
\end{aligned}
\end{gather*}
In the first two sums $s$ ranges over all tensor state words;
in the last, $z$ ranges over dominant words.

\begin{lemma}\label{lem:bar-symmetric-coefficients}
Suppose
$b_w=\sum_z \mathcal K_{z,w}(\tpar)W_z$.
Then
$\mathcal K_{z,w}(\tpar)=\mathcal K_{z,w}(\tpar^{-1})$
for all $z,w$.
\end{lemma}

\begin{proof}
Both $b_w$ and every web $W_z$ are fixed by the tensor bar, cf. \autoref{lem:web-bar}
and \autoref{thm:canonical}. Applying the bar to the expansion above therefore gives
$b_w=\sum_z \mathcal K_{z,w}(\tpar^{-1})W_z$.
Since the $W_z$ form a basis, uniqueness of the expansion implies
$\mathcal K_{z,w}(\tpar)=\mathcal K_{z,w}(\tpar^{-1})$.
\end{proof}

Write $\mathcal M$ for the matrix from ne coordinates to s coordinates,
$\mathcal N$ for the matrix from dc coordinates to s,
and $\mathcal K$ for the change-of-basis matrix from dc to
ne. Thus $\mathcal N=\mathcal M\mathcal K$, with
\[
[e_a]b_w=[e_a]W_w+\sum_{z\prec w}\mathcal K_{z,w}(\tpar)[e_a]W_z.
\]
The coefficient used by FKK is an entry of $\mathcal N$, and the correction
from $W_w$ to $b_w$ is an entry of $\mathcal K$.

\subsection{What the correction removes}

The following is the usual intermediate-basis-calculation as e.g. in
\cite[Section 4.2]{LeclercToffin} or \cite{TubbenhauerCrystal}.

\begin{lemma}\label{lem:elimination}
Let $w\in\mathcal D_N$. Process the lower dominant labels
$z\prec w$ from largest to smallest. If the current coefficient at
$e_z$ has constant and negative part
$\alpha_0+\sum_{k>0}\alpha_{-k}\tpar^{-k}$, subtract
$(\alpha_0+\sum_{k>0}\alpha_{-k}(\tpar^{-k}+\tpar^k))b_z$.
Repeating this produces $b_w$ and stays in the invariant subspace.
\end{lemma}

\begin{proof}
The correction coefficient has the required bad part and is
bar invariant. The remaining coordinates of $b_z$ lie below $z$, so
elimination cannot spoil a coordinate already treated. To see that
dominant labels suffice, expand $W_w$ in the invariant basis
$\{b_z\mid z\in\mathcal D_N\}$. Its leading term is $b_w$, and the
lower coefficients are bar-invariant Laurent polynomials. At each
step the bad part recovers the full coefficient of the largest
remaining $b_z$, since $[e_z]b_w\in \tpar\Z[\tpar]$. Thus all lower summands
are removed, leaving $b_w$.
\end{proof}

The tensors remain invariant throughout, but the congruence in
\autoref{thm:canonical} concerns all tensor coordinates, including nondominant states. For instance, at a lower
dominant label $z$, a residual coefficient $\tpar^{-1}+3\tpar$ is corrected by
subtracting $(\tpar+\tpar^{-1})b_z$, leaving $2\tpar$ at $e_z$.

\begin{example}\label{ex:oneblockchange}
Consider $c_1=\word{ABAAAC}\,\word{BBBAACACC}\,\word{BBCBCC}$.
Put
\[
z_1=\word{ABABACCBC}(ABC)^4,\quad
z_2=\word{ABCAABCBC}(ABC)^4,\quad z_3=(ABC)^7.
\]
Each $W_{z_i}$ has no internal faces, cf.
\autoref{fig:oneblockchange}, and is therefore dual canonical
by the argument in \cite[proof of Proposition 3]{KK}. For $i=1$ it is a nine-point web
followed by four disjoint tripods; for $i=2$ it is a tripod followed by the
six-point web in \autoref{ex:tableau-to-web} and four further tripods;
for $i=3$ it is seven tripods.

The full nonpositive $\tpar$-degree part of $W_{c_1}$ is
\[
e_{c_1}+e_{z_1}+e_{z_2}+e_{z_3}.
\]
The corresponding part of $W_{z_i}$ is just $e_{z_i}$. Hence
\begin{equation}\label{eq:oneblockchange}
b_{c_1}=W_{c_1}-W_{z_1}-W_{z_2}-W_{z_3}.
\end{equation}
This is an actual web-to-dual-canonical change of basis. On the
four-dimensional span of these webs, its matrix is
\[
(b_{z_1},b_{z_2},b_{z_3},b_{c_1})
=(W_{z_1},W_{z_2},W_{z_3},W_{c_1})
\begin{pmatrix}1&0&0&-1\\0&1&0&-1\\0&0&1&-1\\0&0&0&1\end{pmatrix}.
\]
For $a_1=(ABC)^7$, direct tensor contraction gives
\[
\begin{aligned}
[e_{a_1}]W_{c_1}&=1+10\tpar^2+30\tpar^4+24\tpar^6+3\tpar^8,\\
[e_{a_1}](W_{z_1}+W_{z_2}+W_{z_3})&=1+2\tpar^2,\\
[e_{a_1}]b_{c_1}&=8\tpar^2+30\tpar^4+24\tpar^6+3\tpar^8.
\end{aligned}
\]
The correction removes the extra constant coefficient.
\end{example}

\begin{figure}[ht]
\centering
\includegraphics[width=.94\linewidth]{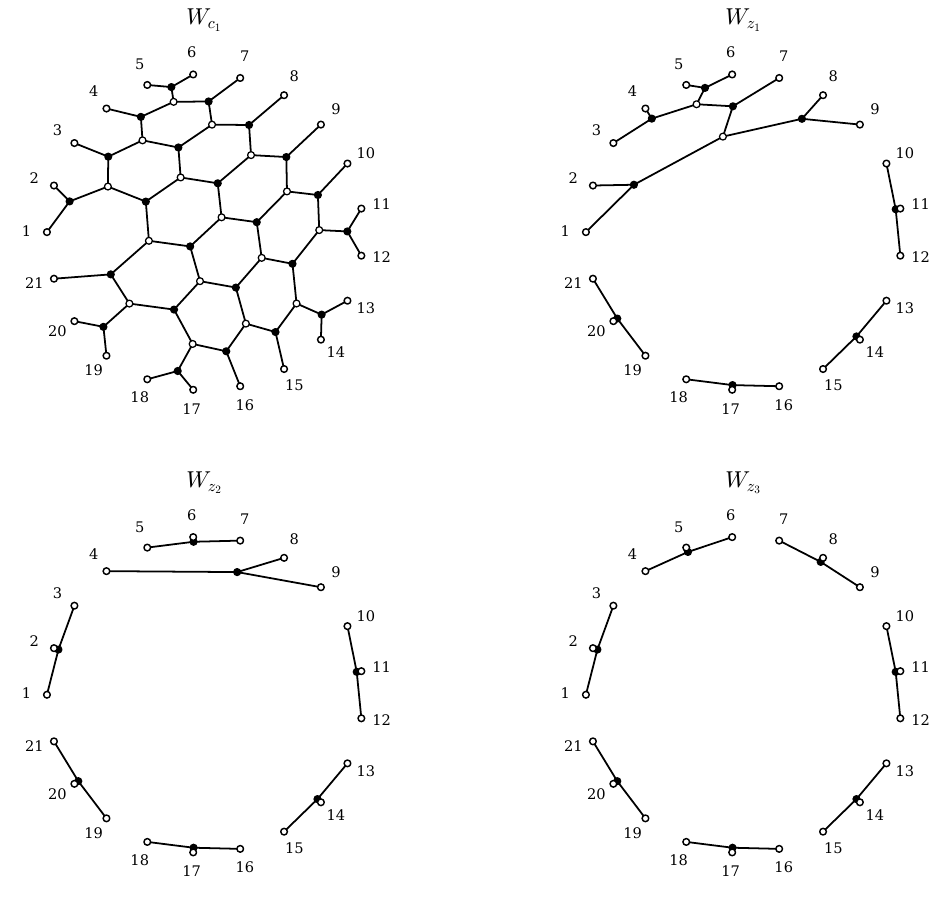}
\caption{The four webs in \autoref{eq:oneblockchange}. The lower webs
have isolated tripod components, which are retained at their actual
boundary positions. All four diagrams have $21$ positive boundary points.}
\label{fig:oneblockchange}
\end{figure}

\subsection{The two-block example: the actual lower webs}

For $c_2=\wordstart\wordmiddle^2\wordend$ the tensor coordinate $[e_{a_2}]W_{c_2}$ is already enough to
detect a failure of canonicality:
\begin{equation}\label{eq:raw-two-coordinate}
\begin{aligned}
[e_{a_2}]W_{c_2}
&=\tpar^{-1}+18\tpar+125\tpar^3+409\tpar^5+617\tpar^7+377\tpar^9+77\tpar^{11}+4\tpar^{13},\\
[e_{a_2}]b_{c_2}
&=2\tpar+70\tpar^3+355\tpar^5+607\tpar^7+377\tpar^9+77\tpar^{11}+4\tpar^{13}.
\end{aligned}
\end{equation}
The first line contains $\tpar^{-1}$ and therefore cannot be the
coordinate of a dual canonical vector with another leading state.
The correction contributes $\tpar^{-1}+16\tpar+55\tpar^3+54\tpar^5+10\tpar^7$.
Thus the required coefficient of $\tpar$ is $18-16=2$.
This is $\mu(x_2,y_2)$ by \autoref{thm:FKK}: for the length gap $45$,
the FKK normalization sends the critical power $\qpar^{22}$ to
$\tpar^{45-2\cdot22}=\tpar$.
The complete correction has $94$ lower webs. Split their indexing words
into sets $\mathcal R$ and $\mathcal F$ according to the multiplier:
\begin{equation}\label{eq:94correction}
b_{c_2}=W_{c_2}
-\sum_{z\in\mathcal R}W_z
-(\tpar+\tpar^{-1})\sum_{z\in\mathcal F}W_z,
\qquad |\mathcal R|=91,\quad |\mathcal F|=3.
\end{equation}
All words in these sets have ten occurrences of each letter and are
strictly below $c_2$ in the state order. The three words in
$\mathcal F$ are particularly simple:
\[
f_1=\word{ABABACCBC}(ABC)^7,\qquad
f_2=\word{ABCAABCBC}(ABC)^7,\qquad f_3=(ABC)^{10}.
\]
Thus $W_{f_1}$ and $W_{f_2}$ have the same nine-point cores as the lower
webs in \autoref{ex:oneblockchange}, with seven tripods appended.
The web $W_{f_3}$ consists of ten tripods. Three examples from
$\mathcal R$ are
\[
\begin{aligned}
r_1&=\word{ABAAACBBBABCACCCBC}(ABC)^4,\\
r_2&=\word{ABAAACBBBCABACCCBC}(ABC)^4,\\
r_3&=\word{ABAAACBBCBACABCCBC}(ABC)^4.
\end{aligned}
\]
Their pictures are in \autoref{fig:twoblocklower}; the indicated
coefficients are entries of $\mathcal K$, not tensor coordinates.
\begin{figure}[p]
\centering
\includegraphics[width=.90\linewidth,height=.82\textheight,keepaspectratio]{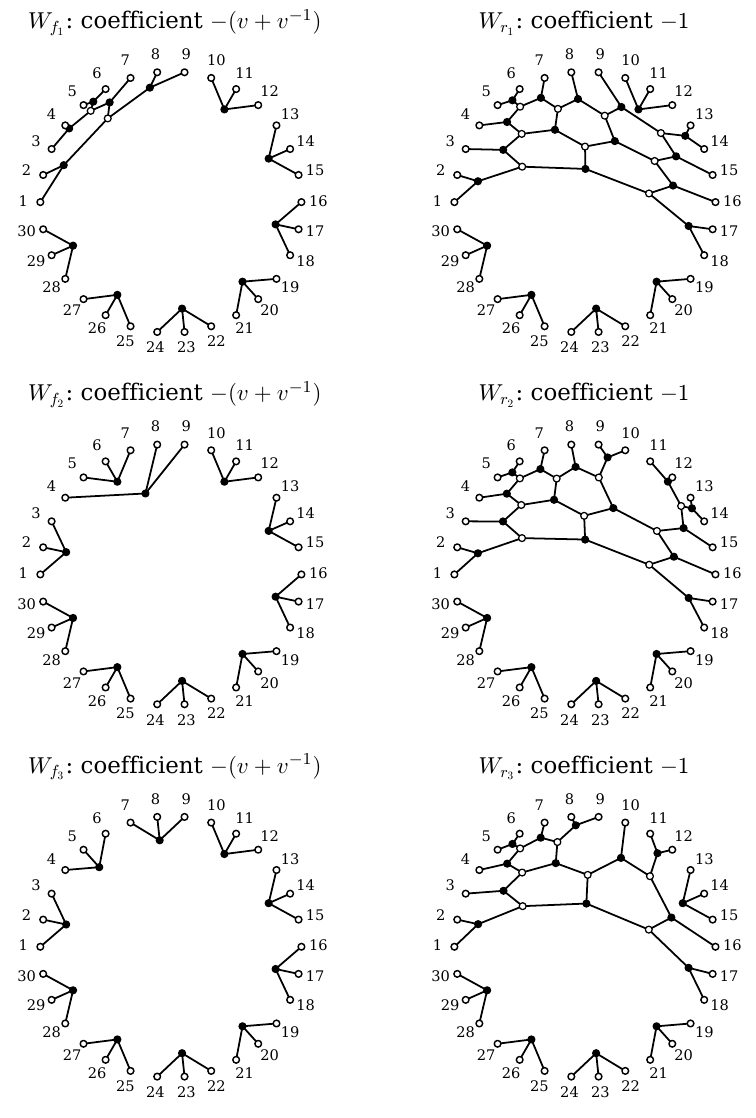}
\caption{Six lower webs in \autoref{eq:94correction}. The left column
contains all three Laurent corrections; the right column shows three
of the $91$ corrections with coefficient $-1$. Every diagram has
$30$ positive boundary points, numbered in the same order as $c_2$.
The first two left webs share the seven tripods at positions $10$ to $30$.}
\label{fig:twoblocklower}
\end{figure}

In this example every one of the $94$ lower webs is itself dual
canonical: $W_z=b_z$ for $z\in\mathcal R\cup\mathcal F$.
Consequently the inverse change of basis is just
\[
W_{c_2}=b_{c_2}+\sum_{z\in\mathcal R}b_z
+(\tpar+\tpar^{-1})\sum_{z\in\mathcal F}b_z.
\]
The raw web therefore has one desired canonical summand and $94$ lower
ones. Removing the latter gives $b_{c_2}$. These identities and the
canonicality of the lower webs are checked in all tensor coordinates
by the accompanying script.

Here is how the pictured terms affect the coordinate at $a_2=(ABC)^{10}$.
\begin{center}
\begin{tabular}{ccc}
\toprule
Web&Coefficient in $b_{c_2}$&Tensor coordinate $[e_{a_2}]W_z$\\\midrule
$W_{f_1}$&$-(\tpar+\tpar^{-1})$&$\tpar^2$\\
$W_{f_2}$&$-(\tpar+\tpar^{-1})$&$\tpar^2$\\
$W_{f_3}$&$-(\tpar+\tpar^{-1})$&$1$\\\midrule
$W_{r_1}$&$-1$&$\tpar+5\tpar^3+3\tpar^5$\\
$W_{r_2}$&$-1$&$\tpar^3+4\tpar^5+\tpar^7$\\
$W_{r_3}$&$-1$&$\tpar+3\tpar^3+\tpar^5$\\\bottomrule
\end{tabular}
\end{center}
For the three Laurent terms, multiplication by $\tpar+\tpar^{-1}$ gives
\[
(\tpar+\tpar^{-1})(\tpar^2+\tpar^2+1)=\tpar^{-1}+3\tpar+2\tpar^3.
\]
The remaining $91$ terms give
\[
[e_{a_2}]\sum_{z\in\mathcal R}W_z
=13\tpar+53\tpar^3+54\tpar^5+10\tpar^7.
\]
This accounts for every coefficient of the correction in
\autoref{eq:raw-two-coordinate}. In particular, the three Laurent terms
remove $3\tpar$, the other terms remove $13\tpar$, and $18-3-13=2$ remains.
The negative power is canceled by the $f_3$ term. Its paired positive
power is forced by bar invariance, so it also changes the coefficient
we are trying to compute.

One can also see why the final check must include nondominant states.
Take $s=\word{BAABACCBC}(ABC)^7$, which starts with $B$ and is therefore
not dominant. Direct contraction gives
\[
[e_s]W_{f_1}=\tpar,\qquad
[e_s]\bigl((\tpar+\tpar^{-1})W_{f_1}\bigr)=1+\tpar^2.
\]
The raw web has constant coefficient $1$ at $e_s$, and this correction
removes it. A calculation confined to the dominant indexing words
would miss this coordinate. The raw web has $156$ nonzero coordinates
in nonpositive $\tpar$-degree, of which $58$ are nondominant. After the
full correction, only the leading coordinate $e_{c_2}$ remains in
those degrees.

\begin{remark}
The calculation above can be checked by hand, but arguably this is
neither pleasant nor a good use of time.
For a computer check, \texttt{M2Data.lean} lists all $94$ lower words,
their multipliers and their coordinates at $a_2$, and
\texttt{M2Flow.lean} provides the tensor contraction routine.
Both files are in \repo{code/Lean/WebMu/Revised/}
in the repository \cite{GoLaTuVa-mu-code}.
\end{remark}

\subsection{From one column to the uniform construction}

Listing such a column separately for every $m$ would not give a
uniform proof. Instead, \autoref{sec:construction} builds a web
combination $\beta_m$ from five fixed maps.
Here $\middlemap$ is the block map and $\createmap,\removemap$ are correction maps.
Identity maps on external strands are understood in these compositions.
For two blocks it is
\[
\beta_2=\startmap\middlemap^2\finishmap-\startmap\createmap\removemap\finishmap-\startmap\middlemap\createmap\tau^2+d\startmap\createmap\tau^5=(\startmap,0)\cdot\begin{pmatrix}\middlemap&-\createmap\\\removemap&-d\tau^3\end{pmatrix}^2\cdot\binom{\finishmap}{\tau^2},
\]
where $\startmap\middlemap^2\finishmap=W_{c_2}$ and $d=\tpar+\tpar^{-1}$. Similarly, we have
\[
\beta_1=\startmap\middlemap\finishmap-\startmap\createmap\tau^2=(\startmap,0)\cdot\begin{pmatrix}\middlemap&-\createmap\\\removemap&-d\tau^3\end{pmatrix}\cdot\binom{\finishmap}{\tau^2}.
\]
Expanding the three correction
terms of $\beta_2$ in the non-elliptic web basis gives the column above. Grouping
them by the three-strand interface produces a formula that extends
to every $m$. The local degree argument in \autoref{sec:canonicality}
proves that this construction indeed gives $\beta_m=b_{c_m}$.

\subsection{The categorical interpretation}

For readers who prefer categorification, a web defines a projective
module over the web algebra \cite{MPT}. Such a projective can split.
The dual canonical basis corresponds instead to normalized
indecomposable projectives \cite[Corollary 7.9]{MackaayDual}. The
cyclotomic Hecke/KLR comparison and its canonical-basis-consequences
are developed in \cite{BrundanKleshchevBlocks,BrundanKleshchev}.
Robert's red graphs detect certain splittings geometrically
\cite{RobertFamily,RobertCharacterization}. This explains why reducing
a web can reveal correction terms. The tensor calculation still has
to establish the coefficients and the leading state of the desired
canonical vector. In our proof these are supplied by the fixed maps
and their local degree bounds.

\section{A corrected tensor on a fixed interface}\label{sec:construction}

We construct $\beta_m$ from fixed prefix (start), block (middle) and tail (end) maps, with corrections at each block. We will prove $\beta_m=b_{c_m}$ below.
The raw web has three connecting strands, with $3^3=27$
tensor states. Some correction terms have no connecting strands across
the cut; a scalar coordinate records this case. The connecting space
therefore has dimension
$27+1=28$, independently of $m$. This fixed dimension makes the computation possible: after
fixing the output state, the iteration reduces to calculations with
$28\times28$ matrices.

\subsection{Growth strings and local maps}\label{sec:localmaps}

The following strings specify the maps completely. They use the
KK growth operations \cite[Section 5]{KK}.
Recall that $c_m=\wordstart\wordmiddle^m\wordend$, with a six-letter prefix $\wordstart$, a nine-letter
block $\wordmiddle$, and a six-letter tail $\wordend$.
A downward growth string reduces the output boundary to the input
boundary. Read it from left to right, renumbering the remaining
strands from zero after each step. The signs $+,-$ specify
$V^\pm$, independently of the state labels $A,B,C$. The operation $Y_i$ merges the equal-sign
strands at positions $i,i+1$; $H_i$ applies the two-vertex web $H$ to
opposite signs; and $a_i$ removes an adjacent opposite-sign pair.
The remaining signs are the input signs of the map.

To construct the map upward, read the string in reverse, replacing
$Y_i$ by a split $y$ and $a_i$ by a cup $\sigma'$.
The operation $H_i$ uses the planar web $H$ in
\autoref{fig:generators}; identity maps act on the other strands.
This fixes the graph and its cyclic orders.

To compute a covector from a fixed output state, read downward
using the ordinary coordinate transposes of the upward maps.

The correction tables also use maps on three positive input strands.
Put
\[
\lambda_1=\lambda\otimes I,
\qquad \lambda_2=I\otimes\lambda,
\qquad E_1=(y\lambda)\otimes I,
\qquad E_2=I\otimes(y\lambda).
\]
Thus $\lambda_1$ acts on the first two strands and $\lambda_2$ on the
last two; the subscripts here are numbered from one. Products of maps
are composed from right to left, so $\lambda_2E_1$ first applies $E_1$.
This convention for composition is distinct from the left-to-right
reading of a downward growth string.

The prefix $\startmap$, block $\middlemap$ and tail $\finishmap$ correspond to
$\wordstart,\wordmiddle,\wordend$, respectively. Together with the correction
maps $\createmap,\removemap$, they have types
\[
\begin{array}{ll}
\startmap:V^{\otimes3}\to V^{\otimes6},&
\middlemap:V^{\otimes3}\to V^{\otimes3}\otimes V^{\otimes9},\\
\finishmap:\one\to V^{\otimes3}\otimes V^{\otimes6},&
\createmap:\one\to V^{\otimes3}\otimes V^{\otimes9},\\
\removemap:V^{\otimes3}\to V^{\otimes9}.&
\end{array}
\]
For the block and tail, we have bent three negative inputs
into the first three positive outputs, preserving the boundary order.
Before bending, these pieces have types
$(V^-)^{\otimes3}\otimes V^{\otimes3}\to V^{\otimes9}$ and
$(V^-)^{\otimes3}\to V^{\otimes6}$, respectively.
The first three outputs are the connecting strands; the remaining
outputs carry $\wordmiddle$ or $\wordend$.
The map $\createmap$ creates the three connecting strands, and
$\removemap$ removes them.
Start with $6$, $12$ and $9$ positive output strands for the
prefix, block and tail, respectively. The following strings
define $\startmap,\middlemap,\finishmap$; see \autoref{fig:pieces}.
\begin{equation}\label{eq:growth}
\begin{aligned}
\startmap &:Y_0Y_3H_2H_1Y_0,\\
\middlemap &:Y_2Y_6Y_7H_2Y_1H_2H_1Y_0H_3Y_2H_3Y_4H_2Y_1H_2H_1Y_0,\\
\finishmap &:Y_2Y_3Y_2Y_1Y_2Y_1Y_0a_0.
\end{aligned}
\end{equation}
With one-based numbering, the prefix string is precisely
$Y_1Y_4H_3H_2Y_1$. In particular, its first operation merges the first
two positive strands, whose leading state labels are $A,B$.

\begin{figure}[ht]
\centering
\begin{tikzpicture}
\node[inner sep=0,anchor=south west] (pieces) at (0,0)
  {\includegraphics[width=\linewidth]{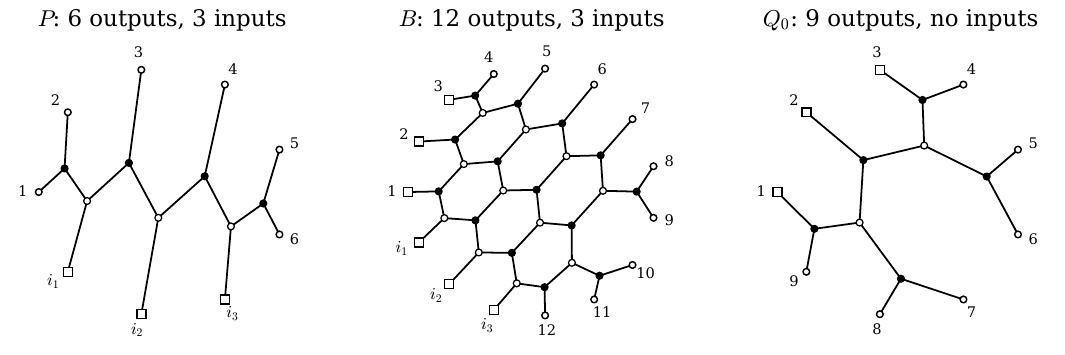}};
\begin{scope}[x={(pieces.south east)},y={(pieces.north west)}]
\fill[white] (0,.90) rectangle (1,1);
\node[anchor=north,inner sep=0,font=\normalsize] at (.1524,.989)
  {$\startmap$: 6 outputs, 3 inputs};
\node[anchor=north,inner sep=0,font=\normalsize] at (.4996,.989)
  {$\middlemap$: 12 outputs, 3 inputs};
\node[anchor=north,inner sep=0,font=\normalsize] at (.8462,.989)
  {$\finishmap$: 9 outputs, no inputs};
\end{scope}
\end{tikzpicture}
\caption{The three pieces of the raw web. The inputs $i_1,i_2,i_3$
of $\startmap$ and $\middlemap$ are joined to outputs $1,2,3$ of the next piece.
Squares mark the three-strand interfaces.
All labels refer to the local piece, before gluing.}
\label{fig:pieces}
\end{figure}
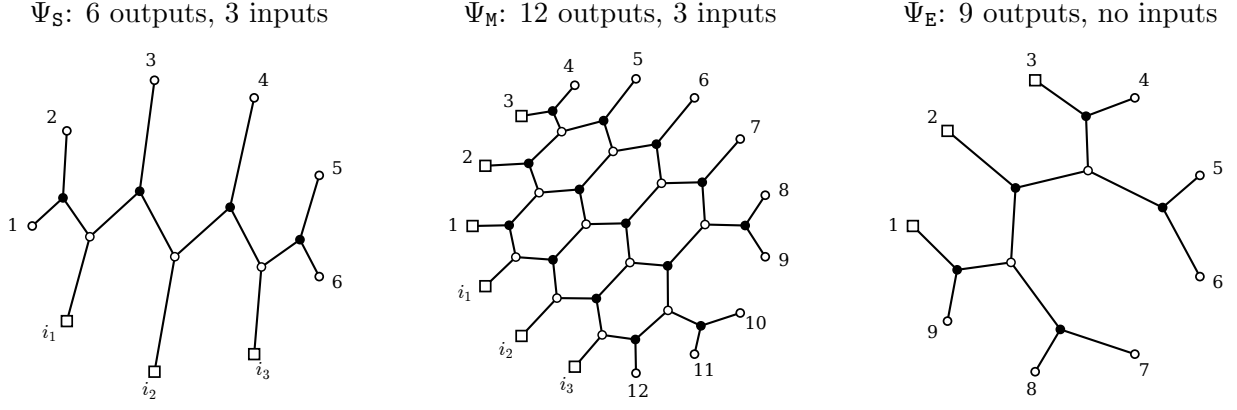

These strings reduce the leading states as
$\wordstart\to AAA$, $AAA\,\wordmiddle\to AAA$, and $AAA\,\wordend\to\varnothing$.
For the tail, the first three merges give
$AAA\,\wordend\to AAABCBCC\to AAACBCC\to AABBCC$.
The remaining steps reduce $AABBCC$ to the empty boundary using only
$Y$ and $a$; no $H$ is needed.
For the longer strings, abbreviate
$\mathsf R_i=H_{i+1}Y_i$ and $\mathsf S=H_2H_1Y_0$.

To define $\createmap$, let $O$ be the map with twelve positive outputs and
downward string $Y_2Y_6Y_7\mathsf R_1$. Its eight surviving input
signs are $+--++--+$. Pair these inputs consecutively by four cups,
and denote the resulting tensor by $\kappa_8$. Then $\createmap=O\kappa_8$.
The resulting web is shown in \autoref{fig:correction}.
It gives the correction in \autoref{ex:oneblockchange}:
\[
\startmap\createmap\tau^2=W_{z_1}+W_{z_2}+W_{z_3}.
\]

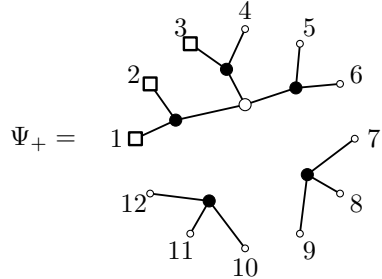
\begin{figure}[ht]
\centering
\begin{tikzpicture}[x=1cm,y=1cm,font=\small]
\node[anchor=east] at (-2.1,0) {$\createmap=$};
\foreach \i in {1,...,12} {
  \pgfmathsetmacro{\angle}{210-30*\i}
  \coordinate (b\i) at (\angle:1.45);
  \node at (\angle:1.70) {$\i$};
}
\coordinate (L) at (165:.95);
\coordinate (U) at (105:.95);
\coordinate (R) at (45:.95);
\coordinate (O) at (0,.45);
\draw[web] (b1)--(L)--(b2) (b3)--(U)--(b4)
  (b5)--(R)--(b6) (L)--(O)--(U) (O)--(R);
\foreach \v in {L,U,R}
  \node[sink] at (\v) {};
\node[source] at (O) {};
\coordinate (T1) at (-30:.95);
\coordinate (T2) at (-120:.95);
\draw[web] (b7)--(T1)--(b8) (T1)--(b9)
  (b10)--(T2)--(b11) (T2)--(b12);
\foreach \v in {T1,T2}
  \node[sink] at (\v) {};
\foreach \i in {1,2,3}
  \node[rectangle,draw=black,fill=white,line width=.9pt,inner sep=2.4pt] at (b\i) {};
\foreach \i in {4,...,12}
  \node[circle,draw=black,fill=white,inner sep=.9pt] at (b\i) {};
\end{tikzpicture}
\caption{The correction map
$\createmap=W_{ABACBC}\otimes\tau^2$.
Squares mark outputs $1,2,3$, which are reserved for the next attachment.}
\label{fig:correction}
\end{figure}

The second correction is $\removemap=\sum_{i=1}^8D_i$. Each term is
$D_i=P_i h_i$, where $P_i$ has nine positive outputs and the downward
string in the table. The last column gives the input signs of $P_i$;
they agree with the output signs of the small map $h_i$.
\begin{center}\small
\begin{tabular}{llll}
\toprule
Map&Downward growth of $P_i$&$h_i$&Input signs of $P_i$\\\midrule
$D_1$&$Y_0Y_3Y_4a_4\mathsf S$&$I$&$+++$\\
$D_2$&$Y_0Y_3Y_4Y_3Y_2\mathsf R_0$&$I$&$+++$\\
$D_3$&$Y_0Y_2a_2Y_2\mathsf R_0Y_1$&$\lambda_2$&$+-$\\
$D_4$&$Y_0Y_1Y_2Y_1Y_2\mathsf R_0Y_1$&$\lambda_2E_1$&$+-$\\
$D_5$&$Y_0Y_1Y_0Y_1Y_2Y_1Y_1$&$\lambda_2$&$+-$\\
$D_6$&$Y_0a_0Y_1Y_2Y_1Y_1$&$\lambda_2E_1$&$+-$\\
$D_7$&$Y_0a_0Y_0a_0Y_0$&$\lambda_1$&$-+$\\
$D_8$&$Y_0a_0Y_0a_0Y_0a_0$&$\eta$&$\varnothing$\\\bottomrule
\end{tabular}
\end{center}
All eight terms have coefficient $1$. The signs of the correction
are specified in \autoref{eq:maps}. See also \autoref{fig:atlas-term4} for the corresponding webs.

\subsection{Gluing the basic pieces}

Start with $\finishmap$, leaving its last six outputs external. At each
of the $m$ steps, join the first three outputs to the inputs of
$\middlemap$, in order. Keep the block's first three outputs for the next
attachment and leave its other nine outputs external. Finally
attach $\startmap$, whose six outputs complete the prefix.

\begin{lemma}
For every $m\ge0$, this gluing gives
\[
W_{c_m}=\startmap\middlemap^m\finishmap,\qquad
6+9m+6=12+9m\text{ external outputs}.
\]
\end{lemma}

\begin{proof}
The strings in \autoref{eq:growth} are admissible growth
steps on $\wordstart$, $AAA\,\wordmiddle$ and $AAA\,\wordend$. Applying them
successively from the prefix to the tail reduces $c_m$ to the empty
boundary. The growth algorithm therefore identifies the assembled
web with $W_{c_m}$ \cite[Section 5, Lemmas 1--3]{KK}.
\end{proof}

The main proof identifies $\beta_m$ directly by its tensor congruence
and does not depend on this raw web identification. 
In the notation $\startmap\middlemap^m\finishmap$, each new block acts only on
the three reserved strands; identity maps on the external strands
are understood.

\begin{example}\label{ex:assembly}
Start with $\finishmap$, attach one copy of $\middlemap$, attach a second copy, and
close the three reserved outputs with $\startmap$; see \autoref{fig:assembly}.
The blocks are attached from the tail towards the prefix.
The intermediate diagrams have $9,18,27$ outputs, of which three
are still reserved. The final diagram has $30$ external outputs.
Their order is
\[
\underbrace{1,\ldots,6}_{\wordstart}\quad
\underbrace{7,\ldots,15}_{\wordmiddle}\quad
\underbrace{16,\ldots,24}_{\wordmiddle}\quad
\underbrace{25,\ldots,30}_{\wordend}.
\]
This constructs $W_{c_2}$. The tensor giving
$\mu=2$ is obtained by adding the correction in
\autoref{def:corrected}; the raw web still has the forbidden
$\tpar^{-1}$ coordinate in \autoref{eq:raw-two-coordinate}.
\end{example}

\begin{figure}[ht]
\centering
\begin{tikzpicture}
\node[inner sep=0,anchor=south west] (assembly) at (0,0)
  {\includegraphics[width=.96\linewidth]{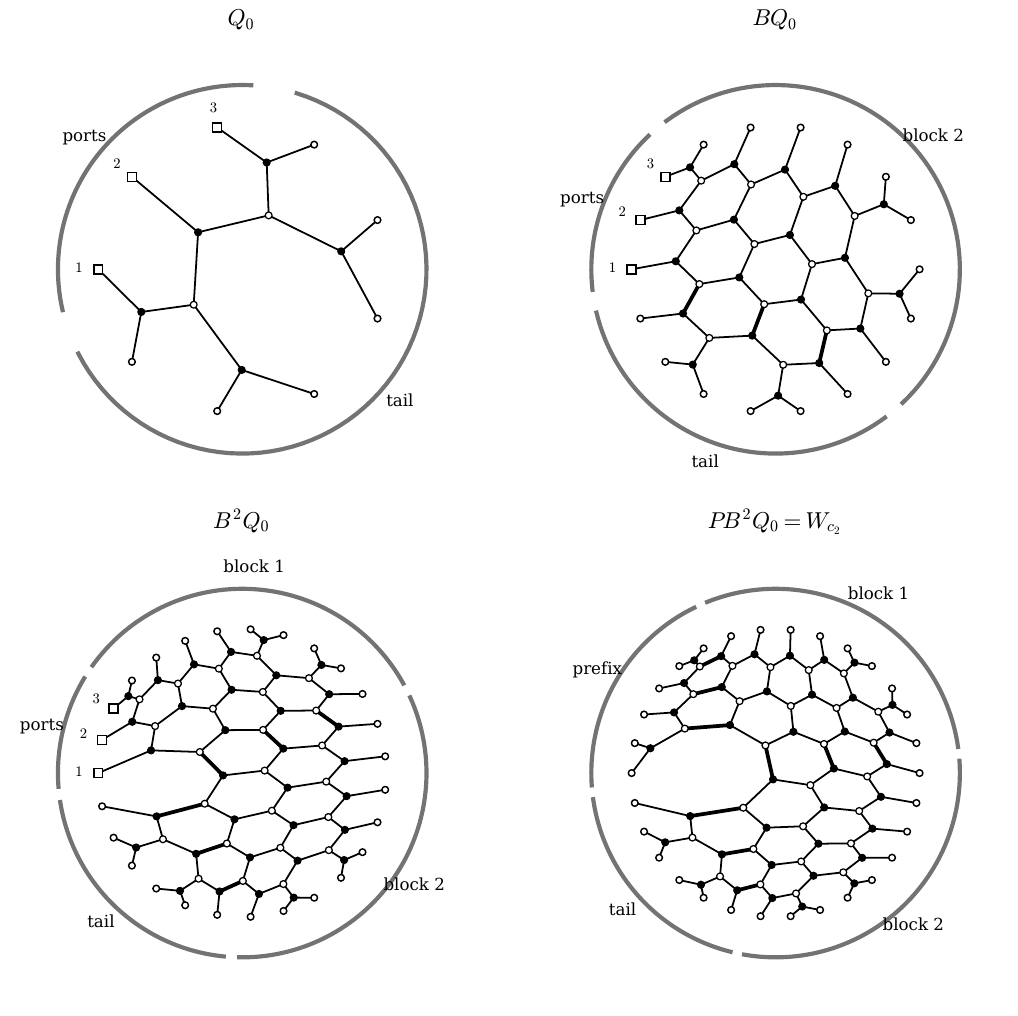}};
\begin{scope}[x={(assembly.south east)},y={(assembly.north west)}]
\fill[white] (0,.963) rectangle (1,1);
\fill[white] (0,.465) rectangle (1,.51);
\node[anchor=north,inner sep=0,font=\normalsize] at (.2365,.994)
  {$\finishmap$};
\node[anchor=north,inner sep=0,font=\normalsize] at (.7609,.994)
  {$\middlemap\finishmap$};
\node[anchor=north,inner sep=0,font=\normalsize] at (.2365,.502)
  {$\middlemap^2\finishmap$};
\node[anchor=north,inner sep=0,font=\normalsize] at (.7609,.502)
  {$\startmap\middlemap^2\finishmap=W_{c_2}$};
\end{scope}
\end{tikzpicture}
\caption{The assembly $\finishmap\longmapsto \middlemap\finishmap\longmapsto \middlemap^2\finishmap
\longmapsto \startmap\middlemap^2\finishmap$. The boundary arcs group the external outputs.
Thicker edges are the three-strand joins; they are ordinary web edges.
Square boundary markers denote the three outputs still to be joined.}
\label{fig:assembly}
\end{figure}

\subsection{Iteration}

The iteration is shown in \autoref{fig:interface}.

\begin{definition}\label{def:corrected}
Let $\mathcal E=V^{\otimes3}\oplus\one$. Its standard basis consists
of the $27$ three-letter states and a scalar state $e_*=(0,1)$.
With the maps just defined, put
\begin{equation}\label{eq:maps}
\pp=(\startmap,0),\qquad
T=\begin{pmatrix}\middlemap&-\createmap\\\removemap&-d\tau^3\end{pmatrix},\qquad
\tailmap=\binom{\finishmap}{\tau^2}.
\end{equation}
Their types are
\[
\pp:\mathcal E\to V^{\otimes6},\quad
T:\mathcal E\to\mathcal E\otimes V^{\otimes9},\quad
\tailmap:\one\to\mathcal E\otimes V^{\otimes6}.
\]
Define $\tailmap_0=\tailmap$ and, with $I_j$ the identity on $V^{\otimes j}$,
\begin{equation}\label{eq:iterate}
\tailmap_{k+1}=(T\otimes I_{9k+6})\tailmap_k\quad(k\ge0),
\qquad \beta_m=(\pp\otimes I_{9m+6})\tailmap_m.
\end{equation}
Thus $\beta_m:\one\to V^{\otimes(12+9m)}$ is an invariant tensor.
\end{definition}

\begin{figure}[ht]
\centering
\begin{tikzpicture}[x=1cm,y=1cm,font=\small,every node/.append style={text=black}]
\draw[black] (0,0) rectangle (1.1,1.0) node[pos=.5] {$\pp$};
\draw[black] (2.2,0) rectangle (3.3,1.0) node[pos=.5] {$T$};
\draw[black] (4.4,0) rectangle (5.5,1.0) node[pos=.5] {$T$};
\draw[black] (6.6,0) rectangle (7.7,1.0) node[pos=.5] {$\tailmap$};
\draw[black,<-] (1.1,.5)--(2.2,.5);
\draw[black,<-] (3.3,.5)--(4.4,.5);
\draw[black,<-] (5.5,.5)--(6.6,.5);
\node at (1.65,.8) {$\mathcal E$};\node at (3.85,.8) {$\cdots$};
\node at (6.05,.8) {$\mathcal E$};
\draw[black] (.55,0)--(.55,-.65) (2.75,0)--(2.75,-.65)
(4.95,0)--(4.95,-.65) (7.15,0)--(7.15,-.65);
\node[below] at (.55,-.65) {$V^{\otimes6}$};
\node[below] at (2.75,-.65) {$V^{\otimes9}$};
\node[below] at (4.95,-.65) {$V^{\otimes9}$};
\node[below] at (7.15,-.65) {$V^{\otimes6}$};
\end{tikzpicture}
\par\medskip
\begin{tabular}{c|cc}
$T$&input: three strands&input: no strands\\\hline
output: three strands&$\middlemap$&$-\createmap$\\
output: no strands&$\removemap$&$-d\tau^3$
\end{tabular}
\caption{Composition through the $28$-dimensional connecting
space $\mathcal E$. Attachments run from the tail on the right to the
prefix on the left; external outputs are ordered from left to right.
The table shows the four transitions in $T$. Each transition creates
nine external outputs. The no-strand case is the scalar summand
$\one$, which adds no boundary point.}
\label{fig:interface}
\end{figure}
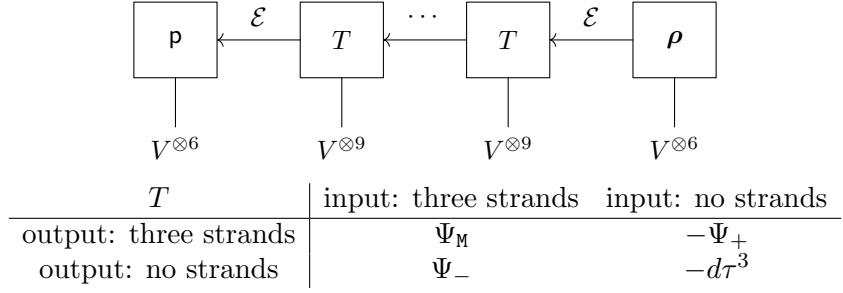

The map $\removemap$ takes the three connecting strands to nine external
outputs, reserving no strands for the next attachment. The map $-\createmap$
starts with no connecting strands and creates three reserved and nine
external outputs. The map $-d\tau^3$ keeps the no-strand state and adds
nine external outputs. Every term of the final tensor still has
the same $12+9m$ positive boundary points.
The maps $\pp$ and $\tailmap$ supply the prefix and tail.

We abbreviate the composition in \autoref{eq:iterate} by
$\pp T^m\tailmap$, with identity maps on previously created external
strands understood. Its output order is the six-strand prefix,
$m$ nine-strand blocks, and the six-strand tail. All coordinates
are integral Laurent polynomials.

We next prove that these correction maps remove all unwanted constant
and negative powers of $\tpar$, so that $\beta_m=b_{c_m}$.

\section{Canonicality from local degree bounds}\label{sec:canonicality}

By \autoref{thm:canonical}, we must show that $\beta_m$ is bar invariant
and congruent to $e_{c_m}$ modulo $\tpar$ in the standard tensor lattice.
The main work is to control the degrees of all tensor coordinates
at once.

We change coordinates on $\mathcal E$ to replace the prefix $\pp$,
block $T$ and tail $\tailmap$ by $\widehat p,\widehat T,\widehat\rho$,
defined in \autoref{loc:gauge}. The final tensor is unchanged:
{
\[
\widehat p\,\widehat T^{\,m}\widehat\rho=\pp T^m\tailmap=\beta_m.
\]
}
We will show that these transformed maps have coefficients in $R=\Z[\tpar]$.
In degree zero, $\widehat\rho$ gives $AAA\otimes e_{\wordend}$,
each $\widehat T$ preserves $AAA$ and adds $\wordmiddle$, and
$\widehat p$ replaces $AAA$ by $\wordstart$. Composing these terms gives
$e_{c_m}$ for every $m$.

\subsection{The local paths}

We begin with the degree estimates for the small pieces of a block.
For the local bounds here and in \autoref{sec:readout}, degree means
degree in $\vpar$. In particular,
\[
\deg_{\vpar}(\tpar^k)=-k,\qquad \deg_{\vpar}0=-\infty.
\]
An integral Laurent polynomial lies in $R$ exactly when its degree
is at most zero; it lies in $\tpar R$ when that degree is strictly negative.
For a map with entries in $R$, write $M[0]$ for its matrix of constant coefficients.

Words in the local tables abbreviate standard tensors. We also write
$|s\rangle=e_s$ and use $\langle s|$ for the corresponding coordinate
functional, and $[a<b]=\delta_{a<b}$ for indicators.

Keep $\mathsf R_i=H_{i+1}Y_i$ and $\mathsf S=H_2H_1Y_0$ from
\autoref{sec:localmaps}, and put $\mathsf L_i=H_iY_{i+1}$.
We also use these growth abbreviations for their upward tensor maps;
then $Y_i$ is a split and $a_i$ is a cup. A path 
chooses one monomial from each local coefficient; its
degree is the sum of their degrees. The elementary tables show that
all coefficients of an individual web are nonnegative and that the
only branch of positive $\vpar$-degree of $H$ is $AC\mapsto \tpar^{-1} AC$.
Thus the degree of a nonzero web coordinate is the largest
degree of its compatible paths.

\begin{lemma}\label{loc:nonpositive}
Every path through $y$, a cup, $\mathsf R_i$, or $\mathsf L_i$ has degree at most
zero. In $\mathsf S$, the only path of positive $\vpar$-degree is
\[
 ABC\xrightarrow{Y_0,0}ABBC
 \xrightarrow{H_1,0}ACAC\xrightarrow{H_2,1}ACAC.
\]
Its degree and multiplicity are one. Furthermore
\begin{equation}\label{loc:Srow}
 \langle ACAC|\mathsf S=\tpar^{-1}\langle ABC|+\langle ACB|+\langle BAC|,
 \qquad \mathsf S(AAA)=AAAB+\tpar BAAA.
\end{equation}
\end{lemma}

\begin{proof}
The assertions for $y$ and a cup follow from their tables. In
$\mathsf R$, a degree-one branch of $H$ requires the right output of
the preceding $y$ to be $A$. That branch of $y$ has degree $-1$, so
the total is at most zero. In $\mathsf L$, the same argument uses the
left output $C$ of $y$.

A path of positive degree through $\mathsf S$ must use the degree-one term
of the last $H$ and a degree-zero path through the preceding
$\mathsf R$. Thus the first $H$ must output $A$ on its right.
The possibilities are $AA\to AA$, $AB\to BA$, $BB\to CA$ and
$CA\to CA$. The first two require a preceding split of degree $-1$;
the last already has degree $-1$. Hence only $BB\to CA$ remains.
Its preceding split is $A\to AB$, and the last $H$ forces the third
input to be $C$. This is the displayed path.

For output $ACAC$, the complete three paths are
\[
\begin{array}{c|ccc|r}
 &\text{after }Y_0&\text{after }H_1&\text{after }H_2&\deg_{\vpar}\\\hline
ABC&ABBC&ACAC&ACAC&1\\
ACB&ABCB&ACBB&ACAC&0\\
BAC&ACAC&ACAC&ACAC&0
\end{array}.
\]
They exhaust the choices in the $H$ table. The formula at $AAA$
follows by applying that table to the two terms of $y(A)$.
\end{proof}

In the notation above, the block in \autoref{eq:growth} has downward string
$\middlemap:Y_2Y_6Y_7\mathsf R_1\mathsf S
\mathsf R_2\mathsf L_3\mathsf R_1\mathsf S$.
The only possible degree increases in $\middlemap$ therefore come from its two
occurrences of $\mathsf S$. To study the part between them, put
\[
X_4=V^-\otimes V^+\otimes V^+\otimes V^-,\qquad
X_7=V^{---+--+},\qquad X_8=V^{+--++--+}.
\]
Here $V^S$ denotes the signed tensor product defined in
\autoref{sec:webs}.
Let $M_{\mathrm{loc}}:X_4\to X_7$ have downward string
$\mathsf R_2\mathsf L_3\mathsf R_1$, and set
\[
N_{\mathrm{loc}}=(\mathsf S\otimes I_{V^{+--+}})
M_{\mathrm{loc}}\mathsf S:V^{\otimes3}\to X_8.
\]
The rightmost $\mathsf S$ has positive inputs and the leftmost has
negative inputs. The map $O:X_8\to V^{\otimes12}$ already defined in
\autoref{sec:localmaps} has string $Y_2Y_6Y_7\mathsf R_1$.
Expanding the growth strings gives
$\middlemap=ON_{\mathrm{loc}}$ and $\createmap=O\kappa_8$. This factorization lets us locate
the term canceled by the correction $\createmap$.

\subsection{An integral coordinate system}\label{loc:gauge}

The block $T$ has some positive $\vpar$-degrees, so reducing it
modulo $\tpar$ is not yet possible. We change coordinates on $\mathcal E$
so that the prefix $\pp$, block $T$ and tail $\tailmap$ become maps
$\widehat p,\widehat T,\widehat\rho$ with entries in $R$.

For a three-letter state $s$, define
\[
h_s=[s\in\{ABC,ACB,BAC\}],\quad
\Delta=\operatorname{diag}(\tpar^{-h_s}),\quad
\mathsf J=\begin{pmatrix}I&0\\-\tpar\eta&1\end{pmatrix},\quad
G=\operatorname{diag}(\Delta,-1)\mathsf J.
\]
Put $\widehat T=(G\otimes I_9)TG^{-1}$,
$\widehat p=\pp G^{-1}$, and
$\widehat\rho=(G\otimes I_6)\tailmap$.
With the identity maps on exterior strands understood, adjacent factors
$G^{-1}G$ cancel in $\widehat p\,\widehat T^{\,m}\widehat\rho$, giving $\beta_m$ as above.
We use $G$ only to estimate coordinates.

Here is the reason for the choice of $G$. The map $N_{\mathrm{loc}}$
has one degree-two term, at input $ABC$ and output $(AC)^4$.
The tensor $\kappa_8$ has coefficient $1$ at $(AC)^4$, and
$\eta(ABC)=\tpar^{-3}$. Hence $\tpar\kappa_8\eta$ has exactly the same degree-two
term. The triangular matrix $\mathsf J$ introduces this cancellation, while
$\Delta$ rescales the three states $ABC,ACB,BAC$ to compensate for
the remaining positive degrees. Negating the scalar 
coordinate makes the two constant transitions in
\autoref{sec:readout} positive.

\begin{lemma}\label{loc:outer}
Every coefficient of $(\Delta\otimes I_9)O$ belongs to $R$.
\end{lemma}

\begin{proof}
Only the first three output letters determine the shift $h_s$.
The two splits at positions six and seven have nonpositive degree
and do not affect those letters.
The relevant local piece is $\mathsf R$ followed by $y$ on its second
output, with an unchanged letter $a$ on the left. Write the first three
final letters as $apr$. If this is $ABC$ or $BAC$, then $r=C$ is the
left output of $y$, forcing degree $-1$ at that $y$;
the preceding $\mathsf R$ has degree at most zero by
\autoref{loc:nonpositive}. If it is $ACB$,
then $p=C$ forces degree $-1$ at the first $y$ inside $\mathsf R$.
An $H$ of degree one could cancel this only via $AC\to AC$;
the last $y$ then has input $A$, and producing $r=B$ forces its
degree $-1$ branch. Thus every output with $h_s=1$ has degree at most $-1$ before
rescaling. For outputs with $h_s=0$, \autoref{loc:nonpositive}
applies directly.
\end{proof}

\begin{lemma}\label{loc:offender}
The only monomial of $N_{\mathrm{loc}}$ with degree exceeding $h_s$ at input $s$
is
\[
 \tpar^{-2}\,|ACACACAC\rangle\langle ABC|,
\]
with coefficient one.
\end{lemma}

\begin{proof}
Work backwards along a degree-zero path through $M_{\mathrm{loc}}$
whose output starts with $ABC$. This forces
\begin{equation}\label{loc:middlepath}
ACAC\xrightarrow{\mathsf R_1}ABBBC
\xrightarrow{\mathsf L_3}ABBCAC
\xrightarrow{\mathsf R_2}ABCACAC.
\end{equation}
Indeed, the last $\mathsf R$ must have its first output $C$. Its $y$ has
degree $-1$ and its $H$ must be $AC\to \tpar^{-1} AC$, forcing input $BC$
and output $CAC$. The preceding $\mathsf L$ must have first output $C$;
in degree zero this forces input $BC$, $y(C)=BC$ and $H(BB)=CA$,
again giving $CAC$. The first $\mathsf R$ must now output $BBB$;
its only degree-zero route starts at $CA$ and uses $y(C)=BC$,
$H(CA)=BB$. The unchanged first letter is $A$.

Both copies of $\mathsf S$ have degree at most one. A degree-two path
must use their positive branches and the degree-zero middle path,
which gives exactly the stated term. For $s\notin\{ABC,ACB,BAC\}$,
a positive path would have to use degree zero in the first $\mathsf S$
and in $M_{\mathrm{loc}}$, followed by the positive second $\mathsf S$.
\autoref{loc:middlepath} then requires the first $\mathsf S$ to end at
$ACAC$ in degree zero, excluded by \autoref{loc:Srow}. For $ACB,BAC$, a degree-one term does not exceed $h_s$. For $ABC$,
only degree two exceeds $h_s$. These cases cover all input states.
\end{proof}

\begin{lemma}\label{loc:upper}
Every entry of
\[
\bigl(\Delta(\middlemap-\tpar\createmap\eta)\Delta^{-1},\Delta \createmap\bigr)
\]
belongs to $R$.
\end{lemma}

\begin{proof}
The unique constant monomial of the four-cup tensor $\kappa_8$ is
$(AC)^4$. On a permutation $s$ of $ABC$, the tripod evaluation $\eta$ has coefficient
$\tpar^{\inv(s)-3}$. The largest degree of $\tpar\kappa_8\eta$ at input $s$ is
$2-\inv(s)$, which is at most $h_s$ except at $s=ABC$.
There its unique degree-two monomial is precisely the one in
\autoref{loc:offender}. Thus
\[
(N_{\mathrm{loc}}-\tpar\kappa_8\eta)\Delta^{-1}\in\operatorname{Mat}(R).
\]
The coefficient of the canceled term is $1$ in both maps. Since
$\middlemap=ON_{\mathrm{loc}}$, $\createmap=O \kappa_8$, and $\Delta O$ has entries in $R$
by \autoref{loc:outer}, multiplication proves both assertions.
\end{proof}

\subsection{The lower row and a capping identity}\label{loc:lower}

We next estimate the lower row of $\widehat T$. For this we use a
web identity obtained by capping the first three outputs of $\middlemap$. Its square resolutions produce $d\removemap$ and eleven further
web maps. Define their sum by $\mathcal H=\sum_{j=1}^{11}\mathcal H_j$, where
$\mathcal H_j=P_j'h_j'$ and $P_j'$ has nine positive outputs. As in the table
for $\removemap$, the last column gives the input signs of the prefix map.
\begin{center}\small
\begin{tabular}{c l c c}\toprule
 &Downward string for $P_j'$&$h_j'$&Input of $P_j'$\\\midrule
$\mathcal H_1$&$Y_1Y_3Y_4a_4\mathsf R_1$&$I$&$+++$\\
$\mathcal H_2$&$Y_1Y_3Y_4Y_3Y_2Y_1$&$I$&$+++$\\
$\mathcal H_3$&$Y_1Y_2a_2Y_2Y_1Y_1$&$\lambda_2$&$+-$\\
$\mathcal H_4$&$Y_1Y_3Y_4\mathsf R_0Y_1\mathsf R_0Y_1$&$\lambda_2E_1$&$+-$\\
$\mathcal H_5$&$Y_1a_1Y_1Y_2Y_1Y_1$&$\lambda_2$&$+-$\\
$\mathcal H_6$&$Y_0Y_3Y_4Y_3Y_3\mathsf S$&$I$&$+++$\\
$\mathcal H_7$&$Y_0Y_2a_2Y_2a_2$&$\lambda_1E_2$&$-+$\\
$\mathcal H_8$&$Y_0Y_2a_2Y_2\mathsf R_0$&$I$&$+++$\\
$\mathcal H_9$&$Y_0Y_1Y_2Y_1Y_2a_2$&$\lambda_1$&$-+$\\
$\mathcal H_{10}$&$Y_0Y_1Y_2Y_1Y_2\mathsf R_0Y_1$&$\lambda_2$&$+-$\\
$\mathcal H_{11}$&$Y_0Y_1Y_2Y_1Y_1a_1$&$\lambda_1$&$-+$\\\bottomrule
\end{tabular}
\end{center}
The last row of the table for $\removemap$ gives $D_8=\tau^3\eta$.
Set $D_0=\removemap-D_8$, and write
$\operatorname{asc}(abc)=[a<b]+[b<c]$ for the number of adjacent
ascents in a three-letter input. We first bound the degrees of
$D_0$ and $\mathcal H$, then prove the capping identities.

\begin{lemma}\label{loc:ascent}
At every exterior state $\omega$ and input $s$,
\[
\deg_{\vpar}(D_0)_{\omega,s},\ \deg_{\vpar}\mathcal H_{\omega,s}\le\operatorname{asc}(s)\le1+h_s.
\]
\end{lemma}
\begin{proof}
Except for $P_1$ in $D_1$ and $P_6'$ in $\mathcal H_6$, all prefix strings
use only nonpositive $y$, cup, and $\mathsf R$ pieces. In those two
exceptions, $\mathsf S$ acts first on the original three-strand input and
the remaining pieces are nonpositive; their bound is $[s=ABC]$.
Their attached maps $h_1$ and $h'_6$ are identities, so this bound
is at most $\operatorname{asc}(s)$.

It remains to bound the small maps attached to the nonpositive
prefixes. The maps $I,\lambda_1,\lambda_2$ have degree at most
$\operatorname{asc}(abc)$ from \autoref{eq:localmaps}. For $\lambda_2E_1$, when $a\ne b$,
the two possible degrees are
\[
\epsilon(a,b)+[b<c],\qquad [a<c],\quad
\epsilon(a,b)=\begin{cases}1&a<b,\\-1&a>b.\end{cases}
\]
The first is at most $\operatorname{asc}(abc)$; the second is too, since no adjacent
ascent implies $a\ge b\ge c$. When $a=b$, the map is zero.
For $\lambda_1E_2$ the degrees are $[a<b]+\epsilon(b,c)$ and $[a<c]$,
with the same proof. Composition and summation preserve the bounds.
Finally, two strict ascents force $s=ABC$, so $\operatorname{asc}(s)\le1+h_s$.
\end{proof}

\begin{lemma}\label{loc:openidentity}
With all boundary positions fixed,
\begin{equation}\label{loc:openidentities}
\eta \middlemap=d\removemap+\mathcal H,\qquad \eta \createmap=d^2\tau^3.
\end{equation}
(The equation $\eta \createmap=d^2\tau^3$ is displayed in \autoref{loc:fun}.)
\end{lemma}

\begin{figure}[ht]
\centering
\includegraphics[width=0.9\textwidth]{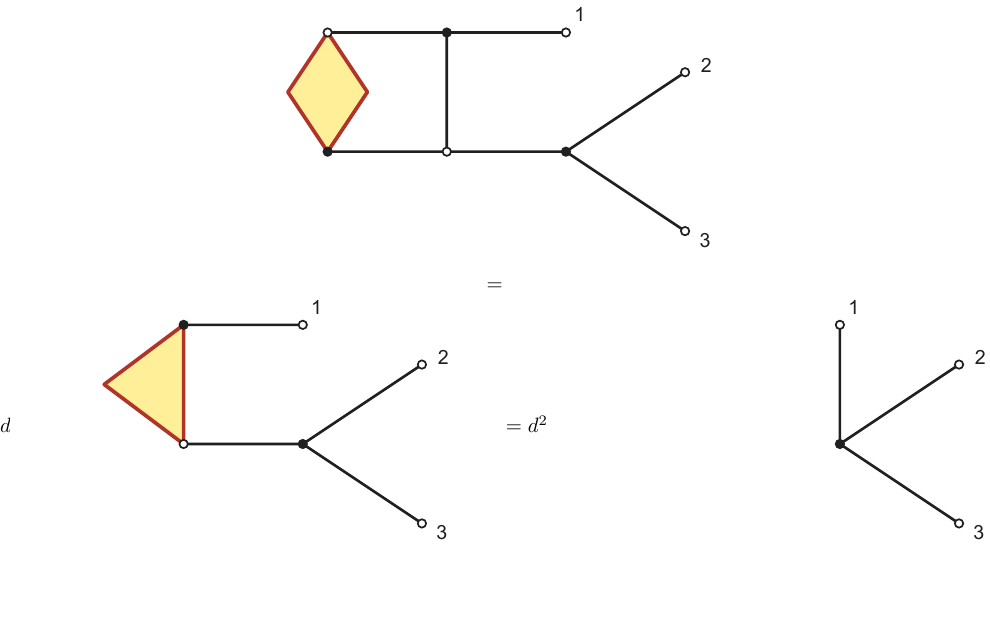}
\caption{The identity $\eta \createmap=d^2\tau^3$, with the two common
tripods suppressed. The six-output core has string
$Y_0H_0a_1Y_0a_0$. Its first three outputs are capped. Each
highlighted digon removes two vertices and contributes $d$.}
\label{loc:fun}
\end{figure}

\begin{proof}
For the first identity, cap the first three outputs of $\middlemap$.
Resolve the indicated square in \autoref{loc:firstsquare}. Its branches
are $R_1$ and $R_{22}$. A digon removal in the second branch gives
$R_{22}=dR_{23}$. The complete square-reduction tree is
\[
\begin{array}{rcl@{\qquad}rcl}
R_1&=&R_2+R_{11}&R_{23}&=&R_{24}+R_{31}\\
R_2&=&R_3+\mathcal H_4&R_{24}&=&R_{25}+D_4\\
R_3&=&R_4+R_7&R_{25}&=&R_{26}+D_6\\
R_4&=&\mathcal H_3+\mathcal H_5&R_{26}&=&D_8+D_7\\
R_7&=&\mathcal H_2+\mathcal H_1&R_{31}&=&R_{32}+R_{35}\\
R_{11}&=&R_{12}+R_{17}&R_{32}&=&D_3+D_5\\
R_{12}&=&R_{13}+\mathcal H_{10}&R_{35}&=&D_2+D_1\\
R_{13}&=&\mathcal H_{11}+\mathcal H_9&&&\\
R_{17}&=&R_{18}+\mathcal H_6&&&\\
R_{18}&=&\mathcal H_7+\mathcal H_8.&&&
\end{array}
\]
\autoref{loc:atlasappendix} displays every local relation and every
terminal beside the corresponding defining map. Each square has
coefficient one in both resolutions. Summing the left column gives
$\mathcal H$, and summing the right column gives $\removemap$.

The morphism $\createmap$ consists of its core with six boundary vertices, with leading word
$ABACBC$, followed by two tripods; this follows by reading its
displayed growth string. Capping the first three legs of that core
exposes a digon and then a second digon. Their removals leave a tripod.
The other two tripods are unchanged, giving the second identity.
\end{proof}

\begin{figure}[ht]
\centering
\includegraphics[width=0.9\textwidth]{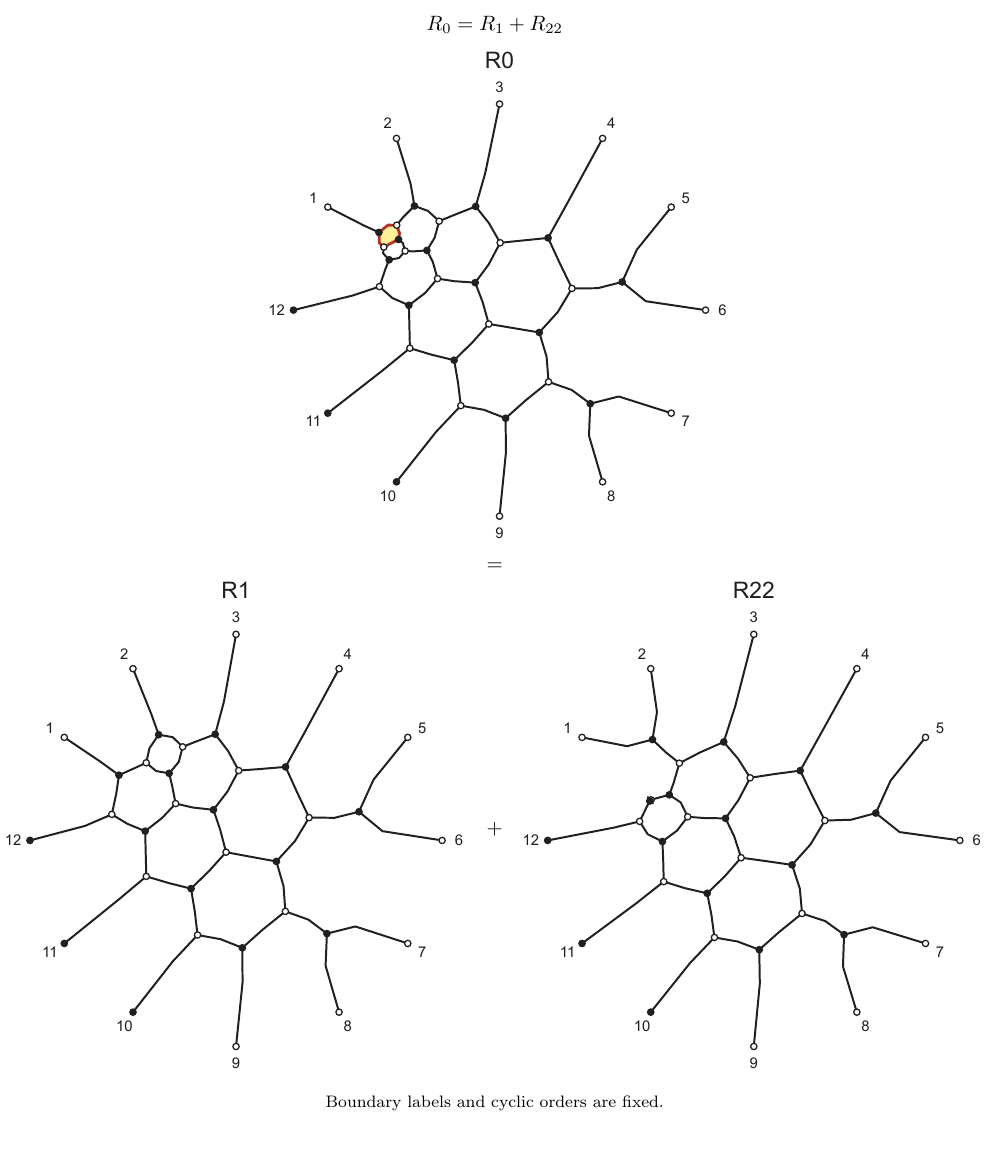}
\caption{First square in $\eta \middlemap$. Boundary $1,\ldots,9$ consists of
the outputs; $12,11,10$ are the inputs in their linear order.
Hollow vertices are sources and filled vertices are sinks, as in
\autoref{fig:relations}.
The highlighted face is replaced by its two non-crossing smoothings.
Subsequent relations are included in \autoref{loc:atlasappendix}.}
\label{loc:firstsquare}
\end{figure}

\begin{lemma}\label{loc:lowerrow}
Every entry in the lower row of $\widehat T$ belongs to $R$.
Its constant column at input $AAA$ is zero.
\end{lemma}

\begin{proof}
Multiplying $(G\otimes I_9)TG^{-1}$, its lower row is
\[
\left(
\bigl(\tpar\eta \middlemap-\removemap-\tpar(\tpar\eta \createmap-d\tau^3)\eta\bigr)\Delta^{-1},
\tpar\eta \createmap-d\tau^3
\right).
\]
Now use \autoref{loc:openidentities}, $\tpar d-1=\tpar^2$ and
$\removemap=D_0+\tau^3\eta$. The lower row becomes
\begin{equation}\label{loc:lowerformula}
\left((\tpar^2D_0+\tpar\mathcal H-\tpar^4\tau^3\eta)\Delta^{-1},\tpar^2d\tau^3\right).
\end{equation}
After the input shift $-h_s$, the first two terms have degrees
at most $-1$ and $0$, respectively,
by \autoref{loc:ascent}. The term containing $\eta$ has degree at most $-2$: its evaluation degrees on $ABC,ACB,BAC,BCA,CAB,CBA$ are $3,2,2,1,1,0$, and the shift is $-4-h_s$. The last term has degree at most $-1$.
At input $AAA$, $\operatorname{asc}(AAA)=0$ and $\eta(AAA)=0$, so the first two terms have
degrees at most $-2$ and $-1$. Thus that column has zero constant.
\end{proof}

\subsection{The constant state and the two ends}\label{loc:ends}

Both rows of $\widehat T$ now have entries in $R$. To identify its
constant action, we need only the input state $AAA$. The following
three lemmas track it through the block, prefix and tail.

\begin{lemma}\label{loc:constantblock}
The constant part of $\widehat T$ at input $AAA$ is exactly
$AAA\otimes e_{\wordmiddle}$, with coefficient one.
\end{lemma}

\begin{proof}
The tripod evaluation $\eta$ vanishes at $AAA$, so there is no
upper-row correction.
By \autoref{loc:Srow}, $\mathsf S(AAA)=AAAB+\tpar BAAA$.
A degree increase at the second $\mathsf S$ requires the middle
map to output a word beginning with $ABC$. The map
$M_{\mathrm{loc}}$ fixes the first letter, so this is impossible
from $BAAA$. From $AAAB$, producing the second letter $B$ forces
the first $\mathsf R$, at input $AA$, to use its degree $-1$
output $BAA$. The first input of the last $\mathsf R$ is then $A$,
whose split cannot have first output $C$. Thus the second
$\mathsf S$ cannot increase the degree.

Only the $AAAB$ term can contribute degree zero. The unique
degree-zero branches of $\mathsf R$ at $AA$ and $\mathsf L$ at $BB$ force
\[
M_{\mathrm{loc}}(AAAB)[0]=AAAABBC,\qquad N_{\mathrm{loc}}(AAA)[0]=AAABABBC.
\]
All other terms have negative degree and remain negative under
$\Delta O$ by \autoref{loc:outer}. On the remaining term, the upward
operations $\mathsf R_1,Y_7,Y_6,Y_2$ give the sole constant path
\[
\begin{aligned}
AAABABBC&\to AAABBABBC\to AAABBABACC\\
&\to AAABBAACACC\to AAABBBAACACC=AAA\,\wordmiddle.
\end{aligned}
\]
Other branches here have first three letters among $AAA,AAB,ABA,ABB$.
They all have $h_s=0$, so the rescaling adds no further constant term.
The lower-row assertion is \autoref{loc:lowerrow}.
\end{proof}

\begin{lemma}\label{loc:initial}
The coefficients of $\widehat p$ belong to $R$, and
$\widehat p[0](AAA)=e_{\wordstart}$.
\end{lemma}

\begin{proof}
Since the scalar entry of $\pp$ is zero,
$\widehat p=(\startmap\Delta^{-1},0)$. In $\startmap=Y_0Y_3\mathsf S$, the two $Y$'s
are nonpositive and \autoref{loc:nonpositive} gives
\[
\deg_{\vpar} (\startmap)_{w,s}\le[s=ABC]\le h_s.
\]
For $AAA$, the constant path is
$AAA\to AAAB\to AAAAC\to ABAAAC$; the other branch in
\autoref{loc:Srow} has negative degree and stays negative through the remaining splits.
\end{proof}

\begin{lemma}\label{loc:terminal}
All coefficients of $\widehat\rho$ belong to $R$, and
\[
\widehat\rho[0]=AAA\otimes e_{\wordend}.
\]
\end{lemma}

\begin{proof}
Cap the first three outputs of $\finishmap$ with $\eta$.
A square resolution and a digon removal give the relation in
\autoref{fig:moresquares}:
\begin{equation}\label{loc:tailidentity}
\eta \finishmap=d\tau^2+W_{AABBCC}.
\end{equation}

\begin{figure}[ht]
\centering
\begin{tikzpicture}
\node[inner sep=0,anchor=south west] (tailidentity) at (0,0)
  {\includegraphics[width=0.9\textwidth]{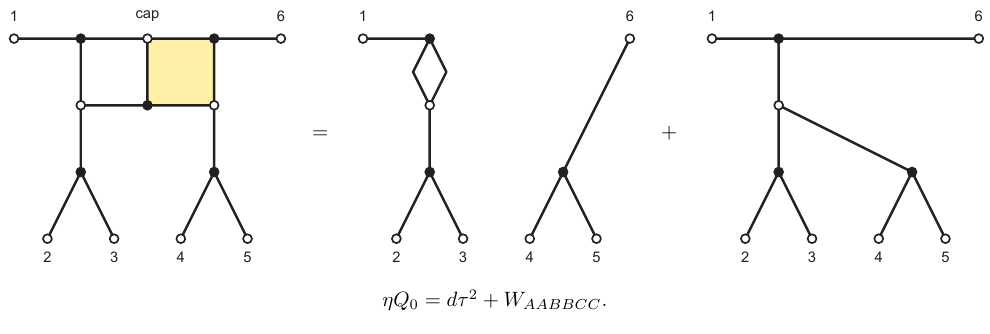}};
\begin{scope}[x={(tailidentity.south east)},y={(tailidentity.north west)}]
\fill[white] (0,0) rectangle (1,.18);
\node[anchor=north,inner sep=0,font=\small] at (.5,.15)
  {$\eta\finishmap=d\tau^2+W_{AABBCC}.$};
\end{scope}
\end{tikzpicture}
\caption{The terminal square relation \autoref{loc:tailidentity}. The second
branch is the tree with word $AABBCC$; the first becomes two tripods
after one digon removal. Boundary positions are fixed.}
\label{fig:moresquares}
\end{figure}

The tree has downward string
$Y_1Y_2Y_1Y_0a_0$. Thus
\[
\widehat\rho=\binom{\Delta \finishmap}{\tpar^2\tau^2+\tpar W_{AABBCC}}.
\]
The scalar entry is strictly negative in $\vpar$-degree. The tree and
$\tau^2$ contain only nonpositive $Y$ and cup factors.

The upward tail string is $a_0,Y_0,Y_1,Y_2,Y_1,Y_2,Y_3,Y_2$.
It has one degree-zero path and eight paths of degree $-1$: choose
one $-1$ branch and zero branches elsewhere. They give
\[
\begin{array}{c|l}
\text{changed upward step}&\text{complete output word}\\\hline
\text{none}&AAABBCBCC\\
0&AAACBCBCB\\
1&BAABBCACC\\
2&ABACBCABC\\
3&AAABBCCBC\\
4&ABABACBCC\\
5&AABCABBCC\\
6&AAABCBBCC\\
7&AABABCBCC.
\end{array}
\]
None starts with $ABC,ACB,BAC$. The coordinates with these first
three letters therefore have degree at most $-2$ before the shift
by $\Delta$, and remain strictly negative afterwards. All other
coordinates are unshifted. This leaves the single constant state
$AAA\,\wordend$.
\end{proof}

\subsection{Identifying the dual canonical vector}

We have the following.

\begin{lemma}\label{prop:canonical}
For every integer $m\ge0$, the tensor in \autoref{eq:iterate} satisfies
\[
\beta_m=b_{c_m}.
\]
\end{lemma}

\begin{proof}
The transformed maps have entries in $R$ by
\autoref{loc:upper}, \autoref{loc:lowerrow}, \autoref{loc:initial}
and \autoref{loc:terminal}. Their constant terms are
\[
\widehat\rho[0]=AAA\otimes e_{\wordend},\quad
\widehat T[0](AAA)=AAA\otimes e_{\wordmiddle},\quad
\widehat p[0](AAA)=e_{\wordstart}.
\]
Since the maps have entries in $R$, the constant term of a
composition is the composition of their constant terms.
Only the $AAA$ column of $\widehat T[0]$ is reached in this
induction. Induction therefore gives
\[
\beta_m=\widehat p\,\widehat T^{\,m}\widehat\rho
\equiv e_{\wordstart}\otimes e_{\wordmiddle}^{\otimes m}\otimes e_{\wordend}
=e_{c_m}\pmod {\tpar\sum_w R e_w}.
\]
Thus every nonleading tensor coordinate belongs to $\tpar R$.
Give $\mathcal E=V^{\otimes3}\oplus\one$ the componentwise
tensor bar. Each entry in \autoref{eq:maps} is a web map or a sum of
web maps, and the only nonconstant scalar is $-d$, with
$\overline d=d$. By \autoref{lem:web-bar}, the three maps commute
with the corresponding bars, so their composite $\beta_m$ is fixed
by the tensor bar. Together with the congruence above,
\autoref{thm:canonical} identifies it uniquely with $b_{c_m}$.
\end{proof}

We now compute the coordinate required for the KL coefficient.

\section{A two-state calculation of the coefficient}\label{sec:readout}

Fixing the boundary state $a_m$ turns the tensor construction into
a matrix calculation on the interface.

\subsection{The selected rows}

We have identified $\beta_m$ with $b_{c_m}$. By \autoref{thm:FKK},
it remains to extract the coefficient of $\tpar$ in its coordinate at
$a_m=(ABC)^{N_m}$.

Fix the six external outputs of $\widehat p$ and $\widehat\rho$ to
$ABCABC$, and the nine external outputs of $\widehat T$ to $(ABC)^3$.
Denote the resulting row, $28\times28$ matrix and column by $p_a$,
$\mathsf A$ and $\rho_a$, respectively. Then
\[
[e_{a_m}]\beta_m=p_a \mathsf A^m \rho_a.
\]
All entries lie in $R$. To find the coefficient of $\tpar$, we need the
leading terms of the ends and two rows of $\mathsf A[0]$.
The local tables in \autoref{sec:elementarytensors} give the exact end vectors
\begin{equation}\label{loc:selectedends}
p_a=\langle ABC|+\tpar\langle ACB|+\tpar\langle BAC|,
\qquad
\rho_a=\tpar\,ABC+\tpar^2(ACB+BAC+2e_*).
\end{equation}
Here $e_*$ is the scalar interface state from \autoref{def:corrected}.
For the prefix, the covector before multiplication by $\Delta^{-1}$ is
$\tpar^{-1}\langle ABC|+\langle ACB|+\langle BAC|$ by \autoref{loc:Srow}.
For the tail, the only nonzero unshifted coordinates are
$(\finishmap)_{ABC,ABCABC}=\tpar^2$, $(\finishmap)_{ACB,ABCABC}=(\finishmap)_{BAC,ABCABC}=\tpar^3$. Multiplication by $G$ gives
\autoref{loc:selectedends}, including $[e_{ABCABC}](\tpar\eta \finishmap-\tau^2)=2\tpar^2$ in the scalar
coordinate.

\begin{lemma}\label{loc:tworows}
The rows of $\mathsf A[0]$ indexed by $ABC$ and $*$ vanish outside
the columns indexed by these same two states. In the order $ABC,*$,
their restriction is
\[
K=\begin{pmatrix}0&1\\2&0\end{pmatrix}.
\]
\end{lemma}

\begin{proof}
Both selected output states and the external state $(ABC)^3$ have
weight zero. Since the rescaling and tripod evaluation in $G$
preserve weights, only inputs of weight zero can contribute.
These are the six permutations of $ABC$ and the scalar state.
The table gives the complete rows; we need only their constant terms.
\begin{center}\small
\begin{tabular}{c|c|c}
input $s$&$\mathsf A_{ABC,s}$&$\mathsf A_{*,s}$\\\hline
$ABC$&$6\tpar+6\tpar^3$&$2+12\tpar^2+6\tpar^4$\\
$ACB$&$5\tpar^2+2\tpar^4$&$2\tpar+8\tpar^3+2\tpar^5$\\
$BAC$&$5\tpar^2+3\tpar^4$&$2\tpar+10\tpar^3+3\tpar^5$\\
$BCA$&$2\tpar^2$&$2\tpar+2\tpar^3$\\
$CAB$&$\tpar^2$&$\tpar+2\tpar^3$\\
$CBA$&$0$&$\tpar^2$\\
$*$&$1$&$\tpar+\tpar^3$
\end{tabular}.
\end{center}
We derive the table from the elementary maps. Put $\omega=(ABC)^3$.
In the following formulas, a word $s$ denotes $\langle s|$.
Starting with the fixed output $ABC\omega$, the factorization
$\middlemap=ON_{\mathrm{loc}}$ gives the following covectors as we work
downward:
\begin{align*}
\langle ABC\omega|O={}&\tpar^2 AACCACAC+\tpar ACACACAC,\\
\langle ABC\omega|O\mathsf S={}&(1+\tpar^2)ABCACAC+\tpar BACACAC+\tpar ACBACAC,\\
\langle ABC\omega|O\mathsf S M_{\mathrm{loc}}={}&(\tpar+\tpar^3)AACC+(1+3\tpar^2+\tpar^4)ACAC\\
&+(\tpar+2\tpar^3)ACBB+(\tpar+\tpar^3)BBAC\\
&+\tpar^2(BBBB+BABC+ABCB).
\end{align*}
Applying the remaining $\mathsf S$ gives the raw $\middlemap$ row.
Pairing the first covector with $\kappa_8$ gives ${(\createmap)_{ABC,\omega}}=\nobreak \tpar$.
For $D_0$ and $\mathcal H$, use the prefix covectors listed in
\autoref{loc:covectorappendix} and then apply the small maps from
their defining tables. In the input order
\[
ABC,\ ACB,\ BAC,\ BCA,\ CAB,\ CBA,
\]
their sums are
\begin{align*}
(D_0)_{\omega,\bullet}
&=(5\tpar^{-1}+3\tpar,\ 4+\tpar^2,\ 5+2\tpar^2,\ 2\tpar,\ 2\tpar,\ \tpar^2),\\
\mathcal H_{\omega,\bullet}
&=(2\tpar^{-2}+8+3\tpar^2,\ 2\tpar^{-1}+5\tpar+\tpar^3,\ 2\tpar^{-1}+6\tpar+\tpar^3,\ 2+\tpar^2,\ 1+\tpar^2,\ \tpar).
\end{align*}
Substitute these rows and \autoref{eq:tripod} into the upper-row
formula in \autoref{loc:upper} and the lower-row formula
\autoref{loc:lowerformula}. This gives the displayed table. Its only
constant entries are $\mathsf A_{ABC,*}[0]=1$ and
$\mathsf A_{*,ABC}[0]=2$, which proves the lemma.

The constant $2$ comes from $\mathcal H_4$ and $\mathcal H_7$ in $\mathcal H$. Their coefficients at $(\omega,ABC)$ are
$\tpar^{-2}+2$ and $\tpar^{-2}$. Indeed,
$\lambda_2E_1(ABC)=\lambda_1E_2(ABC)=\tpar^{-2}AC+\tpar^{-1} BB$; the respective
prefix covectors are $(1+\tpar^2)\langle AC|+\tpar\langle BB|$ and
$\langle AC|$. The factor $\tpar$ in the lower row and the input factor $\tpar$ from
$\Delta^{-1}$ turn each $\tpar^{-2}$ term into $1$.
All other $\mathcal H_j$ have degree at most zero there.
\end{proof}

\subsection{The coefficient}

The two constant rows now give the coefficient for every block count.

\begin{lemma}\label{loc:coefficient}
For every $m\ge0$,
\[
[\tpar][e_{a_m}]b_{c_m}=
\begin{cases}2^{m/2}&m\text{ even},\\0&m\text{ odd}.\end{cases}
\]
\end{lemma}

\begin{proof}
The three maps have entries in $R$, and \autoref{loc:selectedends}
gives $\rho_a[0]=0$ and $[\tpar]\rho_a=e_{ABC}$. Thus the factor $\rho_a$ supplies
the single power of $\tpar$; all other factors must contribute their
constant terms. Since $p_a[0]=\langle ABC|$, we obtain
\[
[\tpar](p_a \mathsf A^m \rho_a)=e_{ABC}^{\mathsf T}\mathsf A[0]^m e_{ABC}
=(1,0)K^m(1,0)^{\mathsf T}.
\]
By \autoref{loc:tworows}, multiplying a row supported on $ABC,*$
by $\mathsf A[0]$ keeps it supported on those two coordinates. This
justifies the reduction to $K$. Finally, $K^2=2I$, so the expression
is $2^{m/2}$ for even $m$ and zero for odd $m$.
The identification with $b_{c_m}$ is \autoref{prop:canonical}.
\end{proof}

\begin{example}\label{loc:smallreadout}
For $m=0$, \autoref{loc:selectedends} gives
\[
[e_{a_0}]b_{c_0}=p_a\rho_a=\tpar+2\tpar^3.
\]
Here the gap is $11$, and \autoref{eq:FKK} yields
$P^{J,+}_{x_0,y_0}(\qpar)=2\qpar^4+\qpar^5$.
Its critical coefficient, and hence $\mu(x_0,y_0)$, is $1$.
For $m=2$, the first-order-calculation already suffices:
\[
[\tpar][e_{a_2}]b_{c_2}=(1,0)K^2(1,0)^{\mathsf T}=2.
\]
This is the coefficient of $\qpar^{22}$ in the ordinary polynomial.
The local argument proves the value without computing that polynomial.
\end{example}

\subsection{Wrap up}\label{S:Wrap}

We are ready to prove the final results.

\begin{proof}[Proof of \autoref{thm:main}]
By \autoref{prop:canonical} and \autoref{loc:coefficient},
$[\tpar]([e_{a_{2r}}]b_{c_{2r}})=2^r$. The FKK identification in
\autoref{thm:FKK} therefore gives $\mu(x_{2r},y_{2r})=2^r$.
\end{proof}

\begin{corollary}\label{thm:rankbound}
For $n\ge12$, we have
\[
\max_{x,y\in S_n}\mu(x,y)\ge2^{\lfloor(n-12)/18\rfloor}.
\]
The same holds, mutatis mutandis, in finite and affine types ABCD.
\end{corollary}

\begin{proof}
Take $r=\lfloor(n-12)/18\rfloor$ in \autoref{thm:main} and use
$S_{12+18r}\subseteq S_n$. KL polynomials are unchanged under this
standard parabolic embedding.
\end{proof}

\begingroup
\renewcommand{\remarkend}{}

\begin{remark}
The following are consequences of \autoref{thm:rankbound}
in conjunction with the cited references.
\begin{enumerate}[label=\emph{\upshape(\alph*)}]
\color{black}
\item
First extension dimensions between simple modules for algebraic
groups and finite groups of Lie type can grow at least exponentially
with the rank, in sufficiently large characteristic. This follows from
\cite{ParshallScott,SX}.

\item
Indecomposables of composition length two in three-block parabolic
category $\mathcal O$ of type A occur in projective families of
dimension growing at least exponentially with the rank. This follows from
\cite{BGS,KoMazorchukMrden}.

\item
Every complex projective variety occurs as a submodule variety
in a fixed such category, already for objects of Loewy length at
most two with composition factors drawn from two fixed simples. This follows from
\cite{KoMazorchukMrden,RingelGrassmannians}.

\item
Schubert normal slices with no nontrivial product decomposition
can have local intersection cohomology in the highest permitted
degree growing at least exponentially with the slice dimension. This follows from
\cite{deCataldoMigliorini,KLPositivity}.

\item
Tensoring an indecomposable Soergel bimodule with a single
generating bimodule can yield summand multiplicities growing
at least exponentially with the rank. This follows from \cite{EliasWilliamson,KL};
see also \cite{GrowthAffineHecke} for related tensor power growth.
\nobreak\hfill{$\Diamond$}
\end{enumerate}
\end{remark}
\endgroup

\section{Methods and AI assistance}\label{sec:methods}

We record how we found the family and how we used AI assistance.

\subsection{Human experience (personal account)}

One lesson we have learned over our careers is that if things get
bad, they get very bad (see, e.g., \cite{TorsionProliferation,LinkDetection} 
for examples from our works). Looking at Warrington's tables \cite{Warrington},
we found it hard to believe that $\mu$-coefficients could be bounded.

Our experience with webs pointed in the same direction. The move
from $\mathfrak{sl}_2$ to $\mathfrak{sl}_3$ is a serious jump: often, $\mathfrak{sl}_2$ is nice, while
$\mathfrak{sl}_3$ seems to contain all the complexity.
And, at least in type A, KL theory and webs are essentially the
same thing. The familiar hexagon-in-hexagons picture in \cite{KK} therefore gave us a natural starting
point. We expected unboundedness to come from pictures of this kind. (The contrast is: 
for $\mathfrak{sl}_2$, the web basis consists of Temperley--Lieb
diagrams without internal faces, and the corresponding
$\mu$-coefficients lie in $\{0,1\}$; even stronger, the corresponding parabolic 
KL polynomials themselves are monomials.)

\begin{remark}
Our intuition that $\mathfrak{sl}_3$ contains all the complexity
goes further. Polo's theorem \cite{Polo99} says that all polynomials
occur as KL polynomials for symmetric groups, subject to the obvious
restrictions.
We believe that an appropriate version of this statement already
holds for permutations whose RSK shapes have at most three rows.
We do not yet have enough experimental evidence to formulate
a precise conjecture.
\end{remark}

The difficult part was turning that intuition into a proof.
We knew from the start that the necessary calculations would be
very hard by hand, but they seemed amenable to a machine.
That is where AI (mid-2026 models) came in: to help us get the calculations and
details straight, e.g. it found the correct 
family to look at as ours was a bit off.

We now give a few more details.

\subsection{The starting point}\label{sec:origin}

We began by computing KL left cell representations with an
implementation of Nguyen's method \cite{Nguyen}. We generated the
cell $W$-graphs through $S_{19}$ and for selected larger partitions.
The larger weights led us to three-row tableaux and their webs.
We then looked for a boundary coefficient that we could compute
after correcting a web to its dual canonical vector.

\begin{remark}
Our preliminary computational results were obtained with a Rust program which can be found 
on our associated page \cite{GoLaTuVa-mu-code}.
\end{remark}

Our starting example was the following pair of shape $(9,7,7)$,
with computed $W$-graph weight $2$:
\begingroup\small\setlength{\arraycolsep}{3pt}
\[
S^\circ=\begin{array}{rrrrrrrrr}
1&2&5&8&11&14&17&20&23\\
3&6&9&12&15&18&21&&\\
4&7&10&13&16&19&22&&
\end{array},\qquad
T^\circ=\begin{array}{rrrrrrrrr}
1&2&3&5&8&9&11&18&20\\
4&6&7&14&15&16&17&&\\
10&12&13&19&21&22&23&&
\end{array}.
\]
\endgroup
We completed both tableaux to shape $(9,9,9)$ by adding $24,26$ to
row two and $25,27$ to row three, obtaining $S^+,T^+$, see \autoref{fig:cores1}. Deleting the
added entries in reverse order removes the same corner from both
tableaux at each step. By \cite[Proposition 3.3]{GossowYacobi}, this
preserves the symmetrized $\mu$-coefficient, that is, the $W$-graph
weight of the unordered pair.

\begin{figure}[!ht]
\centering
\includegraphics[width=.76\linewidth]{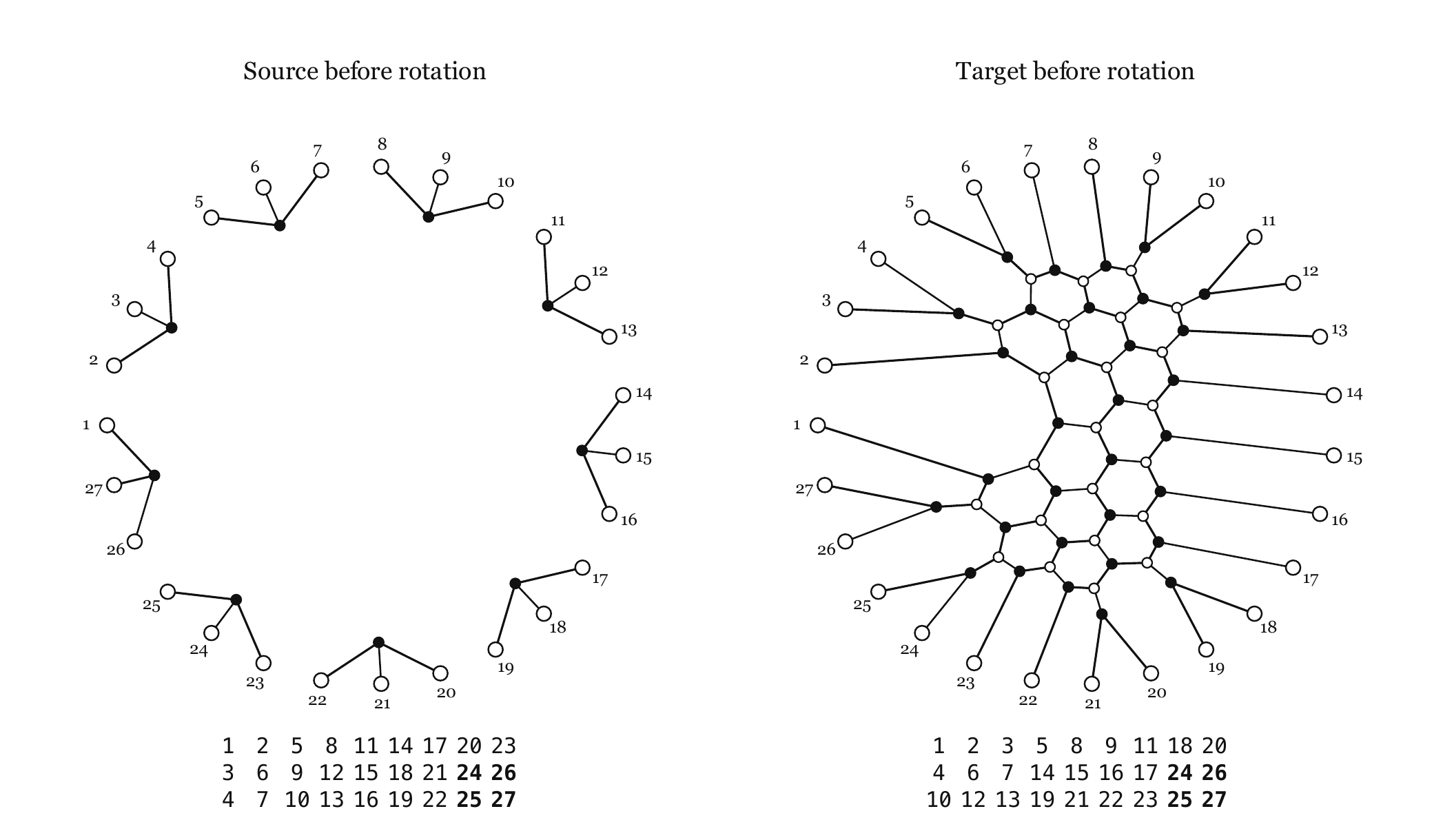}
\caption{The webs for $S^+,T^+$ and their
tableaux. Bold entries are the four added boxes.}
\label{fig:cores1}
\end{figure}

Rotate both webs by one boundary position, moving boundary vertex
$1$ to position $27$ and each vertex $i>1$ to position $i-1$.
With the boundary labels fixed, this gives the tableaux
$\operatorname{pr}(S^+),\operatorname{pr}(T^+)$, where
$\operatorname{pr}$ is promotion \cite[Theorem 2.5]{PPR}.
Simultaneous promotion preserves the symmetrized $\mu$-coefficient
for rectangular tableaux \cite[Proposition 3.2]{Rhoades}, so the
rotated pair still has $W$-graph weight $2$.

Motivated by the KK counterexample \cite{KK}, we
extended the hexagonal strip in the rotated target, with overlap at
boundary positions $13,22,31,\ldots$. This suggested the row words
\[
(ABC)^{6+3k},\qquad R_{\mathrm{seed}}\wordmiddle^k\wordend,\qquad
R_{\mathrm{seed}}=\word{AABABBAACACC},\quad k\ge0,
\]
with $\wordmiddle,\wordend$ as in \autoref{eq:words}. The case $k=1$ is the promoted pair just described, shown on the
left in \autoref{fig:cores2}. We computed its canonical correction, but did not find
a correction that worked for this whole family. We therefore looked
for a simpler start piece.

\begin{figure}[!ht]
\centering
\includegraphics[width=.70\linewidth]{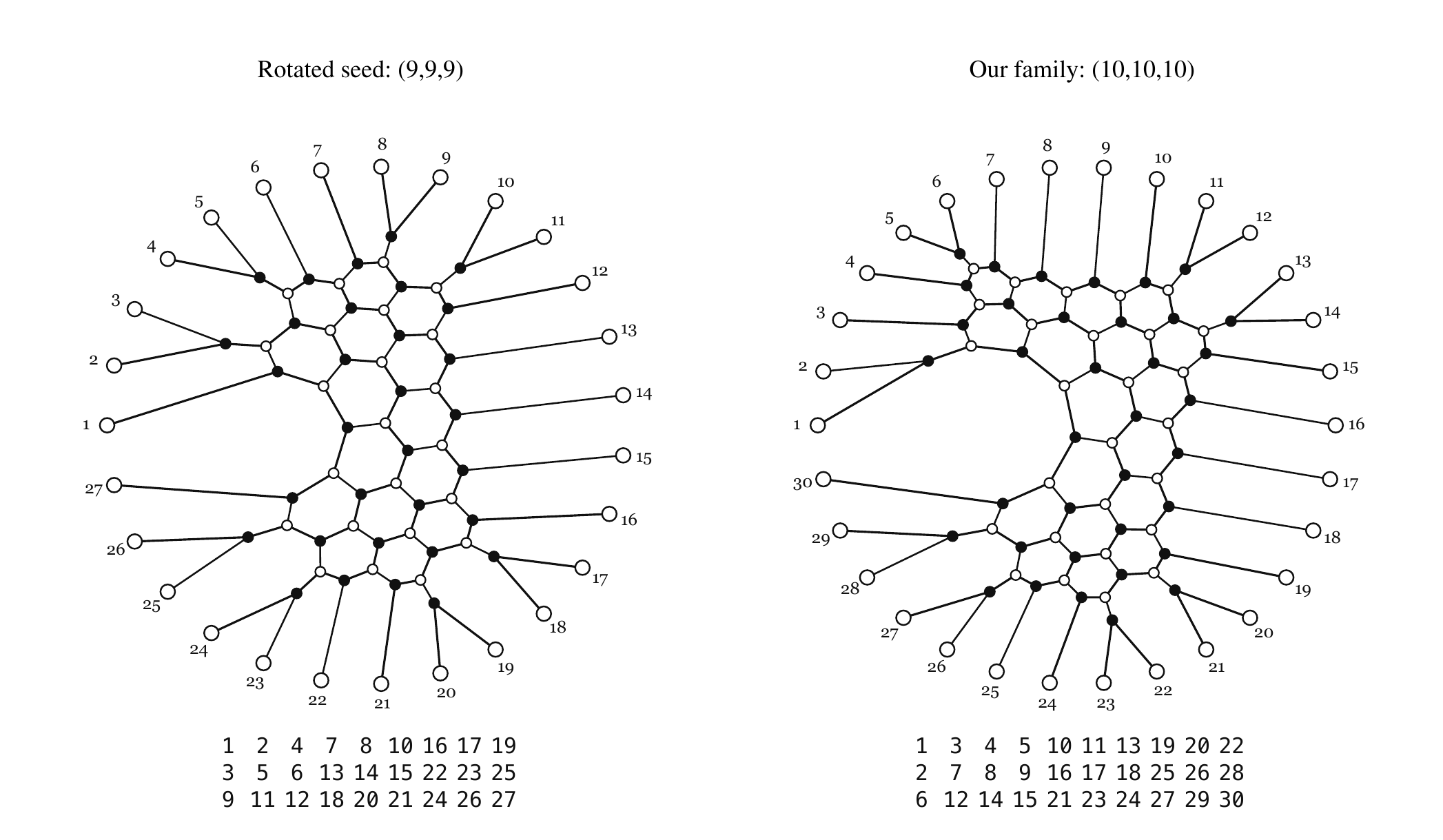}
\caption{Left: the rotated seed
$R_{\mathrm{seed}}\wordmiddle\wordend$. Right: the member $\wordstart\wordmiddle^2\wordend$.
Both source webs are products of consecutive tripods.}
\label{fig:cores2}
\end{figure}

\subsection{Choosing the start piece}

Now we explain how the right web in \autoref{fig:cores2} came about.
Robert's red-graph-reductions \cite{RobertFamily,RobertCharacterization}
suggested local correction terms. Such a reduction removes edges
bounding selected faces and reconnects the remaining half-edges by
specified pairings, preserving the ordered boundary
\cite[Definition 4.1]{RobertCharacterization}. For an exact paired
red graph in a non-elliptic web, the reduced web's projective module
is a direct summand of the original web module
\cite[Theorem 3.12 and Proposition 4.4]{RobertCharacterization}.

In the seed, the six hexagons surrounding an endpoint hexagon form
such a red graph. Reducing at that end and applying the square and
digon relations suggested the six-letter prefix $\wordstart$. Keeping $\wordmiddle$
and $\wordend$ gave the family in \autoref{eq:words}. The two families
are compared below; their indices $k$ and $m$ count blocks in their
respective constructions.
\begin{center}\small
\begin{tabular}{lccc}\toprule
Family&Source word&Target word&Block count\\\midrule
Seed&$(ABC)^{6+3k}$&$R_{\mathrm{seed}}\wordmiddle^k\wordend$&$k\ge0$\\
Theorem&$(ABC)^{4+3m}$&$\wordstart\wordmiddle^m\wordend$&$m\ge0$\\\bottomrule
\end{tabular}
\end{center}
See also \autoref{fig:cores2}.
Robert's theorem concerns module summands. The values of $\mu$ for
the modified start piece follow from \autoref{sec:canonicality} and
\autoref{sec:readout}. Its base member has $\mu=1$, whereas the
initial seed has $W$-graph weight $2$.

\subsection{From a seed calculation to a uniform argument}

For our family with two blocks, the complete canonical
correction has $94$ lower webs, as in \autoref{eq:94correction}.
At the selected boundary state, the raw critical coefficient is $18$
and the correction contributes $16$, leaving $2$. Grouping the
correction terms by their behavior across a three-strand cut splits
this calculation into $18-10-6=2$: $\startmap\middlemap\createmap\tau^2$ contributes
$10\tpar$, and $\startmap\createmap\removemap\finishmap-d\startmap\createmap\tau^5$ contributes $6\tpar$ to the amount
subtracted. These are the same terms grouped by their web-basis
coefficients as $18-3-13$ in \autoref{sec:basischange}.

This grouping suggested adding the scalar coordinate to the
three-strand interface in \autoref{sec:construction}. It allows one
fixed block to incorporate the correction for every $m$.

\subsection{AI assistance and verification}

We did the usual back-and-forth dance with AI (classic 2026).
We supplied the computational program, the selected $W$-graph
edges, their web interpretation and the first conjectural family.
We first used AI assistance for figures. Later models accessed
through Claude, ChatGPT and Codex helped with code, the modified
start piece, correction terms, the repeated block, Lean formalization and the proof draft.
AI-assisted symbolic calculations also produced the interface basis
change and the square-reduction tree.

We checked the local formulas by the calculations in the paper and
by independent tensor computations. The code is available in \cite{GoLaTuVa-mu-code}, and the scope of
the Lean formalization is described in \autoref{rem:verification}.
We are responsible for the mathematical statements and for verifying
the theoretical and computational inputs.

\appendix

\section{Details for \autoref{loc:tworows}}\label{loc:covectorappendix}

\addtocounter{subsection}{1}

Fix the nine outputs of each prefix $P_i$ or $P_j'$ to
$\omega=(ABC)^3$, and work downward through its growth string.
The tables give the resulting covector and the small map attached
to that prefix. Thus
$\langle\omega|D_i=(\langle\omega|P_i)h_i$, and similarly for $\mathcal H_j$.
Here a word $s$ denotes $\langle s|$.

The first table lists $D_1,\ldots,D_7$, whose sum is $D_0$.
The term $D_8=\tau^3\eta$ is treated separately in
\autoref{loc:lowerformula}. The second table supplies the terms of
$\mathcal H$. All covectors follow from the elementary tables in
\autoref{sec:webs}.
\begin{center}\small
\begin{minipage}[t]{.47\linewidth}\vspace{0pt}\centering
\begin{tabular}{@{}clc@{}}\toprule
Prefix&Covector&$h_i$\\\midrule
$P_1$&$\tpar^{-1} ABC+ACB+BAC$&$I$\\
$P_2$&$\tpar ABC+\tpar^2BAC$&$I$\\
$P_3$&$AC+\tpar BB$&$\lambda_2$\\
$P_4$&$\tpar AC+\tpar^2BB$&$\lambda_2E_1$\\
$P_5$&$\tpar^2AC$&$\lambda_2$\\
$P_6$&$\tpar AC$&$\lambda_2E_1$\\
$P_7$&$AC$&$\lambda_1$\\\bottomrule
\end{tabular}
\end{minipage}\hfill
\begin{minipage}[t]{.51\linewidth}\vspace{0pt}\centering
\begin{tabular}{@{}clc@{}}\toprule
Prefix&Covector&$h_j'$\\\midrule
$P_1'$&$ABC+\tpar ACB$&$I$\\
$P_2'$&$\tpar^2ABC$&$I$\\
$P_3'$&$\tpar AC$&$\lambda_2$\\
$P_4'$&$(1+\tpar^2)AC+\tpar BB$&$\lambda_2E_1$\\
$P_5'$&$\tpar^3AC$&$\lambda_2$\\
$P_6'$&$ABC+\tpar ACB+\tpar BAC$&$I$\\
$P_7'$&$AC$&$\lambda_1E_2$\\
$P_8'$&$ABC+\tpar BAC$&$I$\\
$P_9'$&$\tpar AC$&$\lambda_1$\\
$P_{10}'$&$\tpar AC+\tpar^2BB$&$\lambda_2$\\
$P_{11}'$&$\tpar^3AC$&$\lambda_1$\\\bottomrule
\end{tabular}
\end{minipage}
\end{center}

For example, the first four downward steps of the $P_1$ string give
\[
ABCABCABC\to ACABCABC\to ACACABC\to ACACAC\to ACAC.
\]
The coefficient along this path is $1$. Applying the remaining
$\mathsf S$ gives
$\tpar^{-1}\langle ABC|+\langle ACB|+\langle BAC|$ by
\autoref{loc:Srow}, the first covector in the table.
Composing the listed covectors on the right with their small maps
and summing gives the $D_0$ and $\mathcal H$ rows used in
\autoref{loc:tworows}.

\section{Details for \autoref{loc:openidentity}}\label{loc:atlasappendix}

\addtocounter{subsection}{1}

The atlas proves $\eta \middlemap=d\removemap+\mathcal H$ from
\autoref{loc:openidentity}. The diagrams $R_i$ form a reduction
tree: at each step we resolve the highlighted square or remove the
highlighted digon. The terminal diagrams are the eight maps $D_i$
and the eleven maps $\mathcal H_j$. The tables below locate each step and
each terminal identification. Each terminal is drawn once, with
its label in the reduction tree and its defining map written above it. The symbols $\mathsf R_i,\mathsf L_i$
in the growth strings denote local operations, whereas $R_i$ denotes
a diagram in this tree.

Boundary labels are fixed. Outputs $1,\ldots,9$ are followed around
the circle by inputs $12,11,10$. Hollow vertices are sources and
filled vertices are sinks. Each square resolution has coefficient
$1$; removing the sole digon contributes $d$.

\begin{center}\small
\begin{minipage}[t]{.52\linewidth}\centering
\begin{tabular}{@{}lc@{}}\toprule
Reduction&Figure\\\midrule
$R_{0}=R_{1}+R_{22}$&\ref{fig:atlas-r0}\\
$R_{1}=R_{2}+R_{11}$&\ref{fig:atlas-r1}\\
$R_{2}=R_{3}+R_{10}$&\ref{fig:atlas-r2}\\
$R_{3}=R_{4}+R_{7}$&\ref{fig:atlas-r3}\\
$R_{4}=R_{5}+R_{6}$&\ref{fig:atlas-r4}\\
$R_{7}=R_{8}+R_{9}$&\ref{fig:atlas-r7}\\
$R_{11}=R_{12}+R_{17}$&\ref{fig:atlas-r11}\\
$R_{12}=R_{13}+R_{16}$&\ref{fig:atlas-r12}\\
$R_{13}=R_{14}+R_{15}$&\ref{fig:atlas-r13}\\
$R_{17}=R_{18}+R_{21}$&\ref{fig:atlas-r17}\\
$R_{18}=R_{19}+R_{20}$&\ref{fig:atlas-r18}\\
$R_{22}=dR_{23}$&\ref{fig:atlas-r22}\\
$R_{23}=R_{24}+R_{31}$&\ref{fig:atlas-r23}\\
$R_{24}=R_{25}+R_{30}$&\ref{fig:atlas-r24}\\
$R_{25}=R_{26}+R_{29}$&\ref{fig:atlas-r25}\\
$R_{26}=R_{27}+R_{28}$&\ref{fig:atlas-r26}\\
$R_{31}=R_{32}+R_{35}$&\ref{fig:atlas-r31}\\
$R_{32}=R_{33}+R_{34}$&\ref{fig:atlas-r32}\\
$R_{35}=R_{36}+R_{37}$&\ref{fig:atlas-r35}\\
\bottomrule
\end{tabular}
\end{minipage}\hfill
\begin{minipage}[t]{.43\linewidth}\centering
\begin{tabular}{@{}lc@{}}\toprule
Terminal map&Figure\\\midrule
$R_{5}=\mathcal H_{3}$&\ref{fig:atlas-term6}\\
$R_{6}=\mathcal H_{5}$&\ref{fig:atlas-term7}\\
$R_{8}=\mathcal H_{2}$&\ref{fig:atlas-term5}\\
$R_{9}=\mathcal H_{1}$&\ref{fig:atlas-term5}\\
$R_{10}=\mathcal H_{4}$&\ref{fig:atlas-term6}\\
$R_{14}=\mathcal H_{11}$&\ref{fig:atlas-term10}\\
$R_{15}=\mathcal H_{9}$&\ref{fig:atlas-term9}\\
$R_{16}=\mathcal H_{10}$&\ref{fig:atlas-term9}\\
$R_{19}=\mathcal H_{7}$&\ref{fig:atlas-term8}\\
$R_{20}=\mathcal H_{8}$&\ref{fig:atlas-term8}\\
$R_{21}=\mathcal H_{6}$&\ref{fig:atlas-term7}\\
$R_{27}=D_{8}$&\ref{fig:atlas-term4}\\
$R_{28}=D_{7}$&\ref{fig:atlas-term4}\\
$R_{29}=D_{6}$&\ref{fig:atlas-term3}\\
$R_{30}=D_{4}$&\ref{fig:atlas-term2}\\
$R_{33}=D_{3}$&\ref{fig:atlas-term2}\\
$R_{34}=D_{5}$&\ref{fig:atlas-term3}\\
$R_{36}=D_{2}$&\ref{fig:atlas-term1}\\
$R_{37}=D_{1}$&\ref{fig:atlas-term1}\\
\bottomrule
\end{tabular}
\end{minipage}
\end{center}
\newcommand{\AtlasRelation}[1]{%
\ifnum#1=12
\(\vcenter{\hbox{\includegraphics[page=#1,trim=125bp 285bp 122bp 33bp,clip,width=.35\linewidth]{figures/local-atlas-reflow.pdf}}}
\quad=d\quad
\vcenter{\hbox{\includegraphics[page=#1,trim=125bp 29bp 122bp 290bp,clip,width=.35\linewidth]{figures/local-atlas-reflow.pdf}}}\)%
\else
\(\vcenter{\hbox{\includegraphics[page=#1,trim=125bp 298bp 122bp 22bp,clip,width=.30\linewidth]{figures/local-atlas-reflow.pdf}}}
\;=\;
\vcenter{\hbox{\includegraphics[page=#1,trim=0bp 44bp 249bp 277bp,clip,width=.30\linewidth]{figures/local-atlas-reflow.pdf}}}
\;+\;
\vcenter{\hbox{\includegraphics[page=#1,trim=248bp 44bp 0bp 277bp,clip,width=.30\linewidth]{figures/local-atlas-reflow.pdf}}}\)%
\fi}
\newcommand{\AtlasTerminal}[2]{%
\begin{minipage}[t]{.30\linewidth}\vspace{0pt}\centering
{\small$#1$}\par\smallskip
\includegraphics[page=#2,width=\linewidth,height=.20\textheight,keepaspectratio]{figures/terminal-maps.pdf}%
\end{minipage}}

\begingroup
\setlength{\abovecaptionskip}{3pt}
\setlength{\belowcaptionskip}{0pt}

\clearpage
\begin{figure}[p]
\centering
\AtlasRelation{1}
\caption{Square resolution at $R_{0}$.}
\label{fig:atlas-r0}
\par\medskip
\AtlasRelation{2}
\caption{Square resolution at $R_{1}$.}
\label{fig:atlas-r1}
\par\medskip
\AtlasRelation{3}
\caption{Square resolution at $R_{2}$.}
\label{fig:atlas-r2}
\par\medskip
\AtlasRelation{4}
\caption{Square resolution at $R_{3}$.}
\label{fig:atlas-r3}
\end{figure}

\clearpage
\begin{figure}[p]
\centering
\AtlasRelation{5}
\caption{Square resolution at $R_{4}$.}
\label{fig:atlas-r4}
\par\medskip
\AtlasRelation{6}
\caption{Square resolution at $R_{7}$.}
\label{fig:atlas-r7}
\par\medskip
\AtlasRelation{7}
\caption{Square resolution at $R_{11}$.}
\label{fig:atlas-r11}
\par\medskip
\AtlasRelation{8}
\caption{Square resolution at $R_{12}$.}
\label{fig:atlas-r12}
\end{figure}

\clearpage
\begin{figure}[p]
\centering
\AtlasRelation{9}
\caption{Square resolution at $R_{13}$.}
\label{fig:atlas-r13}
\par\medskip
\AtlasRelation{10}
\caption{Square resolution at $R_{17}$.}
\label{fig:atlas-r17}
\par\medskip
\AtlasRelation{11}
\caption{Square resolution at $R_{18}$.}
\label{fig:atlas-r18}
\par\medskip
\AtlasRelation{12}
\caption{Digon removal at $R_{22}$, with coefficient $d$.}
\label{fig:atlas-r22}
\end{figure}

\clearpage
\begin{figure}[p]
\centering
\AtlasRelation{13}
\caption{Square resolution at $R_{23}$.}
\label{fig:atlas-r23}
\par\medskip
\AtlasRelation{14}
\caption{Square resolution at $R_{24}$.}
\label{fig:atlas-r24}
\par\medskip
\AtlasRelation{15}
\caption{Square resolution at $R_{25}$.}
\label{fig:atlas-r25}
\par\medskip
\AtlasRelation{16}
\caption{Square resolution at $R_{26}$.}
\label{fig:atlas-r26}
\end{figure}

\clearpage
\begin{figure}[p]
\centering
\AtlasRelation{17}
\caption{Square resolution at $R_{31}$.}
\label{fig:atlas-r31}
\par\medskip
\AtlasRelation{18}
\caption{Square resolution at $R_{32}$.}
\label{fig:atlas-r32}
\par\medskip
\AtlasRelation{19}
\caption{Square resolution at $R_{35}$.}
\label{fig:atlas-r35}
\end{figure}

\clearpage
\begin{figure}[p]
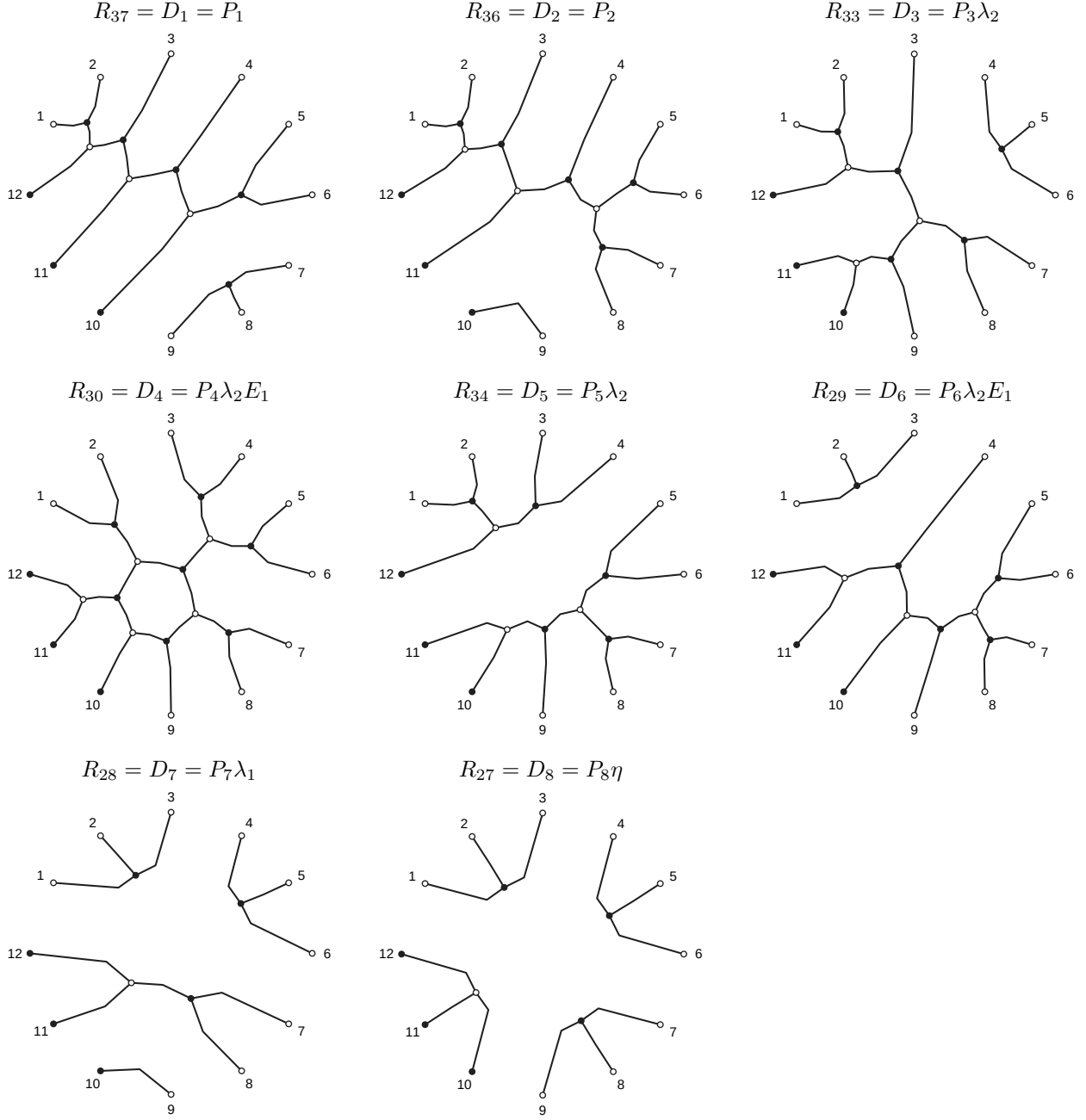

\centering
\begin{tabular}{@{}c@{\hspace{.04\linewidth}}c@{\hspace{.04\linewidth}}c@{}}
\AtlasTerminal{R_{37}=D_1=P_1}{1} &
\AtlasTerminal{R_{36}=D_2=P_2}{2} &
\AtlasTerminal{R_{33}=D_3=P_3\lambda_2}{3}\\[7pt]
\AtlasTerminal{R_{30}=D_4=P_4\lambda_2E_1}{4} &
\AtlasTerminal{R_{34}=D_5=P_5\lambda_2}{5} &
\AtlasTerminal{R_{29}=D_6=P_6\lambda_2E_1}{6}\\[7pt]
\AtlasTerminal{R_{28}=D_7=P_7\lambda_1}{7} &
\AtlasTerminal{R_{27}=D_8=P_8\eta}{8} &

\end{tabular}
\begingroup
\caption{The terminal maps $D_1,\ldots,D_8$.
Each diagram shows the full composite $D_i=P_i h_i$.
Outputs are labeled $1,\ldots,9$; the three inputs are labeled
$12,11,10$ in their linear order.}
\label{fig:atlas-term1}
\label{fig:atlas-term2}
\label{fig:atlas-term3}
\label{fig:atlas-term4}\endgroup
\end{figure}

\clearpage
\begin{figure}[p]
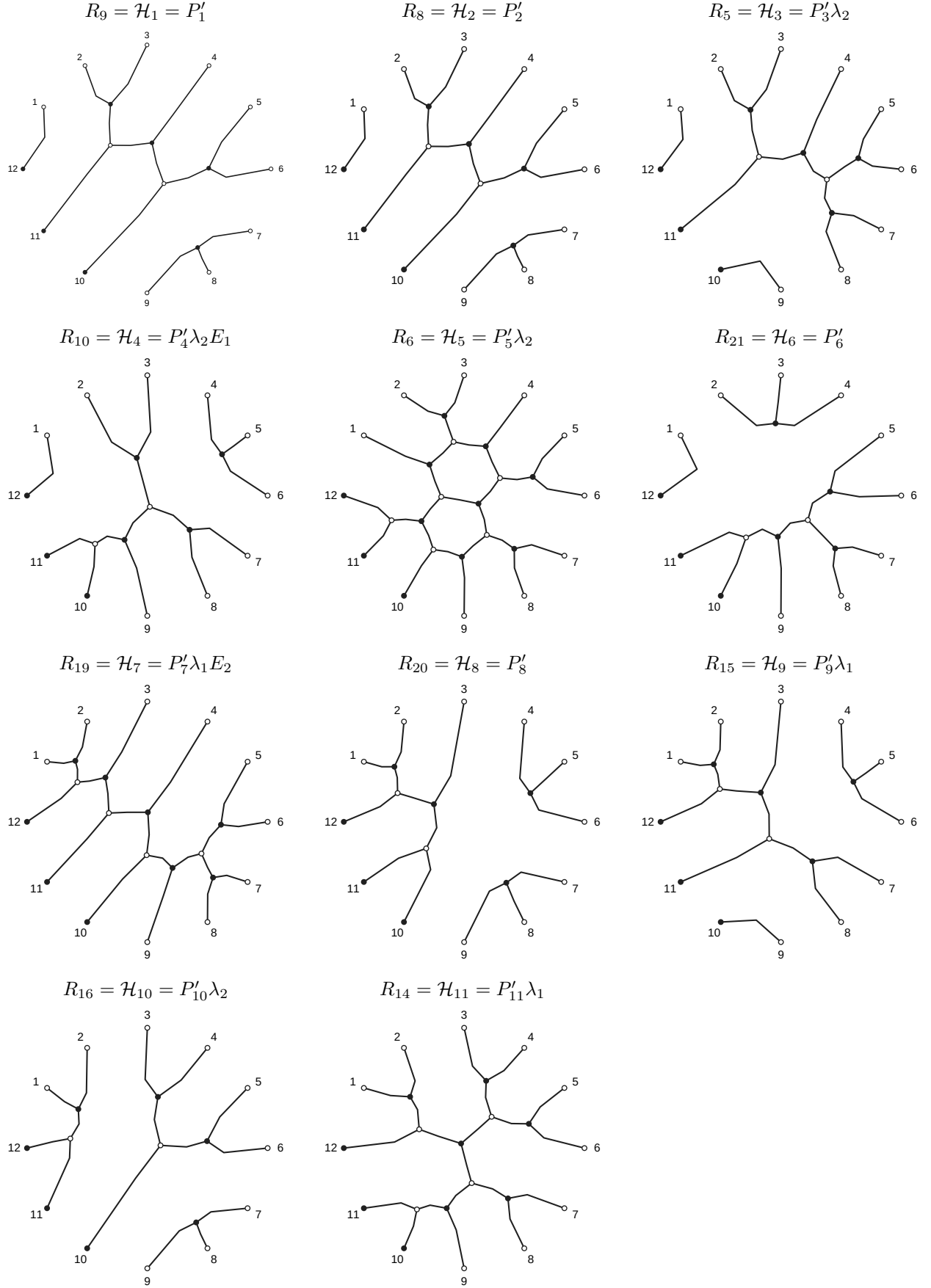

\centering
\begin{tabular}{@{}c@{\hspace{.04\linewidth}}c@{\hspace{.04\linewidth}}c@{}}
\AtlasTerminal{R_9=\mathcal H_1=P'_1}{9} &
\AtlasTerminal{R_8=\mathcal H_2=P'_2}{10} &
\AtlasTerminal{R_5=\mathcal H_3=P'_3\lambda_2}{11}\\[7pt]
\AtlasTerminal{R_{10}=\mathcal H_4=P'_4\lambda_2E_1}{12} &
\AtlasTerminal{R_6=\mathcal H_5=P'_5\lambda_2}{13} &
\AtlasTerminal{R_{21}=\mathcal H_6=P'_6}{14}\\[7pt]
\AtlasTerminal{R_{19}=\mathcal H_7=P'_7\lambda_1E_2}{15} &
\AtlasTerminal{R_{20}=\mathcal H_8=P'_8}{16} &
\AtlasTerminal{R_{15}=\mathcal H_9=P'_9\lambda_1}{17}\\[7pt]
\AtlasTerminal{R_{16}=\mathcal H_{10}=P'_{10}\lambda_2}{18} &
\AtlasTerminal{R_{14}=\mathcal H_{11}=P'_{11}\lambda_1}{19} &

\end{tabular}
\caption{The terminal maps $\mathcal H_1,\ldots,\mathcal H_{11}$.}
\label{fig:atlas-term5}
\label{fig:atlas-term6}
\label{fig:atlas-term7}
\label{fig:atlas-term8}
\label{fig:atlas-term9}
\label{fig:atlas-term10}
\end{figure}
\clearpage
\endgroup

\end{document}